\documentclass[letterpaper,11pt]{article}

\usepackage[round]{natbib}

\usepackage{amsmath,amssymb,xcolor,amsfonts,amsthm,diagbox}
\usepackage{enumitem}
\usepackage{mathrsfs,bbm}
\usepackage[all]{xy}
\xyoption{tips}
\SelectTips{lu}{11}
\usepackage{tikz} 
\usepackage{cases} 
\usepackage{blkarray} 
\usepackage{comment, placeins,subcaption}
\usepackage{wasysym}

\def\P{{\mathbb P}} 
\def\E{{\mathbb E}}
\def\cA{{\mathcal A}}
\def\cC{{\mathcal C}}
\def\cE{{\mathcal E}}
\def\cF{{\mathcal F}}

\def\cK{{\mathcal K}}
\def\cS{{\mathcal S}}
\def\cH{{\mathcal H}}
\def\cL{{\mathcal L}}
\def\cN{{\mathcal N}}
\def\cP{{\mathcal P}}
\def\cX{{\mathcal X}}
\def\cY{{\mathcal Y}}
\def\cZ{{\mathcal Z}}

\allowdisplaybreaks 
\def\R{\mathbb{R}}

\def\N{\mathbb{N}}
\def\1{\mathbbm{1}}
\def\dd{\mathrm{d}}
\def\ee{\mathrm{e}}

\newcommand{\defeq}{\mathrel{\mathop:}=}
\newcommand{\eqdef}{=\mathrel{\mathop:}}
\newcommand{\defiff}{\mathrel{\mathop:}\iff}
\newtheorem{defn}{Definition}[section]
\newtheorem{thm}[defn]{Theorem}

\newtheorem{lemma}[defn]{Lemma}

\newtheorem{example}[defn]{Example}
\theoremstyle{definition}

\theoremstyle{remark}
\newtheorem{remark}[defn]{Remark}

\usepackage{graphicx} 
\usepackage[hidelinks]{hyperref}

\DeclareFontFamily{U}{mathx}{}
\DeclareFontShape{U}{mathx}{m}{n}{<-> mathx10}{}
\DeclareSymbolFont{mathx}{U}{mathx}{m}{n}
\DeclareMathAccent{\widehat}{0}{mathx}{"70}
\DeclareMathAccent{\widecheck}{0}{mathx}{"71}

\title{A conditional coalescent for diploid exchangeable population models given the pedigree}
\author{Frederic Alberti\footnote{fralbert@uni-mainz.de}, Matthias Birkner\footnote{birkner@mathematik.uni-mainz.de}, Wai-Tong Louis Fan\footnote{louisfan@unc.edu},  and John Wakeley\footnote{wakeley@fas.harvard.edu}}

\date{\today}

\begin{document}

\maketitle

\begin{abstract}
We study coalescent processes conditional on the population pedigree under the exchangeable diploid bi-parental population model of Birkner~et~al.\ \cite{BirknerEtAl2018}. While classical coalescent models average over all reproductive histories, thereby marginalizing the pedigree, our work analyzes the genealogical structure embedded within a fixed pedigree generated by the diploid Cannings model. In the large-population limit, we show that these conditional coalescent processes differ significantly from their marginal counterparts when the marginal coalescent process includes multiple mergers. We characterize the limiting process as an inhomogeneous $(\Psi,c)$-coalescent, where $\Psi$ encodes the timing and scale of multiple mergers caused by  sweepstake generations (SGs), and $c$ is a constant rate governing binary mergers. 

Our results reveal fundamental distinctions between quenched (conditional) and annealed (marginal) genealogical models, demonstrate how the fixed pedigree structure impacts multi-locus statistics such as the site-frequency spectrum, and have implications for interpreting  patterns of genetic variation among unlinked loci in the genomes of sampled individuals. They significantly extend the results of Diamantidis et al.\ \cite{DiamantidisEtAl2024}, which considered a sample of size two under a specific Wright-Fisher model with a highly reproductive couple, and those of Tyukin\ \cite{TyukinThesis2015}, where the Kingman coalescent was the limiting process.
Our proofs adapt coupling techniques from the theory of random walks in random environments.
\end{abstract}

{\small {\bf Keywords:} coalescent theory; population pedigree; genealogy; multiple mergers; ancestral inference 

}


\section{Introduction}

Coalescent models describe the sampling properties of genetic data with special reference to the hidden or underlying gene genealogy of the sample.  They are most commonly described as backward-time dual processes of the classic forward-time models of population genetics \cite{Donnelly1986,Ewens1990,Mohle1999}.  Predictions about samples of genetic data are the same under both approaches, though coalescent models are much simpler than forward-time models when there is no selection \cite{Tavare1984,Ewens1990,DonnellyAndKurtz1996a,DonnellyAndKurtz1999}.  The undeniable relevance of the gene genealogy as a latent variable adds to the appeal of coalescent models.  

The initial formal description of the standard neutral coalescent process called the $n$-coalescent \cite{Kingman1982a,Kingman1982b,Kingman1982c} was expressly for haploid organisms with exchangeable distributions of offspring numbers \cite{Cannings1974}. As usual, we denote by $n$ the sample size. Early biologically-oriented derivations assumed diploidy and bi-parental reproduction \cite{Hudson1983a,Hudson1983b,Tajima1983} under the Wright-Fisher model of reproduction \cite{Fisher1930,Wright1931}.  These implicitly used the notion of effective population size \cite{Wright1931,SjodinEtAl2005} in which a diploid Wright-Fisher population of size $N$ is understood to be equivalent to a haploid population of size $2N$, especially in the limit $N\to\infty$ \cite[section 3.7]{Ewens2004}.  

However, for diploid organisms there is another latent variable which is logically prior to the gene genealogy.  This is the population pedigree or organismal genealogy specifying the outcomes of bi-parental reproduction in every past generation \cite{BallEtAl1990,WollenbergAndAvise1998,WakeleyEtAl2012,Ralph2019}.  Genetic transmission within this pedigree according to Mendel's laws is what produces gene genealogies.    

The classic forward-time models make their predictions by averaging over outcomes of reproduction.  The traditional notion of an effective population size depends on this \cite{DiamantidisEtAl2024}.  The same averaging necessarily occurs in the corresponding backward-time dual coalescent models.  Rigorous treatments make this explicit, in particular by computing coalescent transition probabilities as expectations over the outcomes of reproduction, e.g.\ as in Lemma 3.1 of \cite{MohleAndSagitov2003}.  The resulting predictions about gene genealogies are marginal to all other aspects of the population, including the population pedigree.  The same is true of the numerous extensions of the coalescent approach, for example to cases in which multiple mergers occur in the ancestral process  \cite{DonnellyAndKurtz1999,Pitman1999,Sagitov1999,Schweinsberg2000,MohleAndSagitov2001,Sagitov2003} or when substantial inbreeding affects the gene genealogy \cite{NordborgAndDonnelly1997,Mohle1998a,SeversonEtAl2021,CotterEtAl2021}.  

Understanding how the population pedigree might affect the gene genealogy is a relatively new concern for population genetics.  \cite{TyukinThesis2015} gave the first rigorous account, proving that a broad class of diploid population models which have Kingman's $n$-coalescent as their limiting ancestral process under the traditional approach, i.e.\ marginal to the pedigree, also have the $n$-coalescent as their limiting ancestral process conditional on the pedigree.  \cite{DiamantidisEtAl2024} proved that the analogous statement \textit{does not} hold in general for coalescent processes with multiple mergers, using a simple model of large reproduction events and $n=2$.  \cite{NewmanEtAl2025} did the same for inbreeding, using a simple model of self-fertilization and $n=2$.  Both \cite{DiamantidisEtAl2024} and \cite{NewmanEtAl2025} showed that novel coalescent processes conditional on the pedigree emerge as $N\to\infty$, in which only certain events in the pedigree are retained in the limit.  

Here we consider the coalescent process for a sample of arbitrary size conditional on the pedigree under the exchangeable diploid bi-parental population model of \cite{BirknerEtAl2018}.  Our focus is on cases in which the coalescent process marginal to the pedigree includes (simultaneous) multiple mergers.  The corresponding marginal, pedigree-averaged or `unconditional' results are available in \cite{BirknerEtAl2018}.  Note that since each diploid individual possesses two copies of the genome, these coalescent processes marginal to the pedigree are always $\Xi$-coalescents \cite{Schweinsberg2000,MohleAndSagitov2001,Sagitov2003} rather than $\Lambda$-coalescents \cite{DonnellyAndKurtz1999,Pitman1999,Sagitov1999}.  We find that the resulting coalescent processes conditional on the pedigree, in the large population limit, differ from the corresponding marginal coalescent processes in \cite{BirknerEtAl2018}.  The reason is that the population-level events which lead to multiple-mergers coalescent events are encoded in the pedigree and thus frozen in time.  We show how to account for this, and provide a number of examples including two kinds of conditional diploid Beta coalescents. 

Coalescent processes conditional on the population pedigree are a new type of biologically motivated quenched limit result, analogous to those for random walks in random environments.  See \cite{Zeitouni2006} for a thorough review of the latter.  \cite{HassEtAl2024} provides more recent references as well as a new many-particle diffusion result.  In conditional coalescent processes, the environment is the pedigree generated by a process of reproduction in the population and walks are performed by ancestral genetic lineages according to Mendel's laws.  Whereas numerous properties of random walks in random environments may be of interest, the object of natural interest in conditional coalescent processes is the gene genealogy of a sample of size $n$.   

The method we use to establish convergence towards the quenched limit is to consider a pair of coalescent processes which are conditionally independent given the population pedigree.  This technique was developed for random walks in random environments \cite{BolthausenAndSznitman2002a,BirknerEtAl2013b} and adapted to a coalescent model in \cite{DiamantidisEtAl2024}, in which case the two hypothesized copies of the process have a natural interpretation.  They produce the gene genealogies at two unlinked or independently assorting loci.
Thus our results for the sampling properties of gene genealogies conditional on a shared population pedigree pertain directly to a typical way coalescent models are applied, namely to interpret patterns of genetic variation among loci in the genomes of sampled individuals. 

Understanding how the unobserved gene genealogy of a sample shapes genetic variation at a single locus is foundational to coalescent theory \cite{Donnelly1996}.  This is illustrated by applications to the site-frequency spectrum, which records the numbers $S_i$ of polymorphic sites with a mutant in sample frequency $i/n$ \cite{Tajima1989,FuAndLi1993,BravermanEtAl1995,Fu1995,GriffithsAndTavare1998}.  It is well known, for example, that under the infinite-sites mutation model with no recombination between sites \cite{Watterson1975}, single $n$-coalescent gene genealogies produce site-frequency spectra dramatically different from those expected on average \cite{SlatkinAndHudson1991,SainudiinEtAl2011}.  This can be seen from equations (1.9) and (3.3) in \cite{GriffithsAndTavare1998}, which show that only a fraction $2/(n-1)$ of gene genealogies present any opportunity for mutations in the largest sample frequency class $(n-1)/n$.  

Similarly, it is of interest to understand how population pedigrees affect statistical predictions for multi-locus data. Under the infinite-sites mutation model with free recombination between sites \cite{Fisher1930b,Kimura1969}, it has been standard practice in population genetics to treat each site as if it were an independent random draw from the expected site-frequency spectrum, specifically using the averaged or annealed model rather than the quenched model.  Examples include \cite{Ewens1974}, \cite{SawyerAndHartl1992}, \cite{Nielsen2000}, \cite{AdamsAndHudson2004}, \cite{CaicedoEtAl2007}, \cite{GutenkunstEtAl2009}, \cite{ExcoffierEtAl2013}, \cite{GaoAndKeinan2016} and \cite{SeplyarskiyEtAl2023}.  The only time this is justified is when the same limit-process holds for the quenched and annealed models, as for example in \cite{TyukinThesis2015}.  In cases where the quenched and annealed limit-processes differ, then not only does the gene genealogy matter for each locus but the population pedigree matters for all loci.  This should be evident in the expected site-frequency spectrum and other measures which depend on the distribution of gene genealogies among loci.

This work is the first to rigorously characterize the full coalescent process conditional on a pedigree for samples of arbitrary size in a general exchangeable diploid population. Our results go beyond prior analyses limited to sample size two or restricted to Kingman's coalescent, introducing a new class of inhomogeneous coalescents that explicitly account for the timing and structure of pedigree-encoded reproductive events. A technical challenge is handling the dense set of potential jump times in the limiting process, which we overcome by introducing an operator on the space of partitions (which we call a coagulator) that gives rise to a stochastic flow for the inhomogeneous coalescent. Our proof is built on a coupling framework adapted from random walks in random environments, using a pairwise coupling of coalescents on a shared pedigree. 

\section{Model and main results}

\subsection{Gene genealogies in fixed pedigrees}\label{subsec:fixedpedigrees}

\subsubsection{A general, fixed pedigree}
  We consider a population consisting of some fixed, deterministic number $N$ of diploid individuals.  Every individual in the population has two distinct parents and inherits a single copy of some chromosome under consideration from each of them. For ease of reference, we will call the two homologous copies of the chromosome the \emph{$0$-chromosome} and the \emph{$1$-chromosome} and the two gene copies on them the
\emph{$0$-gene} and \emph{$1$-gene}. 

Individuals are indexed by $i \in [N] \defeq \{1,2,\ldots,N \}$.  Generations are discrete and non-overlapping, and indexed backward in time by $g \in \N_0$, where $g=0$ is the present, $g=1$ is the previous generation, and so on back into the past.  
The parents of individual $i$ in generation~$g$ are denoted by $P_{0,i,g}^{(N)}$ and
$P_{1,i,g}^{(N)}$
and are members of the previous generation $g+1$;
here, for $c \in \{0,1\}$,
$P_{c,i,g}^{(N)}$
is the parent that transmitted the 
$c$-chromosome belonging to individual $i$.
We call $P_{0,i,g}^{(N)}$ and $P_{1,i,g}^{(N)}$ the $0$-parent and the $1$-parent of individual $i$ in generation $g$.


\begin{defn} \label{def:pedigrees} 
The map from $[N] \times \N_0$ to $[N]^2$  
\begin{equation*}
P^{(N)} \defeq \big({P_{0}^{(N)},P_{1}^{(N)}}\big) \defeq \big({P_{0,i,g}^{(N)},P_{1,i,g}^{(N)}}\big)_{i \in [N],g \in \N_0}
\end{equation*}
is called the \emph{population pedigree}. 
\end{defn}


Figure \ref{fig1a} shows a single generation of a population pedigree for $N=6$.  In Section~\ref{SS:Cannings} below, the population pedigree will be given as the result of a stochastic process of bi-parental reproduction. 

\begin{figure}
\centering
\begin{subfigure}[T]{0.47\textwidth}
  \begin{tikzpicture}[xscale=0.9, yscale=0.9] 
    \foreach \x in {1,2,3,4,5,6} {
        \filldraw[black] (\x -0.18,0) circle[radius=0.05]
                         (\x +0.18,0) circle[radius=0.05];
        \draw (\x,0) ellipse[x radius=0.35, y radius=0.25];
    }

    \foreach \x in {1,2,3,4,5,6} {
        \draw (\x,1) ellipse[x radius=0.35, y radius=0.25];
    }

    \draw[-] (2-0.18,0) -- (4,0.75); 
    \draw[-] (2+0.18,0) -- (2,0.75); 
    \draw[-] (3-0.18,0) -- (4,0.75); 
    \draw[-] (3+0.18,0) -- (5,0.75); 
    \draw[-] (5-0.18,0) -- (4,0.75); 
    \draw[-] (5+0.18,0) -- (5,0.75); 
    \draw[-] (6-0.18,0) -- (5,0.75); 
    \draw[-] (6+0.18,0) -- (4,0.75); 

    \draw[line width=0.2mm] (1-0.18,0) -- (1,0.75); 
    \draw[line width=0.2mm] (1+0.18,0) -- (2,0.75); 
    \draw[-] (4-0.18,0) -- (3,0.75); 
    \draw[-] (4+0.18,0) -- (5,0.75); 

    \draw (0.2,0) node[left] {$g$} ;
    \draw (0.2,1) node[left] {$g+1$} ;
  \end{tikzpicture}
  \vspace{3mm}
\caption{A realized pedigree between generations $g$ and $g+1$ with $N=6$ individuals. The 0- and 1-parents of the individual $(i,g)$ are the individuals $P_{0,i,g}^{(N)}$ and $P_{1,i,g}^{(N)}$ in generation $g+1$.}
\label{fig1a}
\end{subfigure}
\hfill
\begin{subfigure}[T]{0.4\textwidth}
 \begin{tikzpicture}[xscale=0.9, yscale=0.9] 
    \foreach \x in {1,2,3,4,5,6} {
        \filldraw[black] (\x -0.18,0) circle[radius=0.05]
                         (\x +0.18,0) circle[radius=0.05];
        \draw (\x,0) ellipse[x radius=0.35, y radius=0.25];
    }

    \foreach \x in {1,2,3,4,5,6} {
        \filldraw[black] (\x -0.18,1) circle[radius=0.05]
                         (\x +0.18,1) circle[radius=0.05];
        \draw (\x,1) ellipse[x radius=0.35, y radius=0.25];
    }

    \draw[-] (2-0.18,0) -- (4+0.18,1); 
    \draw[-] (2+0.18,0) -- (2+0.18,1); 
    \draw[-] (3-0.18,0) -- (4-0.18,1); 
    \draw[-] (3+0.18,0) -- (5-0.18,1); 
    \draw[-] (5-0.18,0) -- (4-0.18,1); 
    \draw[-] (5+0.18,0) -- (5+0.18,1); 
    \draw[-] (6-0.18,0) -- (5+0.18,1); 
    \draw[-] (6+0.18,0) -- (4+0.18,1); 

    \draw[line width=0.2mm] (1-0.18,0) -- (1-0.18,1); 
    \draw[line width=0.2mm] (1+0.18,0) -- (2+0.18,1); 
    \draw[-] (4-0.18,0) -- (3-0.18,1); 
    \draw[-] (4+0.18,0) -- (5+0.18,1); 
  \end{tikzpicture}
  \vspace{3mm}
\caption{The same pedigree together with realized Mendelian coin flips. We use $\big(M_{0,i,g},M_{1,i,g}\big)$ to record the ancestry of the 
$0$-gene and the $1$-gene of individual $(i,g)$.}
\label{fig1b}
\end{subfigure}    
\caption{A single generation of a population pedigree with $N=6$ individuals}
\label{fig:figures}
\end{figure}


\subsubsection{Mendelian inheritance}
Genetic lineages at an autosomal locus are transmitted through the pedigree according to Mendel's law of random segregation.  
For individual~$i$ in generation~$g$, the random pair  
$\big(M_{0,i,g},M_{1,i,g}\big)$ 
 with values in $\big\{{\left(0,0\right),\left(0,1\right),\left(1,0\right),\left(1,1\right)}\big\}$
indicates the gene copies transmitted by $P_{0,i,g}^{(N)}$ and $P_{1,i,g}^{(N)}$, respectively. 
Thus, $M_{c,i,g}$ encodes the choice of the ancestral gene copy in the parent $P_{c,i,g}^{(N)}$ of gene copy $c \in \{0,1\}$ in individual~$i$ in generation~$g$.
In Figure \ref{fig1b} we have
$(M_{0,1,g},\,M_{1,1,g})=(0,1)$, $(M_{0,2,g},\,M_{1,2,g})=(1,1)$, $(M_{0,3,g},\,M_{1,3,g})=(0,0)$, etc.


\begin{defn}\label{Def:M}
Let $(M_{c,i,g})_{c \in \{0,1\}, i \in [N], g \in \N_0}$ be i.i.d.\ with
\begin{equation*}
\P \big ( M_{c,i,g} = 0 \big ) = \frac{1}{2} =
\P \big ( M_{c,i,g} = 1 \big ).
\end{equation*}
We call the random variables $M_{c,i,g}$ \emph{Mendelian coin flips} or simply \emph{Mendelian coins}.
\end{defn}

\subsubsection{Ancestral line of a single gene within a pedigree}

The parental relationships among individuals given by the pedigree together with the Mendelian coin flips define ancestral relationships of gene copies.  In what follows, we use the words `ancestor' for ancestral gene copies and `parent' for ancestral diploid individuals.
\begin{defn} \label{def:ancestor}
Let $(c,i,g) \in \{0,1\} \times [N] \times \N_0$ represent the gene copy~$c$  in individual~$i$ in generation~$g$.  We call 
\begin{equation*}
\big (M_{c,i,g}, P_{c,i,g}^{(N)} ,   g+1 \big )
\end{equation*}
the \emph{ancestor of $(c,i,g)$ in generation $g+1$}. When the generation is clear from the context, we suppress the last coordinate and simply write $\big (M_{c,i,g}, P_{c,i,g}^{(N)}  \big )$.

More generally,  for $i,j \in [N]$, $c,d \in \{0,1\}$ and $h \geqslant  g+1$, we say that $(d,j,h)$ is 
\emph{the ancestor of $(c,i,g)$ in generation $h$} if there is a finite sequence
$(c,i,g) = (c_g^{}, i_g^{},g)$, $(c_{g+1}^{}, i_{g+1}^{},g+1)$, $\ldots$, 
$(c_h^{}, i_h^{},h) = (d,j,h)$ 
such that for all 
$\ell \in  \{g,g+1,\ldots,h-1\}$, 
$(c_{\ell+1},i_{\ell+1},\ell+1)$ is the ancestor of $(c_\ell,i_\ell,\ell)$
in generation $\ell+1$.



\end{defn}

\subsubsection{Gene genealogies of a finite sample} \label{subsec:finitesamplegenealogy}
We now fix an integer $n \in \N_0$ throughout the paper.  We also fix an arbitrary sample of $n$ pairwise distinct gene copies $X_1 (0), \ldots, X_n(0) \in \{0,1\} \times [N]$ from the present generation $0$ and consider their ancestral lines; see Figure \ref{fig:ancestral} for an illustration.  These ancestral lines may be viewed as coalescing random walks in 
the environment determined by the given pedigree. Later, this environment (the pedigree) will itself be random.
  
  
For $j \in [n]$ and $g \in \N_0$, we write $X_j(g) := \big(M_j(g),I_j(g)\big)\in \{0,1\} \times [N]$ for the ancestor in generation $g$ 
of the $j$-th sample $X_j(0)$, where $I_j(g) \in [N]$ is the index of the individual containing the ancestor of the sampled gene copy and $M_j(g) \in \{0,1\}$ specifies whether it is the $0$-gene or 
$1$-gene in that individual.  As in Definition~\ref{def:ancestor}, we suppress the generation coordinate $g$ and write $\bigl(M_j(g),I_j(g)\bigr)$ rather than $\bigl(M_j(g),I_j(g),g\bigr)$.

We emphasize the distinction between two similar notations: $M_j(g)$ pertains to the ancestral line of $X_j(0)$ whereas $M_{c,i,g}$ pertains to the Mendelian coin flips that can be seen as the `driving noise' of these random walks. 


\begin{defn} \label{def:ancestrallines}
For each $j \in [n]$, the  $\{0,1\} \times [N]$-valued process $(X_j (g))_{g \in \N_0}$ is called the \emph{ancestral line of the $j$-th sample}.
\end{defn}

Note that whenever $X_i(g) = X_j(g)$ for some $g \in \N_0$ and
$i,j \in [n]$, we have by  induction that $X_i^{} (h) = X_j^{} (h)$ for all 
$h \geqslant g$. We say that lineages $X_i^{}$ and $X_j^{}$ 
\emph{have coalesced at time $\min \{ g : X_i(g) = X_j (g) \}$}.

\begin{lemma}\label{L:ancestrallines}
Given a fixed pedigree  $P^{(N)}  =\big({P_{0}^{(N)},P_{1}^{(N)}}\big)$, 
the processes $(X_j (g))_{g \in \N_0}$ for $j\in [n]$ form a family of coalescing random walks on $\{0,1\} \times [N]$, whose transition probabilities are given by
\begin{equation}\label{E:ancestrallines}
\P \Big ( X_j (g + 1) =  (c, i) \,\Big|\, P^{(N)} , 
\big ( X_k^{} (h)   \big )_{h \leqslant g, k \in [n]} \Big )
=
\frac{1}{2} \1_{ \big\{
P^{(N)}_{M_j(g),I_j(g),g}
\big\}
}(i)
\end{equation} 
for $c\in\{0,1\}$, $i\in[N]$ and $g \in \N_0$. 


\end{lemma}
See Figure \ref{fig:ancestral} for an illustration of these coalescent random walks.

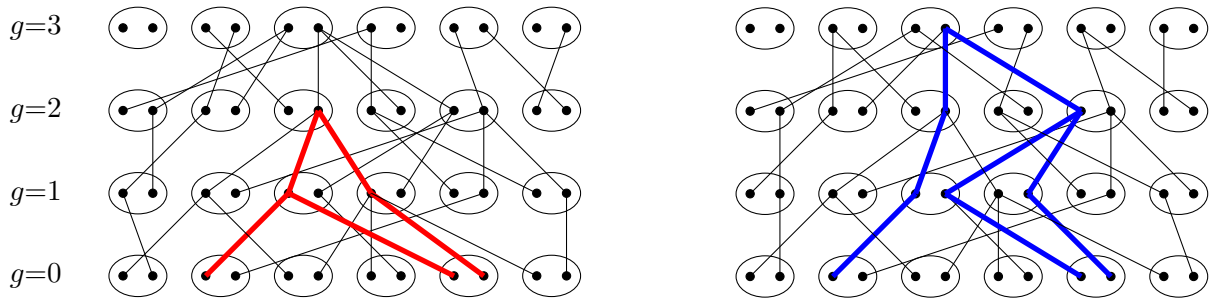
\begin{figure}
\centering
\begin{subfigure}{0.4\textwidth}
  \begin{tikzpicture}[xscale=1.0, yscale=1.0] 
            \foreach \y in {0,1,2,3}
            \foreach \x in {1,2,3,4,5,6} {
              \filldraw[black] (\x -0.18,\y) circle[radius=0.05]
                               (\x +0.18,\y) circle[radius=0.05];
              \draw (\x,\y) ellipse[x radius=0.35, y radius=0.25];
            }
            \draw (1-0.18,0) -- (2-0.18,1);
            \draw (1+0.18,0) -- (1-0.18,1);
            \draw [line width=0.7mm, red] (2-0.18,0) -- (3-0.18,1);
            \draw (2+0.18,0) -- (5+0.18,1);
            \draw (3-0.18,0) -- (2-0.18,1);
            \draw (3+0.18,0) -- (4-0.18,1);
            \draw (4-0.18,0) -- (4-0.18,1);
            \draw (4+0.18,0) -- (3+0.18,1);
            \draw[line width=0.7mm, red] (5-0.18,0) -- (3-0.18,1);
            \draw[line width=0.7mm, red] (5+0.18,0) -- (4-0.18,1);
            \draw (6-0.18,0) -- (4-0.18,1);
            \draw (6+0.18,0) -- (6+0.18,1);

            \draw (1-0.18,1) -- (2-0.18,2);
            \draw (1+0.18,1) -- (1+0.18,2);
            \draw (2-0.18,1) -- (3+0.18,2);
            \draw (2+0.18,1) -- (5+0.18,2);
            \draw [line width=0.7mm, red] (3-0.18,1) -- (3+0.18,2);
            \draw (3+0.18,1) -- (5-0.18,2);
            \draw  [line width=0.7mm, red] (4-0.18,1) -- (3+0.18,2);
            \draw  (4+0.18,1) -- (5-0.18,2);
            \draw (5-0.18,1) -- (4-0.18,2);
            \draw (5+0.18,1) -- (5+0.18,2);
            \draw (6-0.18,1) -- (4-0.18,2);
            \draw (6+0.18,1) -- (5+0.18,2);

            \draw (1-0.18,2) -- (4-0.18,3);
            \draw (1+0.18,2) -- (3-0.18,3);
            \draw (2-0.18,2) -- (2+0.18,3);
            \draw (2+0.18,2) -- (3-0.18,3);
            \draw (3-0.18,2) -- (2-0.18,3);
            \draw (3+0.18,2) -- (3+0.18,3);
            \draw (4-0.18,2) -- (4-0.18,3);
            \draw (4+0.18,2) -- (3+0.18,3);
            \draw (5-0.18,2) -- (3+0.18,3);
            \draw (5+0.18,2) -- (5-0.18,3);
            \draw (6-0.18,2) -- (6+0.18,3);
            \draw (6+0.18,2) -- (5+0.18,3);
    \draw (0.2,0) node[left] {$g$$=$$0$} ;
    \draw (0.2,1) node[left] {$g$$=$$1$} ;
    \draw (0.2,2) node[left] {$g$$=$$2$} ;
    \draw (0.2,3) node[left] {$g$$=$$3$} ;
  \end{tikzpicture}
\end{subfigure}
\hfill
\begin{subfigure}{0.4\textwidth}
 \begin{tikzpicture}[xscale=1.0, yscale=1.0] 
            \foreach \y in {0,1,2,3}
            \foreach \x in {1,2,3,4,5,6} {
              \filldraw[black] (\x -0.18,\y) circle[radius=0.05]
                               (\x +0.18,\y) circle[radius=0.05];
              \draw (\x,\y) ellipse[x radius=0.35, y radius=0.25];
            }
            \draw (1-0.18,0) -- (2-0.18,1);
            \draw (1+0.18,0) -- (1+0.18,1);
            \draw [line width=0.7mm, blue] (2-0.18,0) -- (3-0.18,1);
            \draw (2+0.18,0) -- (5+0.18,1);
            \draw (3-0.18,0) -- (2-0.18,1);
            \draw (3+0.18,0) -- (4-0.18,1);
            \draw (4-0.18,0) -- (4-0.18,1);
            \draw (4+0.18,0) -- (3+0.18,1);
            \draw[line width=0.7mm, blue] (5-0.18,0) -- (3+0.18,1);
            \draw[line width=0.7mm, blue] (5+0.18,0) -- (4+0.18,1);
            \draw (6-0.18,0) -- (4-0.18,1);
            \draw (6+0.18,0) -- (6-0.18,1);

            \draw (1-0.18,1) -- (2-0.18,2);
            \draw (1+0.18,1) -- (1+0.18,2);
            \draw (2-0.18,1) -- (3+0.18,2);
            \draw (2+0.18,1) -- (5+0.18,2);
            \draw [line width=0.7mm, blue] (3-0.18,1) -- (3+0.18,2);
            \draw [line width=0.7mm, blue] (3+0.18,1) -- (5-0.18,2);
            \draw (4-0.18,1) -- (3+0.18,2);
            \draw [line width=0.7mm, blue] (4+0.18,1) -- (5-0.18,2);
            \draw (5-0.18,1) -- (4-0.18,2);
            \draw (5+0.18,1) -- (5+0.18,2);
            \draw (6-0.18,1) -- (4-0.18,2);
            \draw (6+0.18,1) -- (5+0.18,2);

            \draw (1-0.18,2) -- (4-0.18,3);
            \draw (1+0.18,2) -- (3-0.18,3);
            \draw (2-0.18,2) -- (2-0.18,3);
            \draw (2+0.18,2) -- (3+0.18,3);
            \draw (3-0.18,2) -- (2-0.18,3);
            \draw [line width=0.7mm, blue] (3+0.18,2) -- (3+0.18,3);
            \draw (4-0.18,2) -- (4+0.18,3);
            \draw (4+0.18,2) -- (3-0.18,3);
            \draw [line width=0.7mm, blue]  (5-0.18,2) -- (3+0.18,3);
            \draw (5+0.18,2) -- (5-0.18,3);
            \draw (6-0.18,2) -- (6-0.18,3);
            \draw (6+0.18,2) -- (5-0.18,3);
  \end{tikzpicture}
\end{subfigure} 
\vspace{3mm}
\caption{\textbf{Two different gene genealogies within the same pedigree.} The two figures above are two realizations of the Mendelian coin flips in the population  of $N=6$ individuals with the same pedigree. Consider a sample of  $n=3$ genes at generation $g=0$. The sample 
$\big(X_i(0)
\big)_{1\leqslant i \leqslant3}$
is described as follows: $X_1(0)=(0,2)$ is the 0-gene of the second individual from the left, and $X_2(0)=(0,5)$ and $X_3(0)=(1,5)$ both belong to the fifth individual.  
}
\label{fig:ancestral}
\end{figure}

\subsubsection{Coalescent within the pedigree} 

In the next subsection, we will consider a general model for random pedigrees which will be exchangeable in the sense that its law will be invariant with respect to the permutation of individuals, and of the two homologous chromosome copies within each individual. In this setting,
it is customary to reduce the state space 
by ignoring the indices of the ancestral individuals and only keeping track of the coalescence of sample gene genealogies as we go backward in time. This is done via partitions 
of $[n]$, or equivalently, equivalence relations on $[n]$. Let $\cE_n$ denote the set of partitions of $[n]$. 
For $1 \leqslant i, j \leqslant n$,
$g \in \N_0$, we write $i \sim_g j$ if and only if samples $i$ and
$j$ descend from the same chromosome $g$~generations ago, i.e.,
\begin{equation}
  \label{eq:xinNm.alt}
  i \sim_g j \quad \defiff \quad X_i(g) = X_j(g). 
\end{equation}

For diploid models we must also keep track of pairs of ancestral chromosomes which are contained in the same ancestral individual (this information is crucial to retain the Markov property under the averaged measure). We use notation from~\cite{MohleAndSagitov2003} and consider the state space
\begin{equation*}
\begin{split}
  \mathcal{S}_n= \big \{& \left\{\left\{C_1,C_2\right\},\ldots,\left\{C_{2x-1},C_{2x}\right\},C_{2x+1},\ldots,C_b\right\} :
    b \in [n], 0 \leqslant x \leqslant \lfloor {b}/{2} \rfloor, \\ 
 & \{C_1, \ldots, C_b\} \in \mathcal{E}_{n,b} 
\big \},
\end{split}
\end{equation*}
where $\lfloor y\rfloor$ is the largest integer less than or equal
to $y$ and $\cE_{n,b}$ is the set of all partitions of $\cE_n$ into exactly $b$ blocks. We equip both spaces $\mathcal{E}_n$ and
$\mathcal{S}_n$ with the discrete topology, and interpret $\cE_n$ as a subset of $\cS_n$.
Conversely, we define
a map $\mathsf{cd} : \mathcal{S}_n\rightarrow\mathcal{E}_n$ that maps any
\begin{equation}
  \label{eq:xi.in.S_n}
  \xi=\big \{\left\{C_1,C_2\right\},\ldots,\left\{C_{2x-1},C_{2x}\right\},C_{2x+1},\ldots,C_b
  \big \}\in\mathcal{S}_n
\end{equation}
to
$$
\mathsf{cd}(\xi):=\left\{C_1,C_2,\ldots,C_{2x-1},C_{2x},C_{2x+1},\ldots,C_b\right\}
\in \mathcal{E}_n.$$ Following \cite{BirknerEtAl2013c} and \cite{BirknerEtAl2018}, we call
$\mathsf{cd}(\xi)$ the {\it complete dispersion} of $\xi$ and say that $\xi$ is \emph{completely dispersed} if $\mathsf{cd}(\xi) = \xi$.

\begin{remark} \label{rmk:singletonsnobrackets}
For technical reasons, it is useful to think of each element of $\cS_n$ as a collection of subsets of the power set of $[n]$, containing either one or two elements; strictly speaking, it would therefore be more accurate to write $\xi$ in \eqref{eq:xi.in.S_n} as 
\begin{equation*}
\xi=\big \{\left\{C_1,C_2\right\},\ldots,\left\{C_{2x-1},C_{2x}\right\}, \{ C_{2x+1} \},\ldots,\{ C_b \}
  \big \}.
\end{equation*}
However, we decided to drop the brackets around singleton sets for the sake of better readability.
\end{remark}

\begin{defn} \label{def:Pi^{N.n}}
For all $1 \leqslant n\leqslant N$, the \emph{ancestral process within the pedigree} 
is the $\mathcal{S}_n$-valued process 
$\Pi^{N,n} = \big (\Pi^{N,n}_g \big )_{g \in \N_0}$ such that 
\begin{equation*}
\mathsf{cd}(\Pi^{N,n}_g) = \{\text{equivalence classes under $\sim_g$}\}
\end{equation*}
and that the classes of $\mathsf{cd}(\Pi^{N,n}_g)$ are (potentially) grouped
into pairs as in \eqref{eq:xi.in.S_n} to form an element
of $\cS_n$ by the following rule: for two distinct classes $[i]_g$ and $[i']_g$, the unordered pair $\{[i]_g,[i']_g\}$ is an element of $\Pi^{N,n}_g$
if and only if 
\begin{equation}
  \label{eq:xinNm.alt2}
  I_i (g) = I_{i'}(g).
\end{equation}
For definiteness, given
\begin{equation*}
\Pi_0^{N,n} 
=
\big \{
\{C_1,C_2\},\ldots,\{C_{2x-1}, C_{2x}  \},C_{2x+1},\ldots,C_b
\big \},
\end{equation*}
we make the following choice for the initial position of the ancestral lines. For $1 \leqslant i \leqslant x$, we let 
$X_{2i - 1} \defeq (0,i)$ and 
$X_{2i} \defeq (1,i)$, while for
$2x+1 \leqslant i \leqslant b$, we let
$X_i \defeq (0,i)$. (Note that when $b<n$, it suffices to track one sample from each equivalence class of the initial state.)
\end{defn}

Note that if \eqref{eq:xinNm.alt2} holds and $i \not\sim_g i'$, then necessarily $M_i (g) = 1 - M_{i'} (g)$.

\begin{remark} 
Note that the choice for the initial position of the lineages in Definition \ref{def:Pi^{N.n}}
does not agree with the choice made in Figure \ref{fig:ancestral}. In fact, this choice will not matter for our result, as long as it is made \emph{independently of the pedigree}. This is due to the exchangeability (i.e. the invariance with respect to permutations) of the stochastic model generating our pedigrees, which we discuss next.
\end{remark}




\subsection{Diploid Cannings model for pedigrees}\label{SS:Cannings}

So far, we have constructed an ancestralprocess which describes the gene genealogy of a sample embedded within an arbitrary (but fixed) population pedigree. In  what follows, we will consider pedigrees which are randomly generated by a general model of exchangeable bi-parental reproduction.

We recall the general diploid 
(two-parent) exchangeable population model
with discrete generations from \cite{BirknerEtAl2018} which we will use as a model for a random pedigree.
This \emph{general diploid Cannings model}  is characterised by a sequence 
$\big (V^{(g)} \big )_{g \in \N}$ of
\emph{offspring matrices}. These are
i.i.d.\  $N \times N$ random matrices and prescribe the family size ($=$ number of children) of each pair of individuals.
For $g \in \N$, individuals $i$ and $j$ in the $g$-th generation have $V^{(g)}_{i,j}$ 
children in generation $g-1$. 
Here, each child of individuals $i$ and $j$ will independently have 
$i$ as their $0$-parent and $j$ as their $1$-parent, or vice versa, with equal probability $1/2$.

Clearly, $V^{(g)}_{j,i} = V^{(g)}_{i,j}$ for all $i$ and $j$. Throughout this work, we will not permit self-fertilization (hereafter just `selfing'). Hence, each individual has two distinct parents, meaning that $V^{(g)}_{i,i} = 0$ for all
$i \in [N]$.

Reproduction is independent and identically distributed from
generation to generation, i.e., the matrices
$\big(V^{(g)}_{i,j}\big)_{1 \leqslant i,j \leqslant N}$, $g \in \N$ are
i.i.d.\ and we write $V_{i,j}=V_{i,j}^{(1)}$ for simplicity.

Our fundamental assumptions are
that the population size remains fixed
\begin{equation}
 \label{eq:fixedN}
 \sum_{1 \leqslant i < j \leqslant N} V_{i,j} = N 
\end{equation}
and that there is no population structure
\begin{equation}
  \label{eq:exchangeable}
  \left(V_{i,j}\right)_{1 \leqslant i,j \leqslant N} \mathop{=}^d \left(V_{\sigma(i),\sigma(j)}\right)_{1 \leqslant i,j \leqslant N}
  \quad \text{for any permutation $\sigma$ of $[N]$}.
\end{equation}
Mathematically,
$\left(V_{i,j}\right)_{1 \leqslant i,j \leqslant N}$ is a symmetric, finite, jointly
exchangeable array with vanishing diagonal. Its law depends on $N$ but we will not
make this explicit in our notation. 
This is a general population model that unifies and generalizes many classical setups; see Section \ref{sec:examples} for some concrete examples.

\smallskip

We now describe the random pedigree corresponding to the diploid Cannings model. Conditional on the offspring matrices
$\big (V^{(g)} \big )_{g \in \N}$,
we determine $P^{(N)} = \big({P_{0,i,g}^{(N)},P_{1,i,g}^{(N)}}\big)_{i \in [N], g\in \N_0}$ in Definition \ref{def:pedigrees} as follows: For any
$P = (P_{c,k})_{c \in \{0,1\}, k \in [N]}$,
write $v_{i,j}^{} (P)=\#\{ k : \{ P_{0,k}, P_{1,k} \} = \{i,j\} \}$. Independently for each generation $g$,  we
pick $P_{\cdot,\cdot,g}$ uniformly among all choices of $P$ with $v_{i,j}^{}(P) = V^{(g+1)}_{i,j}$ for all $i, j\in[N]$.
Alternatively, this can be explicitly implemented via a balls-in-boxes scheme; 
see Section~\ref{subsec:condpedigree} below for explicit formulas and more explanations.

\subsection{Fundamental convergence assumption for ancestral processes}

Here and in the next subsection we recall the annealed convergence result (restated in Theorem~\ref{thm0}) for the process $\Pi^{N,n}$, to help set the stage for the statement of our main result. 

We first define
the {\em total offspring numbers}
\begin{align}
  \label{def:Vi}
  V_i := \sum_{1 \leqslant j \leqslant N} V_{i,j}, \quad i \in [N],
\end{align}
which are the total number of offspring of each individual $i$ for $1\leqslant i \leqslant N$. They play a crucial role
in the description of the limit as $N\to\infty$. These $V_i$ children may be full or half siblings.  The vector $(V_i)_{1 \leqslant i \leqslant N}$ inherits exchangeability
from the array $(V_{i,j})$,
see~\cite[Eq.~(1.3)]{BirknerEtAl2018}.
Observe that we always have 
$\sum_{i=1}^N V_i = 2N-\sum_{i=1}^NV_{ii} =2N$ because there is no selfing.


We denote by $c_N^{}$ the \emph{annealed pair coalescence probability}, which is the mean probability (averaged over all realisations of the pedigree) that a given pair of genes, \emph{from two distinct individuals} in the same generation,
has a common ancestor (gene) in the previous generation.

We have
\begin{equation}\label{eq:cN1}
c_N^{} = 
\frac{1}{8 (N-1)} \E[ (V_1)_{2 \downarrow}]
\end{equation}
where $(v)_{k \downarrow}:=v(v-1)\cdots(v-k+1)$ denotes the $k$-th falling factorial;  see \cite[Eq.~(1.4) and Lemma~3.1; also the discussion there in Section~1.1, item~3]{BirknerEtAl2018}.

We assume that
\begin{equation}\label{eq:limcN=0}
  \lim_{N\to\infty}c_N^{} = 0,
\end{equation}
since otherwise, there cannot be a continuous-time limit process. 

\smallskip

  Denote by  $V_{(1)} \geqslant V_{(2)} \geqslant \cdots \geqslant V_{(N)}$
the ranked version of the total offspring numbers $(V_1,\dots,V_N)$ and by
\begin{equation} \label{eq:bigphiN}
  \Phi_N := \mathscr{L}
  \Big( {\textstyle \frac{V_{(1)}}{2N}, \frac{V_{(2)}}{2N}, \dots,
    \frac{V_{(N)}}{2N}, 0, 0, \dots} \Big)
\end{equation}
the law of the ranked (total) offspring \emph{frequencies}, viewed
as a probability measure on the infinite-dimensional simplex
$\Delta:=\left\{(x_1^{},x_2^{},\ldots): x_1^{}\geqslant x_2^{} \geqslant \cdots \geqslant 0, \,
  \sum_{i=1}^{\infty}x_i^{}\leqslant 1\right\}$.  For
\mbox{$x=(x_1^{},x_2^{},\ldots)\in\Delta$},
we write $|x|_1^{}:=\sum_{i=1}^{\infty}x^{}_i$ for its $\ell^1$-norm,
$\|x\|^2 \defeq \langle x,x \rangle:=\sum_{i=1}^{\infty}x_i^2$, and put
\mbox{$\mathbf{0} := (0,0,\dots) \in \Delta$}. We equip $\Delta$ with the
topology induced by the product topology on $[0,1]^\infty$, metrized e.g.\ via
$d_\Delta(x,y) = \sum_{i=1}^\infty 2^{-i} |x_i^{}-y_i^{}|$.

Importantly, we assume there  exists a sub-probability measure $\Xi'$ on $\Delta \setminus \{ \mathbf{0} \}$ such that
  \begin{align}
  \label{eq:PhiNconv}
  \frac1{2 c_N^{}} \Phi_N(\dd x) \mathop{\longrightarrow}_{N\to\infty}
  \frac{1}{\langle x,x \rangle} 
  \Xi'(\dd x)
  \quad \text{vaguely on}\; \Delta \setminus \{\mathbf{0}\}.
\end{align}
We extend $\Xi'$ to a probability measure on $\Delta$  by simply putting $\Xi'(\{ \mathbf{0} \}) \defeq 1- \Xi'(\Delta \setminus \{\mathbf{0}\})$.
(This corresponds to a decomposition 
$\Xi' = \Xi'_0 + a \delta_{\bf 0}^{}$ as
in~\cite[Theorem 2]{Schweinsberg2000} with a measure $\Xi'_0$ on 
$\Delta \setminus \{ \mathbf{0} \}$. See also
\cite[Eqs.~(1.5), (1.6) and Appendix~A]{BirknerEtAl2018}.)

It is known that condition \eqref{eq:PhiNconv} together with \eqref{eq:limcN=0} is, in the pedigree-averaged case, necessary and sufficient for convergence of the rescaled process $( \Pi^{N,n}_{\lfloor t/c_N \rfloor})_{t\in (0,\infty)}$ as $N\to\infty$ to a (continuous time) coalescent with simultaneous multiple mergers; see  \cite{Schweinsberg2000,MohleAndSagitov2001,Sagitov2003}. The latter limiting process is an $n$-$\Xi$-coalescent, where 
\begin{equation} \label{eq:Xi}
\Xi \defeq \Xi' \circ \varphi^{-1} \quad \text{with} \;\; \varphi : \Delta \to \Delta, \quad (x'_1,x'_2,\ldots ) \mapsto  (x'_1 / 2, x'_1 / 2, x'_2 / 2, x'_2 / 2, \ldots ).
\end{equation}
  
\smallskip

The convergence \eqref{eq:PhiNconv} can be interpreted as follows. Let $x \in \Delta \setminus \{ \mathbf{0} \}$.  After rescaling time by $c_N^{-1}$ (that is, letting one unit of continuous time correspond to $\lfloor c_N^{-1} \rfloor$ discrete generations), in the limit $N \to \infty$ we encounter at rate $2 \langle x', x' \rangle^{-1} \Xi'(\dd x')$ a \emph{sweepstake generation} (SG) \cite{Hedgecock94,ArnasonEtAl2023}
in the sense that there are individuals whose children make up non-vanishing fractions $0 \neq x'_1 \geqslant x'_2 \geqslant \ldots \geqslant 0$
of the population. 

In an SG with relative offspring fractions $x' = (x'_1,x'_2,\ldots)$, we expect each gene in the offspring population to choose independently either gene carried by parent $i$ as its ancestor with probabilities $x'_i / 2$ respectively, because each diploid parent carries two gene copies.
See Subsection~\ref{subsec:largeoffspring} and, in particular, Lemma~\ref{lem:largerep}, for a more detailed discussion.

Letting  $x=(x_i)_{i=1}^{\infty} \defeq \varphi(x')$,
we expect that encountering an SG (say, at rescaled time $t$) should trigger a transition of the limiting coalescent 
according to a \emph{$x$-merger} as follows:
consider a partition of the unit interval with lengths $x_1,x_2,\ldots$ and the remaining length $1-|x|_1$, we will call the resulting
partition intervals $[x_1+\cdots+x_{i-1}, x_1+\cdots+x_i)$, $i=1,2,\dots$ \emph{buckets}.  We throw each equivalence class of the current state of the coalescent uniformly at random on the unit interval. All classes which land in the same bucket coalesce, if their positions on the interval are at most $|x|_1$. 

More formally, we say that a coalescent process $\Pi$ performs at time $t$ a \emph{merger in the paintbox $x$} or, more concisely, a \emph{$x$-merger}
if, for any $\xi \in\mathcal{E}_n$, the transition from $\Pi_{t-} = \xi$ to $\Pi_t = \eta$ occurs with probability
\begin{align}
  \label{eq:paintboxtransitionprime}
  p(x; \xi, \eta) \defeq \sum_{\ell=0}^s
  \sum_{\substack{i_1,i_2,\dots,i_{r+\ell} \in \N \\ \text{pairw.\ diff.}}}
  \binom{s}{\ell} x_{i_1}^{k_1} x_{i_2}^{k_2} \cdots x_{i_r}^{k_r} x_{i_{r+1}} x_{i_{r+2}} \cdots x_{i_{r+\ell}}
  \Big( 1 - |x|_1 \Big)^{s-\ell}
\end{align}
for any $\eta \in\mathcal{E}_n$  that arises from $\xi$ by merging $r$ groups of blocks of sizes 
$k_1^{}, k_2^{}, \ldots k_r^{} \geqslant 2$. Here,  $s = |\xi| - \sum_{\ell = 1}^r k_\ell^{}$ is the number of blocks not participating in any merger.  Clearly, 
$|\eta| = r+s$.

In \eqref{eq:paintboxtransitionprime}, $i_1^{},\ldots,i_r^{}$
are the indices of the buckets (of sizes $x_{i_1^{}}^{}, x_{i_2^{}}^{},\ldots$)
containing 
$k_1^{},\ldots,k_r^{}$
blocks of $\xi$, and 
$i_{r+1}^{}, \ldots, i_{r+\ell}^{}$ of those that contain only a single block each. 
The outer sum over $\ell$ comes from the fact that a block does not participate in a merger if either it ends up alone in one of the buckets of sizes $x_1^{}, x_2^{}, \ldots$ or it ends up (perhaps together with others) in the last bucket of size $(1 - |x|_1)$;
see
\cite[Theorem~2]{Schweinsberg2000} for details and 
also~\cite{Aldous1985,Kingman1978,Pitman1995} for general background on the construction of random partitions.

The above interpretation applies to both limits: averaging over the sweepstake generations gives the annealed $n$-$\Xi$-coalescent, whereas conditioning on their realized times and frequencies leads to the inhomogeneous coalescent in the quenched limit.

\subsection{The annealed limit}
  



\smallskip


With our set up so far, we recall the annealed convergence result (i.e. when averaging over the pedigree and the Mendelian randomness simultaneously) of the  process $\Pi^{N,n}$.
For each $N\in\mathbb{N}$, we let  $(\Omega^{(N)},\mathcal{F}^{(N)}, \P^{(N)})$
be a probability space on which  the i.i.d.\ random variables $(M_{c,i,g})_{i \in [N], c \in \{0,1\}, g \in \mathbb{Z}}$ (in Definition \ref{Def:M}) and 
the random pedigree
$P^{(N)} \defeq
(P^{(N)}_{0}, P^{(N)}_{1})$ 
as constructed in Section~\ref{SS:Cannings} are defined.
\begin{thm}[\cite{BirknerEtAl2018}, Theorem~1.1]
  \label{thm0}
  Assume  \eqref{eq:fixedN}, \eqref{eq:exchangeable},
  \eqref{eq:limcN=0} and \eqref{eq:PhiNconv} hold. Fix $n\in\mathbb{N}$ and $\xi_0 \in \mathcal{S}_n$. For all
  $N \geqslant n$, we let 
  $\Pi^{N,n}$ 
  be the ancestral process defined in Definition~\ref{def:Pi^{N.n}}, 
  starting from $\Pi^{N,n}_0 = \xi_0$.
  Then 
  \begin{equation*}
    \left( \Pi^{N,n}_{\lfloor t/c_N \rfloor} \right)_{t\in (0,\infty)}
    \mathop{\longrightarrow}_{N\to\infty}
    \left( \Pi^{n,\mathsf{avg}}_t \right)_{t\in (0,\infty)}
  \end{equation*}
  in finite-dimensional distributions in $\mathcal{S}_n$, i.e., for all 
  $0 < t_1^{} < \ldots < t_m^{} < \infty$ and all 
  $\xi_1^{},\ldots,\xi_m^{} \in \mathcal{S}_n$, we have
  \begin{equation}
    \label{Annealed_fdd_conv_explicit}
    \P^{(N)} \big (
    \Pi^{N,n}_{\lfloor t_1^{}/c_N \rfloor} = \xi_1^{}, \ldots,
    \Pi^{N,n}_{\lfloor t_m^{}/c_N \rfloor} = \xi_m^{}
    \big ) \mathop{\longrightarrow}_{N \to \infty}
    \P \big (\Pi^{n,\mathsf{avg}}_{t_1^{}} = \xi_1^{}, \ldots, \Pi^{n,\mathsf{avg}}_{t_m^{}} = \xi_m^{} \big ).
  \end{equation}
  
  The limit process $\Pi^{n,\mathsf{avg}}$ is an
  $n$-$\Xi$-coalescent starting from $\Pi^{n,\mathsf{avg}}_0 = \mathsf{cd}(\xi_0)$, where $\Xi\defeq \Xi' \circ \varphi^{-1}$ 
  is characterized by 
  \eqref{eq:PhiNconv} and \eqref{eq:Xi}. Note that the r.h.s.\ of \eqref{Annealed_fdd_conv_explicit} equals $0$
  if any of the $\xi_i \in \mathcal{S}_n \setminus \mathcal{E}_n$ there.
\end{thm}



\subsection{Inhomogeneous coalescent and the statement of the main result} \label{subsec:inhomcoal}

Theorem~\ref{thm0} describes the annealed large-population limit obtained by averaging over the pedigree: after rescaling time, the ancestral process converges to the homogeneous $n$-$\Xi$-coalescent. Here we describe the corresponding quenched limit, in which the ancestral process is conditioned on a typical realisation of the pedigree generated by the diploid Cannings model.

In contrast to the annealed limit, the quenched limit fixes the times and offspring frequencies of the sweepstakes and is therefore generally time-inhomogeneous. To state our main result, we now introduce  the inhomogeneous $(\psi,c_{\textnormal{pair}}^{})$-coalescent, where $\psi$ records the timing and offspring frequencies in sweepstake generations and $c_{\textnormal{pair}}$, defined by 
\begin{equation} \label{eq:cpair}
c_{\textnormal{pair}}^{} \defeq 1 - \Xi' \big ( \Delta \setminus \{ \mathbf{0} \}  \big ),
\end{equation}
is the rate of the independent pair mergers. In our quenched limit theorem, $\psi$ will be the realisation of a Poisson point process $\Psi$ on 
$[0,\infty) \times ( \Delta \setminus \{ \mathbf{0} \}  )$ with intensity 
\begin{equation*}
\dd t \, \frac{1}{\langle x,x \rangle} \Xi (\dd x).
\end{equation*}

However, to extend our framework, in Definition~\ref{def:inhomcoalescentfullygeneral} below we will define the inhomogeneous $(\psi,c_{\textnormal{pair}})$-coalescent for a general, fixed $\psi$.

\subsubsection{Inhomogeneous $(\psi, c)$-coalescent via coagulator construction} \label{subsec:coagulators}



Informally, for a given simple point measure $\psi$ and a number $c\geq 0$, the inhomogeneous ($n$)-$(\psi,c)$-coalescent $\Pi^{\psi,c}$ is an inhomogeneous Markov chain with càdlàg sample paths
taking values in $\cE_n$ such that (i)
for each 
$(t,x) \in \psi$, $\Pi^{\psi,c}$
performs at time $t$ an $x$-merger as described around \eqref{eq:paintboxtransitionprime}, and (ii) independently for each $\eta$ 
which arises from $\xi$ via a pair coalescence,
$\Pi^n$ jumps from $\xi$ to $\eta$ at an additional exponential rate $c$. 

It is straightforward to construct such a Markov  process with càdlàg sample paths 
whenever 
\begin{equation} \label{eq:discretePsi}
\psi \big ( [0,T] \times \Delta  \big ) < \infty
\quad 
\textnormal{ for all } T\in(0,\infty),
\end{equation}
because in this case the atoms of $\psi$ can simply be treated one after the other in temporal order. Condition \eqref{eq:discretePsi} holds almost surely (i.e. for almost all realizations $\psi$ of $\Psi$) if and only if
$\Xi$ satisfies 
$\int_{\Delta \setminus \{\bf{0}\}} \langle x,x \rangle^{-1} \, \Xi (\dd x) < \infty$.

\smallskip

In the general case where possibly $\int_{\Delta \setminus \{\bf{0}\}} \langle x,x \rangle^{-1} \, \Xi (\dd x) =+\infty$, we have to address the technical problem that the set of potential jump times of $\psi$
may be dense in $[0,\infty)$. 

The solution lies in the following two observations: First, note that a realization $\psi = \sum_{i \geqslant 1} \delta_{(t^{(i)}, \, x^{(i)})}$ of $\Psi$ satisfies the following almost surely:
\begin{equation}
  \label{eq:l2cond}
 \int_{[0,T]\times \Delta} \langle x,x \rangle \, \psi(\dd t, \dd x) =
  \sum_{i \, : \,0 <  t^{(i)} < T} \langle x^{(i)}, x^{(i)} \rangle < \infty \quad \text{for every } T\in(0,\infty).
\end{equation}
Indeed, since $\Psi$ is a Poisson point process with intensity $\dd t 
\frac{1}{\langle x,x \rangle} \Xi (\dd x)$, we have 
\begin{equation*}
\E \Big [ \int_{[0,T]\times \Delta} \langle x,x \rangle \, \Psi(\dd t, \dd x) \Big  ]
= T \, \Xi (\Delta \setminus {\bf 0}) < \infty.
\end{equation*}
We denote by $\cN_{\mathsf{loc}}^2$ the set of all simple point measures satisfying \eqref{eq:l2cond}.

Second, note that for any $\psi= \sum_{i \geqslant 1} \delta_{(t^{(i)}, \, x^{(i)})}\in \cN_{\mathsf{loc}}^2$
and any $\xi \in \mathcal{E}_n$, the product
\begin{align}
  \label{eq:noncoal.psi.s<t.0}
  \prod_{i \, : \, 0 < t^{(i)} < T} p(x^{(i)}; \xi,\xi)
\end{align}
is well-defined and strictly positive for $T\in(0,\infty)$, even if $\psi$ has atoms at a dense set of times. 

To see this, recall the interpretation of
\eqref{eq:paintboxtransitionprime}
in terms of the paintbox construction.
The probability that any given pair of blocks lands in the same bucket is 
$\langle x,x \rangle$. Hence,
for any $\xi \preccurlyeq \eta$ and $\eta \neq \xi$, we have 
$p(x;\xi,\eta) \leqslant 2^{|\xi|} \langle x,x\rangle \leqslant 2^n \langle x,x \rangle$, where 
the factor
$2^{|\xi|}$ comes from a union bound; 
each block can either be involved in a given merger event or not, which bounds the total number of possible mergers by $2^{|\xi|}$. The probability of each merger is bounded by $\langle x,x \rangle$, the probability of any given pair of blocks merging in a paintbox $x$.
Thus 
$p(x;\xi,\xi) \geqslant 1 - K \langle x,x \rangle$ for some 
absolute constant $K$.
This suggests that although infinitely many jumps may occur on arbitrarily small time intervals, most of them will be void so that the (random) set of \emph{effective} jump times will be locally finite. 



\smallskip

Recall 
that for any $x$-merger and given that the coalescent is in state $\xi \in \cE_n$, the jump target $\eta \in \cE_n$ is determined by 
throwing balls into 
buckets whose sizes are given by $x = (x_1^{},x_2^{},\ldots)$. 
Note that this random experiment can be performed \emph{in advance}, without knowledge of $\xi$; we may first draw a \emph{random} partition $\alpha \in \cE_n$ with distribution
$p(x;\xi^{(n)}_0,\cdot)$, where $\xi^{(n)}_0 \defeq \{\{1\},\dots,\{n\}\}$,
and then define 
the jump target $\eta$ by merging all blocks of $\xi$ whose indices are contained in the same block of $\alpha$. 

We determine the full evolution of the inhomogeneous coalescent  by drawing an independent 
$\alpha_{t,x}^{} \sim p(x;\xi_0^{(n)},\cdot)$ for each atom
$(t,x) \in \psi$. In order to use this to define a coalescent process, we need the notion of a 
\emph{coagulator}; given $\alpha \in \cE_n$, the associated $\alpha$-coagulator turns any $\xi \in \cE_n$ into a coarser partition by merging, for each block $A$ of $\alpha$, those blocks of $\xi$ whose indices are contained in $A$;
see also~\cite[Section 2.2]{BertoinLeGall2003} or~\cite[Chapter 4]{bertoin2006random}.

\begin{defn} \label{def:coagulators}
For $\alpha \in \cE_n$, define the 
\emph{$\alpha$-coagulator}
$\mathsf{Coag}_\alpha : \cE_n \to \cE_n$ for
$\xi = \{C_1,\ldots,C_b  \}$
via 
\begin{equation*}
\mathsf{Coag}_\alpha (\xi) =
\Big (
\bigcup_{A \in \alpha} 
\Big \{
\bigcup_{i \in A \cap [b]} C_i
\Big \}
\Big ) \setminus \{ \varnothing \}.
\end{equation*}
\end{defn}

We interpret the family 
$(\alpha_{t,x}^{})_{(t,x) \in \psi}$
as an inhomogeneous point process $\cC_\psi$ on
$(0,\infty) \times \cE_n$, 
letting
\begin{equation} \label{eq:paintboxeff}
\cC_\psi \defeq 
\sum_{
\substack{
(t,x) \in \psi \\
\alpha_{t,x}^{} \neq \xi_0^{(n)}
}
}
\delta_{(t, \, \alpha_{t,x}^{})}.
\end{equation}


Similarly, we can determine the pair-coalescence times \textit{in advance}. 
Independently for each pair $(i,j)$ with 
$1 \leqslant i < j \leqslant n$, we let
$\cK_c^{i,j} \subseteq (0,\infty)$ be a 
Poisson process on $[0,\infty)$ with rate~$c$,
\begin{equation} \label{eq:pairmergereff}
\cK_c^{i,j} \defeq \sum_{k = 1}^{\infty}
\delta_{\sum_{\ell = 1}^k T_\ell^{i,j}},
\end{equation}
where $T_1^{i,j},T_2^{i,j}$ are i.i.d. exponential with rate $c$.
We then define a point process $\cK_c$ on
$(0,\infty) \times \cE_n$ via
\begin{equation} \label{eq:allpairmergers}
\cK_c \defeq \sum_{1\leqslant i < j \leqslant n}
\sum_{t \in \cK_c^{i,j}} \delta_{(t,\xi_{i,j}^{})}
\end{equation}
where $\xi_{i,j}^{}$ is the partition of $[n]$ that is comprised exclusively of singletons and the block $\{i,j\}$.


We observe that the point processes defined in 
\eqref{eq:paintboxeff} and \eqref{eq:allpairmergers} are almost surely locally finite.
For $\cC_\psi$, this is an immediate consequence of $\psi$ being in $\cN_{\mathsf{loc}}^2$. Indeed, we have for all $T > 0$ that
\begin{equation*}
\E \big [ \cC_\psi \big ( (0,T] \times \cE_n \big ) \big ]
=
\sum_{i \, : \, 0 < t^{(i)} \leqslant T} \Big( 1 - p \big (
x^{(i)};\xi_0^{(n)},\xi_0^{(n)}\big)\Big) \leq
K \sum_{i \, : \, 0 < t^{(i)} \leqslant T} 
\langle x^{(i)}, x^{(i)} \rangle < \infty
\end{equation*}
by the discussion surrounding 
\eqref{eq:noncoal.psi.s<t.0} and the fact that $\psi \in \cN_{\mathsf{loc}}^2$.
By a Borel-Cantelli argument,
$\sum_{\ell=1}^\infty T_\ell^{i,j} = \infty$ almost surely, which implies that $\cK_c^{i,j}$ is almost surely locally finite for each choice of $i$ and $j$.

We can now define the 
inhomogeneous $(\psi,c)$-coalescent for a deterministic choice of $\psi \in \cN_{\mathsf{loc}}^2$ and $c > 0$. 
To fit the (yet to be defined) $\varepsilon$-naive coalescent into our framework, we define more generally the inhomogeneous $(\psi,\gamma)$-coalescent for a general distribution $\gamma$ on $(0,\infty)$, which replaces the exponential distribution of the waiting times 
$T_\ell^{i,j}$ in \eqref{eq:pairmergereff}.

With this, we can define the inhomogeneous $(\psi,\gamma)$-coalescent starting from any arbitrary initial state $\xi_0^{} \in \cE_n$. Typically, we will use $\xi_0^{} = \xi_0^{(n)}$.
\begin{defn}
\label{def:inhomcoalescentfullygeneral}
Given $\psi \in \cN_{\mathsf{loc}}^2$ and a probability distribution $\gamma$ on $(0,\infty)$, the \emph{inhomogeneous
$(\psi,\gamma)$-coalescent} is a stochastic process 
$\Pi^{\psi,\gamma} = (\Pi^{\psi,\gamma}_t)_{t \geqslant 0}^{}$ taking values in $\cE_n$ and constructed as follows. Let
$\cC_\psi$ be as in \eqref{eq:paintboxeff} and 
$\cK_\gamma$
be as in \eqref{eq:allpairmergers} when  $T_{\ell}^{i,j}$ are i.i.d. with distribution $\gamma$. Given the initial state
$\Pi^{\psi,\gamma}_0 = \xi_0^{}$, we let, for all $t > 0$,
\begin{equation*}
\Pi_t^{\psi,\gamma}
\defeq
\bigg( \prod_{
(s,\alpha) \in \cC_\psi \cup \cK_\gamma \, : \,
s \leqslant t 
}
\mathsf{Coag}_\alpha \bigg) (\xi_0^{}),
\end{equation*}
where the product is ordered chronologically, i.e. for 
$M = \{ (s_1,\alpha_1), \ldots, (s_k,\alpha_k) \} $ 
with $s_1 < \ldots < s_k$, 
$\prod_{(s,\alpha) \in M} \mathsf{Coag}_\alpha \defeq
\mathsf{Coag}_{\alpha_k} \circ \ldots \circ 
\mathsf{Coag}_{\alpha_1}
$. Whenever $s_i^{} = s_j^{}$ for some $i$ and $j$, we remove all atoms except for one; if there are multiple atoms of $\cK_\gamma$ and no atom of $\cC_\psi$, we keep a uniformly chosen one and discard the others. Otherwise, we give priority to the atom of $\cC_\psi$. 
For $c \geqslant 0$, we commit a slight abuse of notation and refer to the
$\big ( \psi,\mathsf{Exp}(c) \big )$-coalescent as the $(\psi,c)$-coalescent and denote it as $\Pi^{\psi,c}$. 
\end{defn}

\begin{remark}
While $\gamma$ can in principle be an arbitrary probability distribution on $(0,\infty)$, we will later take $\gamma$ to be either an exponential or a rescaled geometric distribution; then, it is not hard to see that $\Pi^{\psi,\gamma}$ is an inhomogeneous Markov process (when fixing $\psi$ but averaging w.r.t.\ $\cK_\gamma$).
\end{remark}

Note that the process
$\Pi^{\psi,\gamma}$ has càdlàg sample paths, as seen immediately from its construction. 
However, for a general waiting time distribution $\gamma$, it is in general not a Markov process. For example, consider the case $n=3$, $\psi = 0$ and $\gamma$ the uniform distribution on $(0,1)$. Imagine that, starting from $\Pi^{0,\gamma}_0 = \big \{ \{1\}, \{2\}, \{3\} \big \}$, we have for some $t \in (0,2)$ that 
$\Pi^{0,\gamma}_t = \big \{
\{1,2\}, \{3 \}
\big \}$. This implies
$T^{1,2}_1 \leqslant t$ and that we observe the jump to $\big \{ \{1,2,3\} \big \}$ at time $T^{1,2}_2$. At time $t$, this jump occurs at rate
$(1 - t + T^{1,2}_1)^{-1}$ which depends on the time of the first jump in addition to $t$ and $\Pi^{0,\gamma}_t$.


\subsubsection{Main result}

Following \cite{DiamantidisEtAl2024}, we denote by $\cA^{(N)}$ the  $\sigma$-algebra generated by the random pedigree
$P^{(N)} \defeq
(P^{(N)}_{0}, P^{(N)}_{1})$  and the initial positions 
$\big ( X_j(0) \big )_{j \in [n]}$ of the sampled genes within it. 
Recall that $\Psi$ is a Poisson point process on 
$[0,\infty) \times ( \Delta \setminus \{ \mathbf{0} \}  )$ with intensity $\dd t \, \frac{1}{\langle x,x \rangle} \Xi (\dd x)$, where $\Xi$ is characterized by 
\eqref{eq:PhiNconv} and \eqref{eq:Xi}. 
We let $(\Omega,\cF,\bf P)$ be the probability space on which the inhomogeneous
$(\Psi,c_{\mathsf{pair}}^{})$-coalescent 
in the sense of Definition~\ref{def:inhomcoalescentfullygeneral}, along with $\Psi$, are defined.

\begin{thm} \label{thm:fddconvergence}
Assume that  \eqref{eq:fixedN}, \eqref{eq:exchangeable},  \eqref{eq:limcN=0} and \eqref{eq:PhiNconv} hold. 
Fix $n\in\mathbb{N}$ and $\xi_0 \in \mathcal{S}_n$. For all
$N \geqslant n$, we let 
$\Pi^{N,n}$ 
be the ancestral process defined in Definition~\ref{def:Pi^{N.n}}, 
starting from $\Pi^{N,n}_0 = \xi_0$. Then,
for any $k\in\N$,  $t_1^{},\ldots,t_k^{} > 0$
and  $\xi_1^{}, \ldots, \xi_k^{} \in \cE_n^{}$, the following convergence holds in distribution:
\begin{equation*}
\P^{(N)} \big (  \Pi^{N,n}_{\widetilde t_1^{}} = \xi_1^{}, \ldots, \Pi^{N,n}_{\widetilde t_k^{}} = \xi_k^{}         \big  | \cA^{(N)}          \big )
\mathop{\longrightarrow}_{N\to\infty}
{\bf P} \big (  \Pi^{n}_{t_1^{}} = \xi_1^{}, \ldots, \Pi^{n}_{t_k^{}} = \xi_k^{}         \big  | \Psi       \big ),
\end{equation*}
where $\widetilde t^{} \defeq \lfloor t^{} / c_N^{} \rfloor$, and $\Pi^{n}$ is
the inhomogeneous $(\Psi,c_{\textnormal{pair}}^{})$-coalescent 
with initial condition $\Pi^{n}_0=\xi_0$.
\end{thm}

The proof is contained in Sections \ref{sec:pedigree}-\ref{sec:paircouplings} and summarized at the end of Section~\ref{sec:inhomcoal}.

\begin{remark} 
Averaging the conditional law of the inhomogeneous $(\Psi,c_{\mathsf{pair}})$-coalescent over the Poisson point process $\Psi$ restores time homogeneity. The resulting marginal process is precisely the $n$-$\Xi$-coalescent $\Pi^{n,\mathsf{avg}}$ of Theorem~\ref{thm0} (cf.\ the Poisson construction in \cite[Section~3]{Schweinsberg2000}). Therefore, by averaging over $\cA^{(N)}$ and $\Psi$ in Theorem~\ref{thm:fddconvergence}, we recover the annealed limit theorem, Theorem~\ref{thm0}, as a corollary. 

In contrast to Theorem~\ref{thm0} which says the law of the (rescaled) process $\Pi^{N,n}$ converges to a deterministic law (that of the $\Xi$-coalescent), Theorem~\ref{thm:fddconvergence} implies that the conditional law of the (rescaled) process $\Pi^{N,n}$ converges, in general, to a random law. Furthermore, the latter is random if and only if $c_{\textnormal{pair}}<1$. We discuss some sufficient conditions for $c_{\mathsf{pair}}=1$ in Appendix~\ref{app:sufficient}.
\hfill \mbox{$\Diamond$}
\end{remark}

\begin{remark} \label{rmk:tyukin}
In the pure Kingman case, i.e.\ the special case when $\Xi'(\Delta \setminus \{\mathbf{0}\}) = 0$, a version of Theorem~\ref{thm:fddconvergence} has been proved in~\cite{TyukinThesis2015}; see also Sections~\ref{ex:diploidWF} and~\ref{ex:diploidind}.
\hfill \mbox{$\Diamond$}
\end{remark}

\begin{remark}
\label{rmk:choiceofstart}
Recall from Definition~\ref{def:Pi^{N.n}} that the evolution of 
$\Pi^{N,n}$ depends, at least in principle, on the choice of the initial positions 
$X_1(0),\ldots,X_{|\xi_0^{}|}(0)$ of the genealogies associated with the blocks of $\Pi^{N,n}$. However, as detailed in Section~\ref{sec:pedigree}, the pedigree is exchangeable, that is, its law is invariant with regard to how individuals are labeled. This means that the law of $\Pi^{N,n}$ is the same for \emph{any deterministic} choice of pairwise distinct $X_1(0),\ldots,X_{|\xi_0^{}|}(0)$ that is consistent with the grouping of samples into individuals in the sense that $I_i(0) = I_j(0)$ for all 
$\{C_i,C_j \} \in \xi_0^{}$.  

Consequently, we may even start with a random sample, under the condition that it be \emph{independent of the pedigree}. On the other hand, it is clear that letting the sample depend on the pedigree by choosing, say, members of the same family whenever possible, will affect the evolution of
$\Pi^{N,n}$ in a non-trivial way, incongruent with Theorem~\ref{thm:fddconvergence}.
\hfill \mbox{$\Diamond$}
\end{remark}

\begin{remark} 
\label{rmk:shortlivedjumps}
Theorem~\ref{thm:fddconvergence} implies in particular that 
$\P^{(N)}\big (\Pi^{N,n}_{\widetilde t} \in \cS_n \setminus \cE_n \, \big | \,
\cA^{(N)} \big )$ converges to $0$ for all fixed $t > 0$. On the other hand, a pair of ancestral genes, even when located in the same parent, remains distinct with conditional probability $1/2$.
Since by the time of the most recent ancestor of the $n$-sample, $n-1$ lineages must have found their ``coalescence partners'' (possibly some of them at the same time), we have
$\P^{(N)} \big (\Pi^{N,n}_{\widetilde t} \in  \cE_n
\textnormal{ for all } t > 0 \, \big | \, \cA^{(N)} \big )\leq 2^{1-n}$.
Consequently, 
\begin{equation*}
\liminf_{N \to \infty} \P^{(N)} \big (\Pi^{N,n}_{\widetilde t} \in \cS_n \setminus \cE_n
\textnormal{ for some } t > 0 \, \big | \, \cA^{(N)} \big ) 
\geqslant 1 - 2^{1-n}.
\end{equation*}
However, states in $\cS_n \setminus \cE_n$ are short-lived as, due to the absence of selfing (see above), these two genes are quickly scattered 
across different parent individuals.

In particular, $\Pi^n$ cannot be the limit of $\Pi^{N,n}$ in the Skorokhod ($J_1$) topology. However, it can be shown that $\mathsf{cd} (\Pi^{N,n})$, the complete dispersal of $\Pi^{N,n}$, does converge in the Skorokhod sense.
This implies convergence of $\Pi^{N,n}$ in the Meyer-Zheng topology~\cite{MeyerZheng1984}. To see this, we use Skorokhod's representation theorem 
to construct a probability space
$(\widetilde \Omega,\widetilde \cF,\widetilde \P)$ carrying $(\widetilde \Pi^{N,n})_{N \geq n}$ and 
$\widetilde \Pi^n$ such that 
$\widetilde \Pi^{N,n}$ has for all $N$ the same distribution as $\Pi^{N,n}$ and $\widetilde \Pi^n$ the same distribution as $\Pi^n$, while
$\mathsf{cd} (\widetilde \Pi^{N,n})$ converges to $\widetilde \Pi^n$ almost surely in the Skorokhod $J_1$-topology. To be more precise, we can obtain such a coupling by applying Skorokhod's representation theorem to the joint law 
\begin{equation*}
\mu_N^{} = \cL \big (
(\Pi^{K,n})_{K\geq n},\mathsf{cd} (\Pi^{N,n})
\big ),
\end{equation*}
where $(\Pi^{K,n})_{K\geq n}$ is an independent coupling, and taking a diagonal sequence.

We claim that $\widetilde \Pi^{N,n}$ converges almost surely to 
$\widetilde \Pi^n$ in the Meyer-Zheng topology. For this, it is sufficient to show that, almost surely
\begin{equation*}
\lim_{N \to \infty} \int_0^\infty
f \big (t,\widetilde \Pi^{N,n}_{\lfloor t/c_N^{} \rfloor } \big )  
\ee^{-t}
\dd t 
=
\int_0^\infty
f \big (t, \widetilde \Pi^n_t \big ) \ee^{-t} \dd t
\end{equation*}
for every bounded and continuous function
$f : \R_{\geqslant 0} \times \cS_n \to \R$; indeed, with $\lambda$ denoting the Lebesgue measure on $\R$ and $T > 0$, the absence of self-fertilisation and the fact that any pair of gene genealogies in different individuals meet in the same individual with probability of order $c_N^{}$ per generation implies that
\begin{equation*}
\lambda \big ( \{
t \in (0,T) : \widetilde \Pi_{\lfloor t/c_N^{} \rfloor}^{N,n} \neq \mathsf{cd}
(\widetilde \Pi^{N,n}_{\lfloor t/c_N^{} \rfloor})
\} \big ) 
\stackrel{N \to \infty}{\longrightarrow} 0
\end{equation*}
almost surely.
 This allows to replace $\widetilde \Pi_{\lfloor t/c_N^{} \rfloor}^{N,n}$ in the integral by
$\mathsf{cd}
(\widetilde \Pi^{N,n}_{\lfloor t/c_N^{} \rfloor})$,
and the convergence follows from the almost sure convergence in the $J_1$-topology of 
$\mathsf{cd} (\widetilde \Pi^{N,n})$ to 
$\widetilde \Pi^n$.
\hfill \mbox{$\Diamond$} 
\end{remark}

The rest of the paper is organized as follows.
Section~\ref{sec:proof-main-result} is devoted to the proof of Theorem ~\ref{thm:fddconvergence}. Readers who are primarily interested in examples and simulations can skip to sections \ref{sec:examples} and \ref{sec:data}. 


Section~\ref{sec:examples} provides a few examples. These include  diploid versions of the Wright-Fisher model, models with random individual fitness as well as a model in which a single pair or a single individual produces a large family with a certain probability, while the remaining families are small. We also discuss briefly a `pseudo' two-sex version, where a distinction between two different sexes can be incorporated into our framework by assigning them in a uniform manner in each generation. 

In Section~\ref{sec:data}, we carry out simulations of the site-frequency spectrum and total branch lengths of the coalescent trees in a simple example. In the final Section~\ref{sect:extensions}, we discuss briefly the difficulties associated with simulating concrete realisations of our limiting model and close by providing some outlook to future work. 

Appendix~\ref{app:jumphold} gives an alternative jump-hold characterisation of the inhomogeneous coalescent, useful for simulation. Appendix~\ref{app:lemmaproof} contains a technical proof concerning aggregate transition probabilities, and Appendix~\ref{app:sufficient} gives sufficient conditions for $c_{\mathsf{pair}}=1$, under which the limiting coalescent has no large-merger component.

\section{Proof of the main result}\label{sec:proof-main-result}


In Subsection~\ref{sec:pedigree}, we describe the pedigree and coalescent transitions conditional on offspring numbers, explain how large families lead to paintbox mergers, and give a roadmap for the proof. Subsection~\ref{sec:inhomcoal} establishes continuity properties of the inhomogeneous coalescent and combines these with an approximation lemma (Lemma~\ref{lem:largerep}) to prove Theorem~\ref{thm:fddconvergence}. The approximation lemma is proved in Subsection~\ref{sec:paircouplings} by coupling coalescents and ancestral processes on a shared pedigree and analysing their joint annealed transitions. A separation-of-timescales argument reduces the required convergence to the asymptotic behaviour of aggregate transition probabilities.

\subsection{Transitions of the process $\Pi^{N,n}$ and proof strategies} \label{sec:pedigree}

Here we provide a detailed formal description of the random pedigree implied by the diploid Cannings model introduced in 
Section~\ref{SS:Cannings}, and use this to derive the transition probabilities of the process $\Pi^{N,n}$, 
in particular in the presence of large families (we only talk about SG in the limit $N\to\infty$).

\subsubsection{The pedigree conditional on offspring numbers}
\label{subsec:condpedigree}

We will now give a more detailed description of the law of the pedigree, conditional on the (random) offspring matrices introduced in Section~\ref{SS:Cannings}. 
This means  that we need to specify, for each $N$, two random maps $P_{0}^{(N)}$ and
$P_{1}^{(N)}$ from $[N] \times \N_0$ to $[N]$, mapping each individual 
$(i,g)$ to its parents $P_{0,i,g}^{(N)}$ and $P_{1,i,g}^{(N)}$.

Conditional on the offspring matrices
$\big (V^{(g)} \big )_{g \in \N}$,
we determine  $P^{(N)}_0$ and 
$P^{(N)}_1$
in two steps. In the first step, for each generation $g \in \N_0$, we will assign each child $(k,g) \in [N] \times\N$ to a pair of parents, using the following
`balls-in-boxes' construction based on the offspring numbers 
$(V_{i,j}^{(g+1)})_{1 \leqslant i, 
j \leqslant N}$ in the previous generation (recall that generations are numbered backward in time). 
\begin{itemize}
\item We consider an array of $N$ boxes, $V_{i,j}^{(g+1)}$ of them corresponding to the pair
$\{i,j\}$ of parents for each $1 \leqslant i < j \leqslant N$. 
\item Into these boxes, $N$ balls are thrown uniformly so that each ball is sorted into a different box. Each ball represents a different child, and ball number $k$ being sorted into one of the $V_{i,i'}^{(g+1)}$ boxes corresponding to $\{i,i'\}$ means that individual $k$ is a child of parents
$i$ and $i'$.
\end{itemize}
Thus, we have
for any child $(k,g)$ and any $i,i' \in [N]$ (with $i\neq i'$ because we exclude selfing), 
\begin{align}
  & \P^{(N)}\left( \big \{ P^{(N)}_{0,k,g}, 
     P^{(N)}_{1,k,g}   \big \}= \{i,i'\}
    \, \Big| \, (V_{i,j}^{(g+1)})_{1 \leqslant i < j \leqslant N} \right)
 \; = \frac{V_{i,i'}^{(g+1)}}{N}.
\end{align}

More generally, 
for $m \in [N]$ pairwise distinct children $k_1,\dots,k_m \in [N]$ and
\begin{equation*}
i_1,\dots,i_m,i'_1,\dots,i'_m \in [N]
\end{equation*}
with $i_\ell \neq i'_\ell$
for $\ell=1,\dots,m$, we have the following expression for
the conditional probability that $i_\ell$ and  $i'_\ell$ are the parents of the child $k_{\ell}$ for each $\ell\in\{1,\dots,m\}$, given the offspring numbers.
\begin{align} \label{eq:ballsinboxes1}
  & \P^{(N)} \left( \big \{ P^{(N)}_{0,k_{\ell}^{},g}, 
     P^{(N)}_{1,k_\ell^{},g}           \big \} = \{ i_\ell, i'_\ell \} \text{ for } \ell=1,\dots,m
    \, \Big| \, (V_{i,j}^{(g+1)})_{1 \leqslant i < j \leqslant N} \right) \notag \\
  & \; = \frac{1}{ (N)_{m\downarrow}}
    \prod_{\ell=1}^m \left( V_{i_\ell,i'_\ell}^{(g+1)} - \sum_{k=1}^{\ell-1} \1\left(\{ i_\ell,i'_{\ell} \} = \{ i_k,i'_k \} \right) \right)
    = \frac{1}{(N)_{m\downarrow}} \prod_{1\leqslant i < j \leqslant N}  (V_{i,j}^{(g+1)})_{b_{ij}\downarrow},
\end{align}
where  we write
\[
  b_{ij} := \#\{ 1 \leqslant \ell \leqslant m : \{ i_\ell, i'_\ell\} = \{i,j\} \}, \quad 1 \leqslant i < j \leqslant N
\]
for the number of children among $\{k_1,\dots,k_m\}$ whose parents are $\{i,j\}$,
and we used the convention $V_{i,j}^{(g+1)}=V_{j,i}^{(g+1)}$.
Note that for any pair $\{i_\ell^{}, i_\ell' \}$ of parents, their first designated child $k_1^{}$ has $V_{i_\ell^{}, i_\ell'}^{(g+1)}$ boxes to choose from, the next one only $V_{i_\ell^{}, i_\ell'}^{(g+1)} - 1$ and so on. In general, the number of boxes for child $k_\ell^{}$ to choose from is reduced by $\sum_{k=1}^{\ell-1} \1\left(\{ i_\ell,i'_{\ell} \} = \{ i_k,i'_k \} \right)$, which is how many children the designated parental pair had among the children $k_1, \ldots k_{\ell - 1}$. 

In the second step, conditionally on the assignment of children to their parents, the roles of the parents are decided via independent tosses of a fair coin. Independently for each child with parents $i$ and $i'$, we let $i$ be the $0$-parent and $i'$ the $1$-parent (and vice-versa) with probability $1/2$, whence
\eqref{eq:ballsinboxes1}
gives
\begin{align} \label{eq:ballsinboxes}
  & \P^{(N)} \left( P^{(N)}_{0,k_\ell^{},g} = i_\ell^{}, 
     \, P^{(N)}_{1,k_\ell^{},g}  =  i'_\ell  \text{ for } \ell=1,\dots,m
    \, \Big| \, (V_{i,j}^{(g+1)})_{1 \leqslant i < j \leqslant N} \right) \notag \\
  & \; = \frac{1}{2^m}
  \P^{(N)}\left( \big \{ P^{(N)}_{0,k_\ell^{},g}, 
     \, P^{(N)}_{1,k_\ell^{},g} \big \} = \{ i_\ell, i'_\ell \} \text{ for } \ell=1,\dots,m
    \, \Big| \, (V_{i,j}^{(g+1)})_{1 \leqslant i < j \leqslant N} \right) \notag \\
  & \; =\frac{1}{2^m (N)_{m\downarrow}} \prod_{1\leqslant i < j \leqslant N}  (V_{i,j}^{(g+1)})_{b_{ij}\downarrow}.
\end{align}

Averaging with respect to $V^{(g+1)}$ gives 
\begin{align}
\label{pedigree.law}
  & \P^{(N)}\left( P^{(N)}_{0,k_\ell^{},g} = i_\ell,
     P^{(N)}_{1,k_\ell^{},g}  = i'_\ell \text{ for } \ell=1,\dots,m \right) \notag \\
    & \quad = \frac{1}{2^m(N)_{m\downarrow}} \E^{(N)}\left[ \prod_{1 \leqslant i < j \leqslant N}  (V_{i,j}^{(g+1)})_{b_{ij}\downarrow} \right]
\end{align}
for any $g \in \N_0$. In particular, we see that the random vector $\big( P^{(N)}_{0,k,g}, P^{(N)}_{1,k,g}
\big)_{1\leqslant k \leqslant N} \in  
([N] \times [N])^{N}$
is exchangeable for each $g \in \N$. Furthermore, by
construction, as a family indexed by $g$, these vectors are i.i.d. 
This completes our description of the distribution of the coalescing random walks given in Section~\ref{SS:Cannings}, and thus of the ancestral process $\Pi^{N,n}$.



When fixing a realisation of the pedigree $P^{(N)}$,  $\Pi^{N,n}$ is not  Markovian due to the fact that its evolution can only be determined by considering the underlying coalescing random walks $X_1,\ldots,X_n$ modeling the actual gene genealogies (see Definition~\ref{def:ancestrallines}); this is because fixing the pedigree destroys the exchangeability within generations. However, because the law of the (time-slices of) the pedigree is exchangeable given the offspring numbers, by considering a partial average and fixing only the offspring matrices $V^{(g)}$, while averaging over the conditional distribution  given in \eqref{eq:ballsinboxes}, we regain exchangeability. This allows for an intrinsic description of the (still time-inhomogeneous) coalescent process, which will be the theme of the following section.

\subsubsection{Transitions of $\Pi^{N,n}$ conditional on offspring numbers}
\label{subsec:transitionsconditionalonoffspringnumbers}
Let $\xi \in \cS_n$ and assume that 
\begin{equation} \label{PiNnstart}
\Pi^{N,n}_{g} = \xi = \big \{ \{ C_1,C_2 \}, \ldots, 
\{ C_{2x-1}, C_{2x} \}, C_{2x+1}, \ldots, C_b \big \}
\end{equation}
for any $g \in \N_0$.
Recall that the conditional distribution of $\big ( P_{0,\cdot,g}^{(N)}, P_{1,\cdot,g}^{(N)}  \big )$ given the offspring numbers is exchangeable with respect to how the offspring are labelled, see \eqref{eq:ballsinboxes}. 
Consequently, the conditional distribution of 
$\Pi^{N,n}_{g+1}$
given $\Pi^{N,n}_g$ and the offspring numbers 
$V^{(g+1)}$
does not depend on the precise positions 
of the ancestral lineages.
Of course, we must respect the grouping of genes into individuals, that is, whenever 
$\{A,B\} \in \xi$ for $A \neq B$, the associated lineages must reside within the same individual. 
We may therefore assume without loss of generality that $X_1 (g) = (0,1), X_2 (g) = (1,1),\ldots,
 X_{2x-1} (g) = (0,x), X_{2x} (g) = (1,x), X_{2x+1} (g) = (0,x+1),
 X_{2x+2} (g) = (0,x+2), \ldots, X_b (g) = (0,b)$.

We fix a realisation of the offspring matrix $V^{(g+1)} = \big ( V^{(g+1)}_{i,j}    \big )_{1 \leqslant i,j \leqslant N}$ and draw a time-slice $\big ( P_{0,\cdot,g}^{(N)}, P_{1,\cdot,g}^{(N)}  \big )$ of the pedigree like Figure \ref{fig1a}, consisting of random maps
$P^{(N)}_{0,\cdot,g}$
and
$P^{(N)}_{1,\cdot,g}$ from $[N]$ to itself which are distributed as in \eqref{eq:ballsinboxes}, along with Mendelian
coin flips $M_{c,j,g}$. Then, we construct 
$X_1(g+1),\ldots,X_b(g+1)$ and $\Pi^{N,n}_{g+1}$ according to Subsection~\ref{subsec:fixedpedigrees} and define for all 
$\xi, \eta \in \cS_n$
\begin{equation} \label{eq:transitionsconditionedonnumbers}
\pi^{N,n}(V^{(g+1)};\xi,\eta) \defeq
\P^{(N)} \big ( \Pi^{N,n}_{g+1}= \eta \, \big | \, V^{(g+1)}, \Pi^{N,n}_g = \xi             \big ).
\end{equation}
Due to time homogeneity, this depends only on the realisation $V^{(g+1)}$ of the offspring matrix but not explicitly on $g$. We therefore assume, for the remainder of this subsection, that $g=0$ and drop the time index, that is, we will write $V$ for $V^{(1)}$,
$P_{c,i}^{(N)}$ for $P_{c,i,0}^{(N)}$ and
$M_{c,i}$ for $M_{c,i,0}$.

After averaging over the entire pedigree including the offspring numbers, $\Pi^{N,n}$ becomes a discrete-time Markov chain on $\cS_n$ with transition matrix
\begin{equation}\label{eq:transitionsconditionedonnumbersaveraged}
\pi^{N,n} (\xi,\eta) \defeq \E^{(N)} \big [  \P^{(N)} \big ( \Pi^{N,n}_1 = \eta \, \big | \, V, \Pi^{N,n}_0 = \xi          \big )          \big ] =
\E^{(N)}[\pi^{N,n}(V;\xi,\eta)],
\end{equation} 
where the expectation $\E^{(N)}$ is taken with regard to
the law of $V$; this is the setting that was considered 
in~\cite{BirknerEtAl2018}. 

We now give a more explicit description 
of $\pi^{N,n} (V;\xi,\cdot)$
when $\xi$ is completely dispersed, i.e. $\xi = \mathsf{cd} (\xi)$. In this case the ancestral genes are all located in distinct individuals. Hence, $x = 0$ in 
\eqref{PiNnstart} and $\xi = \{C_1,\ldots,C_b\}$. By symmetry, we may again assume without loss of generality that the gene genealogies associated with the blocks satisfy
$
X_j (0) = (0,j) 
$
for all $j=1,\ldots,b$. In other words, each of these genes is inherited from  its $0$-parent (recall the paragraph before Definition \ref{Def:M}). We stress that this is purely done for the sake of convenience; by adapting the definition of $\widehat V_i$ (see below) appropriately, we could also deal with the situation that some genes come from the $0$-parent and others from the 
$1$-parent. 

Now note that $C_i$ and $C_j$ will coalesce if and only if individuals $i$ and $j$ have the same
$0$-parent \emph{and pick the same gene} within it. Formally, they coalesce if and only if
\begin{equation*}
P^{(N)}_{0,i} = P^{(N)}_{0,j}
\quad \textnormal{and} \quad
M_{0,i} = M_{0,j}.
\end{equation*}
We describe the children's choice of $0$-parents by mimicking the ``balls-in-boxes'' construction of parent-child relations from the previous subsection. 
For each $i \in [N]$, we let
\begin{equation*}
\widehat V_i = \# \big \{ j \in [N] : 
P^{(N)}_{0,j}  = i \big \}
\end{equation*}
be the number of children that chose $i$ as their $0$-parent. Clearly, $\sum_{i=1}^N \widehat V_i=N$.  Because parental roles are assigned by independent fair coin flips,
for each 
$i \in [N]$, $\widehat V_i$ has conditional distribution
$\textnormal{Binomial}(V_i,1/2)$ given $V_i$ (see also the derivation of the annealed pair coalescence probability
$c_N^{}$ in \eqref{eq:cN1}). Recall also that 
$V_i = \sum_{j=1}^N V_{i,j}^{(0)}$ is the total number of children of individual $i$. 
Thus, for any $i \in [N]$ and $j \in [n]$,
\begin{equation} \label{eq:choiceofparentV}
\P^{(N)}  \Big (  P^{(N)}_{0,j} = i \, \Big | \, 
V
\Big )
=
\E^{(N)} \Big [ \frac{\widehat V_i}{N}                \, \Big | \,  V \Big  ] =\frac{V_i}{2N}
\end{equation}
and, more generally for any 
$m \in [N]$ and 
$i_1^{},\ldots,i_m^{} \in [N]$,
\begin{align*}
\P^{(N)}  \Big (  P^{(N)}_{0,1} = i_1^{}, \ldots,
P^{(N)}_{0,m} = i_m^{}
\, \Big | \, 
V
\Big )
 & =
 \frac{1}{(N)_{m\downarrow}} \E^{(N)} \Big [  \prod_{i=1}^N \big ( \widehat V_i   \big )_{b_i^{} \downarrow}              \, \Big | \,  V \Big  ]\\
& = \frac{1}{(N)_{m\downarrow}} \E^{(N)} \Big [ \prod_{\ell=1}^r \big ( \widehat V_{j_{\ell}}   \big )_{k_{\ell}\downarrow}              \, \Big | \,  V \Big  ],
\end{align*}
where $b_i^{} \defeq \# \{ k \in [m] : i_k = i            \}$
is the number of individuals in the $m$-sample whose $0$-parent is $i$.  In the last equality, we suppose 
$\{i_1^{},\ldots,i_m^{} \}$ is composed of  $r$  distinct elements $\{j_1,\ldots, j_r\}$, and that  $k_{\ell}$ many of $\{i_1^{},\ldots,i_m^{} \}$ are equal to $j_\ell$, for $\ell=1,2,\ldots, r$.

From this, we can read off the (conditional) transition probabilities of $\Pi^{N,n}$.  Suppose $\xi \in \cE_n$ consists of $b$ blocks for the rest of this section. 

Assuming first that $\eta \in \cE_n$ arises from $\xi$ by merging $r$ groups of blocks of sizes $k_1^{},\ldots,k_r^{} \geqslant 1$, we obtain
the more explicit formula
\begin{equation} \label{eq:transprobV}
\pi^{N,n}(V;\xi,\eta) = \frac{2^{r - k_1^{} - \ldots - k_r^{}}}{(N)_{b \downarrow}} 
\sum_{\substack{ i_1,\ldots,i_r \in [N]   \\ \textnormal{pairwise distinct}}}
\E^{(N)} \Big [  \prod_{\ell=1}^r \big ( \widehat V_{i_\ell}^{}   \big )_{k_\ell^{} \downarrow}              \, \Big | \,  V \Big  ].
\end{equation}
The factor $2^{r - k_1^{} - \ldots - k_r^{}}$ in \eqref{eq:transprobV} comes from the fact that for each $\ell\in\{1,2,\ldots,r\}$, if the $k_\ell$ samples have the same $0$-parent $i_\ell^{}$, they pick the same gene copy of 
$i_\ell^{}$ with probability $2^{1-k_\ell}$.

We need to take some care when $\eta \in \cS_n \setminus \cE_n$. 
In that case, the $i_1^{},\ldots,i_r^{}$ in the sum above need not be pairwise distinct; in fact, when two groups containing respectively $k_p^{}$ and $k_q^{}$ blocks in $\xi$ coalesce in two different genes of the same parent in $\eta$, we must have $i_p^{} = i_q^{}$. 

Due to the separation of time scales (ancestral genes are dispersed across different individuals much faster than they coalesce; see Subsection~\ref{subsec:timescalesep}), we will not be particularly concerned with transitions from $\xi$ to a particular $\eta \in \cS_n$. Rather, we need to know the total transition rate from $\xi$ to \emph{some} element of $\cS_n$ that has a certain complete dispersal. That is, we fix $\eta \in \cE_n$ and consider the sum
\begin{equation}\label{eq:transitionsconditionedonnumbersaggregated}
\widetilde \pi^{N,n} (V;\xi,\eta) \defeq
\sum_{\eta' \in \mathsf{cd}^{-1}(\eta)} \pi^{N,n} (V;\xi,\eta').
\end{equation}

Instead of summing over distinct individuals as in 
\eqref{eq:transprobV}, we will sum over distinct \emph{genes} in 
$\{0,1\} \times [N]$.
Recall that
$\eta \in \cE_n$ consists of $r$ blocks which are respectively unions of
$k_1^{},\ldots,k_r^{} \geqslant 1$ blocks of $\xi$.
In order to observe a transition from $\xi$ to some $\eta' \in \mathsf{cd}^{-1} (\eta)$,
there have to be $r$ pairwise different genes 
$(c_1^{},i_1^{}),\ldots,(c_r^{},i_r^{}) \in \{0,1\} \times [N]$ so that a group of 
$k_\ell^{}$ lineages associated with $\xi$ coalesces in gene $(c_\ell^{},i_\ell^{})$.

While $(c_1^{},i_1^{}),\ldots,(c_r^{},i_r^{})$ are pairwise distinct, the individuals $i_1^{},\ldots,i_r^{}$ do not need to be, contrary to \eqref{eq:transprobV}. For $\ell = 1,\ldots,r$, we will write
$\widetilde k_\ell^{}$ for the number of lineages that meet in individual $i_\ell^{}$ if $i_\ell^{}$ appears for the first time at position $\ell$. Otherwise, we let $\widetilde k_\ell^{} \defeq 0$. More precisely, whenever 
$\ell = \min \{ s \in [r] : i_s^{} = i_\ell^{}    \}$ (i.e. $i_{\ell}$ is different from $i_s$ for $s<\ell$), we set 
\begin{equation*} 
\widetilde k_\ell^{} 
\defeq 
\sum_{\substack{s \in [r] \\ i_s^{} = i_\ell^{} }  } k_s^{}
\end{equation*}
and otherwise, we let 
$\widetilde k_\ell^{} \defeq 0$. 

For example, suppose $r=3$ and 
$i_1^{} = i_2^{} \neq i_3^{}$.
Suppose
the $k_1+k_2+k_3$ sampled genes coalesced into $r=3$ different genes in $0$-parents $i_1, i_2,i_3$, and  $k_{s}$ sampled genes have $0$-parent $i_s$ for $s=1,2,3$. Then $k_1$ and $k_2$ sampled genes are inherited respectively from the 2 genes of the same $0$-parent $i_1=i_2$, so that $\widetilde{k}_1=k_1+k_2$, $\widetilde{k}_2=0$ and  $\widetilde{k}_3=k_3$.

Recall from Eq.~\eqref{eq:transprobV} that the probability that the lineages end up in the ``correct'' individuals (ignoring the Mendelian coin flips for now) is given by 
\begin{equation*}
\frac{1}{(N)_{b \downarrow}} 
\E^{(N)} \Big [  \prod_{\ell=1}^r \big (  \widehat V_{i_\ell}^{}   \big )_{\widetilde k_\ell^{} \downarrow}              \, \Big | \,  V \Big  ].
\end{equation*}
Conditionally on this, the probability that the Mendelian coin flips agree with the choice of $c_1^{},\ldots,c_r^{}$ is 
$2^{-k_1^{}- \ldots - k_r^{}}$. Hence, by summing over all choices of 
$(i_1^{},c_1^{}),\ldots,(i_r^{},c_r^{})$, we obtain
\begin{equation} \label{eq:aggregatedtransprobs}
\widetilde \pi^{N,n} (V;\xi,\eta)
=
\frac{2^{-k_1^{} - \ldots - k_r^{}}}{(N)_{b \downarrow}} 
\sum_{\substack{ (c_1^{},i_1^{}),\ldots,(c_r^{},i_r^{}) \in \{0,1\} \times [N] \\ \textnormal{pairwise distinct}}}
\E^{(N)} \Big [  \prod_{\ell=1}^r \big (  \widehat V_{i_\ell}^{}   \big )_{\widetilde k_\ell^{} \downarrow}              \, \Big | \,  V \Big  ].
\end{equation}


\subsubsection{Coarse graining the pedigree in the presence of large progeny}
\label{subsec:largeoffspring}
We will now see that the situation simplifies significantly in the presence of individuals with large individual progeny.
We call a generation $g$ an $\varepsilon$-SG ($\varepsilon$-sweepstake generation) if 
$\| \widetilde V^{(g)} \| \geq \varepsilon$ (see \eqref{eq:discretepaintbox} below).
Recall our assumption \eqref{eq:limcN=0} that for the annealed pair-coalescence probability $c_N^{}$, we have $c_N^{} \to 0$ as $N \to \infty$. 
In particular, this means that, for any $\varepsilon > 0$, $\varepsilon$-SGs (which imply a macroscopic chance to observe a coalescence event), will typically be separated by an order of $1 / c_N^{} \to \infty$ generations. Recall that generations in the pedigree are constructed in an i.i.d.\ fashion and each parental gene is (due to Mendel's law of random segregation) picked as the ancestor with probability $1/2$. We expect this to lead to a sort of `mixing' of the ancestral lines on the pedigree, motivating us to anticipate the following.

We imagine that whenever we encounter an ($\varepsilon$-)SG (in generation $g+1$, say), the ancestral genes of the samples in the offspring generation $g$ are approximately \emph{uniformly dispersed}
in the following sense. 
Whenever $\Pi^{N,n}_g$ consists of, say, $k$ blocks, the positions 
$X_1^{} (g), \ldots, X_k^{} (g)$ 
of the corresponding random walkers from Section~\ref{sec:pedigree} are approximately distributed like
\begin{equation*}
\big ( (M_1,I_1), \ldots, (M_k,I_k)                    \big ),
\end{equation*}
where $I_1,\ldots,I_k$ are sampled uniformly without replacement from $[N]$ and 
$M_1, \ldots, M_k$ are independent $\textnormal{Ber}(1/2)$ distributed and independent of 
$(I_j)_{1 \leqslant j \leqslant k}$. Note that this implies that $\Pi_g^{N,n} \in \cE_n$, i.e. $\Pi_g^{N,n}$ is a partition. Both 
$I_1, \ldots, I_k$ and $M_1, \ldots, M_k$ are independent of the pedigree $P^{(N)}$. 
The underlying idea is that $\varepsilon$-SGs become increasingly rare for large $N$, so that the ancestral lineages have enough time to 
`forget' their initial positions.

In particular, we expect for any
$i \in [N]$ and $c \in \{0,1\}$
the following approximation to hold when $g + 1$ is an
$\varepsilon$-SG
for any $j \in [k]$ and large $N$
\begin{equation*}
\P^{(N)} \Big ( X_j (g+1) = (c,i) \, \Big | \,  P^{(N)} \Big )
\approx
\P^{(N)} \Big ( 
P^{(N)}_{M_j,I_j,g} = i, \, M_{M_j,I_j,g}^{} = c
\, \Big | \, P^{(N)}                \Big ).
\end{equation*}
We stress that this approximation is merely a heuristic and we presently do not 
make it rigorous. 
Instead, our strategy will be to replace this approximation by introducing an auxiliary coalescent process that shows the behaviour indicated by the right-hand side, and show that its conditional finite-dimensional distributions are close to those of the original process (see the roadmap in Subsection~\ref{subsec:roadmap}).

For the time being, we take this step for granted and proceed to evaluate the right-hand side in a rigorous fashion.
Since $(M_j,I_j)$ is independent of $P^{(N)}$, it follows that
$P^{(N)}_{M_j,I_j,g}$ and $M_{M_j,I_j,g}$
are conditionally independent given the offspring numbers
$V^{(g+1)}$. We can therefore take the average with respect to the conditional law of $P^{(N)}_{M_j,I_j,g}$ given 
$V^{(g+1)}$; see \eqref{eq:ballsinboxes}. Indeed, recall from the beginning of this section how, conditional on $V^{(g+1)}$,
$P^{(N)}$ was constructed by sorting balls into boxes; now pick one of them uniformly, corresponding to the gene
$(I_j,M_j)$. Mathematically, this is equivalent to redrawing the position of this ball within the previously described array of boxes. Hence,
\begin{equation}
\label{eq:onechoice}
\begin{split}
&\P^{(N)} \Big (
P^{(N)}_{M_j,I_j,g} = i, \, M_{M_j,I_j,g}^{} = c
\, \Big | \, P^{(N)}                \Big ) \\
& \quad = 
\P^{(N)} \Big ( P^{(N)}_{M_j,I_j,g} = i, \, M_{M_j,I_j,g}^{} = c
\, \Big | \, 
V^{(g+1)} \Big ) \\
& \quad =
\frac{1}{2}  \P^{(N)} \Big ( P^{(N)}_{M_j,I_j,g} = i \, \Big | \, V^{(g+1)}     \Big ) \\
& \quad =
\frac{1}{4} \sum_{i' = 1}^N \frac{V^{(g+1)}_{i,i'}}{N} \\
& \quad =
\frac{V_i^{(g+1)}}{4N}.
\end{split}
\end{equation}
The second equality holds because the conditional probability of $(c,i)$ being picked as the ancestor given that individual $i$ is the 
$M_j$-parent is $1/2$. The third  equality is due to the fact that, conditionally on $I_j$ being a child of the pair $\{i,i'\}$,  $i$ is in fact the 
$M_j$-parent
with probability~$1/2$.

More generally, for 
$i_1^{},\ldots, i_k^{} \in [N]$ and $c_1^{}, \ldots, c_k^{} \in \{0,1\}$, we expect, for large $N$, that
\begin{equation} \label{eq:largeoffspringapprox}
\begin{split}
&\P^{(N)} \Big ( X_1 (g+1) = (c_1^{}, i_1^{}), \ldots, X_k(g+1) = (c_k^{}, i_k^{} ) \, \Big | \,  P^{(N)} \Big ) \\
& \quad \approx
\P^{(N)} \Big ( 
P^{(N)}_{M_1,I_1,g} = i_1^{}, M_{M_1,I_1,g}^{} = c_1^{}  , \ldots, 
P^{(N)}_{M_k,I_k,g} = i_k^{}, M_{M_k,I_k,g}^{} = c_k^{}  
\, \Big | \, V^{(g+1)}     \Big ),
\end{split}
\end{equation}
where
\begin{equation} \label{eq:generalapproximation}
\begin{split}
&\P^{(N)} \Big ( 
P^{(N)}_{M_1,I_1,g} = i_1^{}, M_{M_1,I_1,g}^{} = c_1^{}  , \ldots, 
P^{(N)}_{M_k,I_k,g} = i_k^{}, M_{M_k,I_k,g}^{} = c_k^{}  
\, \Big | \, V^{(g+1)}     \Big ) \\
& \quad = \prod_{\ell = 1}^k \frac{V_{i_\ell^{}}^{(g+1)}}{4N} + O(N^{-1}).
\end{split}
\end{equation}
Here, we used \eqref{eq:onechoice} and that, for every $\ell = 1,\ldots,k$,
\begin{equation*}
\begin{split}
&\P^{(N)} \Big ( 
P^{(N)}_{M_\ell,I_\ell,g} = i_\ell^{}, M_{M_\ell,I_\ell,g}^{} = c_\ell^{} 
\, \Big | \, V^{(g+1)},   
P^{(N)}_{M_k,I_k,g} = i_k^{}, M_{M_k,I_k,g}^{} = c_k^{}  , k=1,\ldots,\ell-1
\Big ) \\
& \quad = \P^{(N)} \Big ( 
P^{(N)}_{M_\ell,I_\ell,g} = i_\ell^{}, M_{M_\ell,I_\ell,g}^{} = c_\ell^{} 
 \, \Big | \, V^{(g+1)} \Big )  + O(N^{-1})
\end{split}
\end{equation*}
since any choice of parents reduces the number of available boxes corresponding to that parental pair by 1 (cf. \eqref{eq:ballsinboxes1}). In essence, we incur an error of order $1/N$ by acting as if individuals sampled their parents with replacement.

For $g\in\N$ we let 
\begin{equation} \label{eq:discretepaintbox}
\widetilde{V}^{(g)} = \bigg( \frac{V^{(g)}_{(1)}}{4N}, \frac{V^{(g)}_{(1)}}{4N}, \frac{V^{(g)}_{(2)}}{4N}, \frac{V^{(g)}_{(2)}}{4N},
  \frac{V^{(g)}_{(3)}}{4N}, \frac{V^{(g)}_{(3)}}{4N}, \dots \bigg) \; \in \Delta \setminus \{ \bf 0   \}
\end{equation}
be the image under $\varphi$
(see \eqref{eq:Xi})
of the order statistics of the total offspring sizes 
$\big (V^{(g)}_i\big )_{i}$.

Associating with each $X_j(g) = (M_j,I_j)$ a block in the partition $(\Pi_g^{N,n})$  and merging all blocks that meet in the previous generation $g+1$,  \eqref{eq:generalapproximation} translates
directly into the following statement on the `aggregate' conditional transition probabilities $\widetilde \pi^{N,n} (V;\xi,\eta)$ introduced in \eqref{eq:transitionsconditionedonnumbersaggregated};
it establishes the connection with the paintbox-construction (see \eqref{eq:paintboxtransitionprime} and the surrounding discussion).
\begin{lemma}\label{lem:largerep}
  For any $\xi,\eta \in \cE_n$, recall that $\widetilde \pi^{N,n}(V;\xi,\eta)$ is the probability of observing a transition from $\xi$ to $\mathsf{cd}^{-1} (\eta)$ for the ancestral process $\Pi^{N,n}$, conditional on the pairwise offspring numbers $V = \big ( V_{i,j}  \big )_{1 \leqslant i,j \leqslant N}$ but averaged w.r.t.\ the conditional 
  (on $V= V^{(g+1)}$) law of the pedigree given in \eqref{eq:ballsinboxes}. Then,
\begin{equation*}
\widetilde \pi^{N,n} (V;\xi,\eta) = p(\widetilde V; \xi, \eta) + O(1/N), 
\end{equation*}
where $\widetilde V\in \Delta$ is defined as in \eqref{eq:discretepaintbox} for the total offspring sizes 
$( V_i )_{1 \leqslant i\leqslant N}$, and  $p$ is given by~\eqref{eq:paintboxtransitionprime}.
\end{lemma}

\begin{proof}
By the exchangeability of the construction of $P^{(N)}$ given $V^{(g)}$, we may  identify the initial position of the lineages associated with the blocks of $\xi$ with the uniform samples
$(C_1,I_1),\ldots,(C_k,I_k)$ used in \eqref{eq:generalapproximation}.
Assume that $\eta$ consists of $r$ blocks, which are respectively unions of $k_1^{},\ldots,k_r^{} \geqslant 1$  blocks of $\xi$. Writing
$\widetilde V = (y_{0,1}^{},y_{1,1}^{},y_{0,2}^{},y_{1,2}^{},\ldots,y_{0,N}^{},y_{1,N}^{},0,\ldots)$,
$|\widetilde V|_1 = 1$ implies that $p(\widetilde V,\xi,\eta)$ as given in \eqref{eq:paintboxtransitionprime} reduces to
\begin{equation*}
p(\widetilde V,\xi,\eta) = 
\sum_{
\substack{(c_1^{},i_1^{}),\ldots,(c_r^{},i_r^{}) \in \{0,1\} \times [N] \\
\textnormal{pairwise distinct}
}} 
y_{c_1^{},i_1^{}}^{k_1^{}} \cdot \ldots \cdot
y_{c_r^{},i_r^{}}^{k_r^{}}.
\end{equation*}
From this representation, the claim follows easily by summing \eqref{eq:generalapproximation} over all 
choices of $(c_1^{},i_1^{}),\ldots,(c_r^{},i_r^{})$ such that $(c_\ell^{},i_\ell^{})$ shows up $k_\ell^{}$ many times on the left-hand side of \eqref{eq:generalapproximation}.
\end{proof}
\begin{remark}
  Recall that in \eqref{eq:paintboxtransitionprime}, we assumed $k_1^{},\ldots,k_r^{} \geqslant 2$. This is owed 
  to the asymmetry between those blocks of $\eta$ that are nontrivial unions of blocks of $\xi$, and those that are isolated blocks of $\xi$; this is due to the fact that, during the paintbox construction, when two or more blocks end up together in the final bucket of size $1 - |y|_1$, they will not merge. However, since in the proof of Lemma~\ref{lem:largerep} we have $|y|_1 = 1$, this final non-merging bucket is absent and we can treat each of the $s$ isolated blocks of $\eta$ as a trivial ``union'' of a single block of $\xi$. 
\end{remark}

We have just explained how
we expect transitions of the coalescent that are induced by a SG to be governed only by the total offspring numbers $V_i$. This allows us to consider a `coarse-grained' version of the pedigree.
Believing this, Lemma~\ref{lem:largerep} tells us that the transition can be (up to a negligible error) described by 
a $\widetilde V$-paintbox merger (see \eqref{eq:discretepaintbox}).


From this point of view, at least as far as SGs are concerned, the relevant information regarding the pedigree in the prelimit reduces to a sequence of paintbox mergers. 
For all $N \in \N$, we define
\begin{equation} \label{eq:PsiNdef}
\Psi^{(N)} := \sum_{g=1}^\infty
\delta_{\big ( g c_N^{}, \widetilde V^{(g)}  \big ) }
\end{equation}
which is a point process on
$[0,\infty) \times \big ( \Delta \setminus \{ {\bf 0}   \}      \big )$.


Just like for $\Psi$, we interpret an atom $(t,x)$ of $\Psi^{(N)}$ as an $x$-paintbox-merger (see the discussion surrounding \eqref{eq:paintboxtransitionprime}) that takes place at time $t$. We now show that $\Psi^{(N)}$ converges
in distribution to $\Psi$, with respect to the vague topology, as $N\to\infty$.

For any $\varepsilon > 0$ and
any $A \subseteq \Delta \setminus 
B_\varepsilon^{} ({\bf 0})$,
denoting the Binomial distribution with success probability $p$ and $n$ trials as $\textnormal{Bin}(n,p)$,
we have for any finite time interval $[S,T] \subseteq (0,\infty)$ that

\begin{equation*}
\Psi^{(N)} \big ( [S,T] \times A    \big )
\sim
\textnormal{Bin} \big (
\# ([S,T] \cap c_N^{} \N), \cL \big ( \widetilde V  \big ) (A)
\big ),
\end{equation*}
and that $\cL \big ( \widetilde V  \big ) = \Phi_N \circ \varphi^{-1}$ with $\Phi_N$ given in \eqref{eq:bigphiN}
Note that \eqref{eq:PhiNconv} is equivalent to 
\begin{align}\label{eq:PhiNconv2}
  \frac{1}{c_N} \Phi_N \circ \varphi^{-1} (\dd x) \mathop{\longrightarrow}_{N\to\infty}
  \frac{1}{\langle x,x \rangle} \Xi(\dd x)
  \quad \text{vaguely on}\; \Delta \setminus \{\mathbf{0}\}.
\end{align}
Hence, 
$c_N^{-1} \cL(\widetilde V) (\dd x)$ converges vaguely to 
$\langle x, x \rangle^{-1} \,  \Xi(\dd x)$ and
$\# ([S,T] \cap c_N^{} \N)$ is asymptotically equivalent to
$c_N^{-1} (T - S)$, whence (as long as $\Xi (\partial A) = 0$),
\begin{equation*}
\begin{split}
&\textnormal{Bin} \big (
\# ([S,T] \cap c_N^{} \N),\, \cL \big ( \widetilde V  \big ) (A) \big )
=
\textnormal{Bin} \big (
\# ([S,T] \cap c_N^{} \N),\, c_N^{} c_N^{-1} \cL \big ( \widetilde V  \big ) (A) 
\big )\\
& \quad \stackrel{N \to \infty}{\longrightarrow}
\textnormal{Poi} \Big  ((T-S) \int_{A} 
\langle x, x \rangle^{-1} \Xi (\dd x)
\Big ).
\end{split}
\end{equation*}
We conclude that as $N \to \infty$, $\Psi^{(N)}$ converges
in distribution, with respect to the vague topology
to 
a Poisson point process with intensity measure
$\dd t \, \langle x, x \rangle^{-1}  \Xi(\dd x) $, i.e.\ $\Psi$. 

Naively, one would like to believe that this immediately yields the convergence of the associated coalescent processes. 
However, boiling the behaviour of the prelimiting coalescents down in this way only works for SGs, because only for SGs we may ignore the error term in Lemma~\ref{lem:largerep}. In addition, 
recall (see \eqref{eq:cpair} and the surrounding discussion) that because of the accumulation of mass at 
${\bf 0}$,
the vague convergence
$
c_N^{-1} \cL(\widetilde V)(\dd x) 
\to
\langle x, x \rangle^{-1} \Xi (\dd x)
$
does not directly explain the appearance of pair mergers due to small families, which is why we had to manually add them back in via adding an atom at 
${\bf 0}$ to $\Xi$. We do something similar on the point process level. 


\subsubsection{Roadmap}
\label{subsec:roadmap}
We now outline the proof strategy that we will pursue in Sections~\ref{sec:inhomcoal} and \ref{sec:paircouplings}.

\begin{enumerate}
\item 
Fix $\varepsilon > 0$ and ignore (at first) generations with relatively small offspring numbers, i.e. generations $g$ with
$\| \widetilde V^{(g)} \| < \varepsilon$. For this, we define the $\varepsilon$-cutoff $\Psi^{(N)}_{\varepsilon}$ of $\Psi^{(N)}$ via
\begin{equation} \label{eq:psicutoff}
\Psi^{(N)}_{\varepsilon} 
\defeq
\sum_{g=1}^\infty 
\1_{\| \widetilde V^{(g)} \| \geqslant \varepsilon}   \delta_{\big ( g c_N^{}, \widetilde V^{(g)}  \big )}.
\end{equation}
\item 
We use $\Psi^{(N)}_{\varepsilon}$ to drive an 
\emph{$\varepsilon$-naive coalescent} $\Pi^{(N,\varepsilon)}$ (see Definition~\ref{def:coalescents} below); every atom
$(t,x)$ of $\Psi^{(N)}_{\varepsilon}$ induces an $x$-merger at time $t$. In anticipation of pair mergers which will appear at rate $c_{\mathsf{pair}}^{}$ in the limit $N \to \infty$, the paintbox mergers are complemented by additional pair mergers as follows: at each point $g c_N^{} \in c_N^{} \N$ in rescaled time for which 
$\| \widetilde V^{(g)} \| < \varepsilon$,
i.e. 
$\Psi_\varepsilon^{(N)} \big (
\{ g c_N^{} \} \times \Delta \setminus
\{ { \bf 0 } \} \big ) = 0$, 
we let each pair of blocks coalesce independently of the others with probability 
$c_N^{} c_{\mathsf{pair}}^{}$. 
\item
For fixed $\varepsilon > 0$, we show that, as $N\to\infty$, the $\varepsilon$-naive coalescent converges to the inhomogeneous $(\Psi_{\varepsilon },c_{\textnormal{pair}}^{})$-coalescent 
as in Definition~\ref{def:inhomcoalescentfullygeneral} with $\Psi$ replaced by the cutoff $\Psi_{\varepsilon}^{}$
defined via
\begin{equation} \label{eq:psilimitcutoff}
\Psi_\varepsilon^{} (A)
\defeq
\Psi \big (A \setminus [0,\infty) \times 
B_\varepsilon^{} ({ \bf 0}) \big )
\end{equation}
for all $A \subseteq [0,\infty) \times \Delta$.
\item 
Letting $\varepsilon \to 0$, we see that the finite-dimensional distributions of the inhomogeneous $\Psi_{\varepsilon}$-coalescent converge to the those of the inhomogeneous $\Psi$-coalescent. This, as well as Step~3, will follow from the continuity statements in the next subsection).
\item 
Finally, in Section~\ref{sec:paircouplings}, we will show that the 
$\varepsilon$-naive coalescent is indeed a good approximation to 
$\Pi^{N,n}$ when $\varepsilon$ is small and $N$ is large. 
This will be achieved via controlling the $L^2$-distance between the conditional finite-dimensional distribution of
$\Pi^{N,n}$ and those of the $\varepsilon$-naive coalescent. To this end, we will analyse the following couplings.
\begin{itemize}
\item
Two conditionally independent realisations of $\Pi^{N,n}$,
\item 
a realisation of the $\varepsilon$-naive coalescent together with a conditionally independent realisation of $\Pi^{N,n}$,
\item 
two conditionally independent realisations of the
$\varepsilon$-naive coalescent. 
\end{itemize}
In all of these couplings, `conditionally independent' refers to them being defined on the same pedigree, but using independent realisations of the Mendelian coins $M_{c,i,g}$ which were introduced in Definition~\ref{Def:M}. This will allow us to evaluate the expectations of products of conditional finite-dimensional distributions that arise when 
evaluating the aforementioned $L^2$-distance.
\end{enumerate}
We summarise the various (inhomogeneous, limiting and pre-limiting) coalescent processes of interest to us in the following definition.
Recall that $\Psi$ is a Poisson point process on 
$(0,\infty) \times ( \Delta \setminus \{ \bf 0 \})$
with intensity 
$\langle x,x \rangle^{-1} \, \Xi (\dd x)$,
and recall also the definitions of
$\Psi_\varepsilon^{}$, $\Psi^{(N)}$ and $\Psi^{(N)}_\varepsilon$ in 
\eqref{eq:psilimitcutoff}, \eqref{eq:PsiNdef} and
\eqref{eq:psicutoff}.

\begin{defn} \label{def:coalescents}
Conditionally on $\Psi$, $\Psi^{(N)}$ and with $c_{\mathsf{pair}}^{}$ given in \eqref{eq:cpair}, we write 
$\Pi = (\Pi_t)_{t \geqslant 0}$ for the inhomogeneous
$(\Psi,c_{\mathsf{pair}}^{})$-coalescent in the sense of Definition~\ref{def:inhomcoalescentfullygeneral}. Additionally, we write $\Pi^\varepsilon$ for the inhomogeneous
$(\Psi_\varepsilon,c_{\mathsf{pair}}^{})$-coalescent and 
$\Pi^{(N,\varepsilon)}$ for the \emph{(prelimiting) $\varepsilon$-naive coalescent}, which we define as the inhomogeneous $(\Psi^{(N)}_\varepsilon,\gamma_N^{})$-coalescent in the sense of Definition~\ref{def:inhomcoalescentfullygeneral} with $\gamma_N^{}$ being a rescaled geometric distribution with success probability $c_N^{} c_{\mathsf{pair}}^{}$, i.e.
$\gamma_N^{} (\{c_N k\}) =
c_N^{} c_{\mathsf{pair}}^{} 
(1 -  c_N^{} c_{\mathsf{pair}}^{}  )^{k-1}$ for $k \in \N$. We assume that 
$\Pi$ and $\Pi^\varepsilon$ are defined on the probability space 
$(\Omega,\cF,{\bf P})$ while  $\Pi^{(N)}$ and
$\Pi^{(N,\varepsilon)}$ are defined on
$(\Omega^{(N)},\cF^{(N)},\P^{(N)})$.
\end{defn}

\subsection{Continuity of inhomogeneous coalescent, proof of Theorem~\ref{thm:fddconvergence}}
\label{sec:inhomcoal}

In this section, we shall accomplish Steps 3 and 4 
in the roadmap in Subsection~\ref{subsec:roadmap}.
These steps will follow once we prove the continuity of the finite-dimensional distributions of the inhomogeneous
$(\psi,\gamma)$-coalescent  in  $(\psi,\gamma)$.

To prove this continuity, we let
\begin{equation*}
\begin{split}
\cN \defeq \bigg \{
&\sum_{i = 1}^k \delta_{(t^{(i)},x^{(i)})} 
:
k \in \N \cup \{\infty\}, \, t^{(i)} \in (0,\infty) 
\textnormal{ pairwise distinct},\\
 & \, x^{(i)} \in \Delta \setminus \{ {\bf 0} \}
\textnormal{ for all } i < k + 1
\bigg \}
\end{split}
\end{equation*}
be the set of simple point measures on 
$(0,\infty) \times \big ( \Delta \setminus \{ \bf 0 \}   \big )$, with either finitely or infinitely many atoms; we use the convention that $\infty + 1 = \infty$. Recall that
$\cN_{\mathsf{loc}}^2$ denotes set of those $\psi \in \cN$ that additionally satisfy \eqref{eq:l2cond}.


In order to talk about convergence, we need to equip $\cN^2_{\mathsf{loc}}$ with an appropriate topology.
\begin{defn} \label{def:N2loctopology}
Let $(\psi_m)_{m \in \N}$ be a sequence in $\cN^{2}_{\mathsf{loc}}$. We say that $\psi_m \stackrel{m \to \infty}{\longrightarrow} \psi \in \cN^2_{\mathsf{loc}}$ if 
\begin{equation*}
\int_{[0,T)\times \Delta}
f(t,x) \langle x,x \rangle \psi_m^{} (t,x)
\stackrel{m \to \infty}{\longrightarrow} 
\int_{[0,T)\times \Delta}
f(t,x) \langle x, x \rangle \psi(t,x)
\end{equation*}
for all continuous and bounded functions
$f : [0,\infty) \times \Delta$ and all $T >0$.
\end{defn}

\begin{remark}
There are different ways to interpret the convergence in Definition~\ref{def:N2loctopology}; first, one can understand this in terms of local-in-time weak convergence with respect to those bounded continuous functions $[0,\infty) \times \Delta$ that decay at least as rapidly as $\langle x, x \rangle$ as $x$ approaches ${\bf 0} \in \Delta$. 
More precisely, such functions can be seen as those bounded and continuous functions $f$ on 
$[0,\infty) \times \Delta \setminus \{ \bf 0 \}$
so that
\begin{equation*}
(t,x) \mapsto \frac{f(t,x)}{\langle x,x \rangle}
\end{equation*}
has a continuous and bounded extension to
$[0,\infty) \times \Delta$; this will be the viewpoint we will take in the proof of Lemma~\ref{lem:inhomcoalescentfddcontinuity}, while the original formulation in Definition~\ref{def:N2loctopology} is particularly well-suited to proving the convergence of $\Psi_\varepsilon$ (see \eqref{eq:psilimitcutoff}) to $\Psi$ as $\varepsilon \to 0$.
\end{remark}


We now show that the finite-dimensional marginals of
$\Pi^{\psi,\gamma}$ are continuous in  $\psi$ and $\gamma$.

\begin{lemma}
\label{lem:inhomcoalescentfddcontinuity}
Let $0 < t_1^{} < \ldots < t_k^{} < \infty$ and 
$\xi_0^{},\xi_1^{},\ldots,\xi_k^{} \in \cE_n$. 
For all $m \in \N$, let $\gamma_m^{}$ be a probability measure on $(0,\infty)$ such that 
$\gamma_m^{}$ converges to $\gamma$ in the usual weak topology
and let
$\psi_m^{} \in \cN_{\mathsf{loc}}^2$ such that $\psi_m^{}$
converges to $\psi$ in the sense of Definition \ref{def:N2loctopology}. Assume moreover that
$\psi \big (\{t_j^{} \} \times (\Delta
\setminus \{ \mathbf 0 \})\big ) = 0$ for all 
$1 \leqslant j \leqslant k$ and that
$\gamma$ has no pure point part, i.e.
$\gamma(\{ t \}) = 0$ for all $t > 0$. 
Then,
\begin{equation*}
\begin{split}
&{\bf P}\big (\Pi_{t_1^{}}^{\psi_m^{},\gamma_m^{}} = \xi_1^{},
\ldots,
\Pi_{t_k^{}}^{\psi_m^{},\gamma_m^{}} = \xi_k^{}
\, | \,
\Pi_0^{\psi_m^{},\gamma_m^{}} = \xi_0^{}
\big ) \\
& \quad \stackrel{m \to \infty}{\longrightarrow}
{\bf P}\big (\Pi_{t_1^{}}^{\psi,\gamma} = \xi_1^{},
\ldots,
\Pi_{t_k^{}}^{\psi,\gamma} = \xi_k^{}
\, | \,
\Pi_0^{\psi,\gamma} = \xi_0^{}
\big ).
\end{split}
\end{equation*}
\end{lemma}
Recall the definition of the random point processes $\cC_\psi$ and $\cK_\gamma^{i,j}$ from 
Section~\ref{subsec:coagulators}.
Note that when $\gamma$ has a pure point part, there is a positive probability that
$\cK_\gamma^{i,j}$ will contain atoms at the same time for multiple $i,j$. When this happens, the finite-dimensional distributions of the coalescent are in fact \textit{not} continuous in $\cC_\psi$ and $\cK_\gamma$. This is because, in Definition~\ref{def:inhomcoalescentfullygeneral}, whenever $\cC_\psi \cup \cK_\gamma$ has two atoms at the same time, all except one are ignored in the definition of the stochastic flow. Consequently, atoms that are vanishingly close to each other in time affect the finite-dimensional distributions differently than those whose time coordinate agrees exactly.

\begin{proof}[Proof of Lemma~\ref{lem:inhomcoalescentfddcontinuity}]
We claim that as $m \to \infty$, $\cC_{\psi_m^{}}$ and 
$\cK_{\gamma_m^{}}^{i,j}$ converge to
$\cC_\psi$ and $\cK_\gamma^{i,j}$ in distribution with respect to the vague topology on $(0,\infty) \times \cE_n$. Assuming this for now, we may further assume by Skorokhod's representation theorem that this convergence holds almost surely. If we denote the atoms (in chronological order) of $\cC_\psi^{}$ by
$(t^{(1)},\alpha^{(1)}),(t^{(2)},\alpha^{(2)}),\ldots$ and those of 
$\cC_{\psi_m^{}}$ by $(t^{(1,m)},\alpha^{(1,m)}),(t^{(2,m)},\alpha^{(2,m)}),\ldots$, we have 
\begin{equation*}
(t^{(i,m)},\alpha^{(i,m)}) \stackrel{m \to \infty}{\longrightarrow}
(t^{(i)},\alpha^{(i)}).
\end{equation*}
Because we equipped $\cE_n$ with the discrete topology,
$\alpha^{(i,m)} \stackrel{m \to \infty}{\longrightarrow} 
\alpha^{(i)}$ means that
$\alpha^{(i,m)} = \alpha^{(i)}$ for sufficiently large $m$. 
By our assumptions, $t^{(1)},t^{(2)},\ldots$ are all distinct from $t_1^{},t_2^{},\ldots$ and from each other. Similarly for 
$\cK_\gamma^{i,j}$ and $\cK_\gamma$
so that we have 
$\Pi_{t_j^{}}^{\psi_m^{},\gamma_m^{}} =
\Pi_{t_j^{}}^{\psi,\gamma}
$
for all $j \in \{1,\ldots,k\}$ when $m$ is sufficiently large.

It remains to verify the claimed convergence of $\cC_{\psi_m^{}}$ and $\cK_{\gamma_m^{}}^{i,j}$.
We start with $\cC_{\psi_m^{}}$.
By~\cite[Theorem 1.1]{Kallenberg1973} it suffices to show for each continuous and compactly supported function 
$f : (0,\infty) \times \cE_n \to [0,\infty)$
that $\cC_{\psi_m^{}} (f)$ converges in distribution
to $\cC_\psi (f)$. More explicitly, for a continuous $f$ supported on
$(0,T) \times \cE_n$,
we need to show that
\begin{equation} \label{eq:distrconvergence}
\sum_{i \, : \, 0 < t^{(i,m)} < T} 
f(t^{(i,m)},\alpha_{t^{(i,m)},x^{(i,m)}}^{}) 
\stackrel{m \to \infty}{\longrightarrow}
\sum_{i \, : \, 0 < t^{(i)} < T} 
f(t^{(i)},\alpha_{t^{(i)},x^{(i)}}) 
\end{equation}
in distribution. To show this, it is enough, by the non-negativity of $f$, to prove convergence of the Laplace transform. Fix $\lambda > 0$ 
and an enumeration $(t_{(1)}^{},x_{(1)}^{}),
(t_{(2)}^{},x_{(2)}^{}),\ldots$ of the atoms of $\psi$ contained in $(0,T) \times (\Delta \setminus 
\{ \mathbf 0  \})$. 

Because the $\alpha_{t_{(i)}^{},x_{(i)}^{}}^{}$ are independent, the Laplace transform of the right-hand side can be written as
\begin{equation*}
\begin{split}
&\E \Big [
\exp \Big ( 
-\lambda 
\sum_{i \, : \, 0 < t^{(i)} < T} 
f(t^{(i)},\alpha_{t^{(i)},x^{(i)}}^{}) 
\Big ) 
\Big ] \\
& \quad =
\prod_{i \, : \, 0 < t_{(i)}^{} < T}
\E \Big [
\exp \Big ( 
-\lambda 
f(t_{(i)}^{},\alpha_{t_{(i)}^{},x_{(i)}^{}}^{})
\1_{\alpha_{t_{(i)}^{},x_{(i)}^{}}^{} \neq \xi_0^{(n)}}
\Big ) 
\Big ] \\
& \quad =
\prod_{i \, : \, 0 < t_{(i)}^{} < T}
\Big ( p \big (x_{(i)}^{};\xi_0^{(n)},\xi_0^{(n)} \big ) +
\sum_{\alpha \in \cE_n 
\setminus \{ \xi_0^{(n)} \}}
p(x_{(i)}^{};\xi_0^{(n)},\alpha) 
\exp \big ( 
-\lambda 
f(t_{(i)}^{},\alpha)
\big )
\Big ) \\
& \quad =
\exp \bigg (\,
\sum_{i \, : \, 0 < t_{(i)}^{} < T}
\log \Big ( p \big (x_{(i)}^{};\xi_0^{(n)},\xi_0^{(n)} \big ) \\
& \quad \qquad\qquad + \sum_{\alpha \in \cE_n 
\setminus \{ \xi_0^{(n)} \}}
p(x_{(i)}^{};\xi_0^{(n)},\alpha) 
\exp \big ( 
-\lambda 
f(t_{(i)}^{},\alpha)
\big )
\Big )
\bigg ) \\
& \quad = 
\exp \big (   
\psi_{|(0,T)} ( 
G_{\lambda,f}
)
\big )
\end{split}
\end{equation*}
where $G_{\lambda,f}$ is a real-valued function on
$(0,\infty) \times 
(\Delta \setminus \{ \mathbf 0  \})
$, defined via
\begin{equation*}
G_{\lambda,f} (t,x) \defeq 
\log \Big ( p \big (x;\xi_0^{(n)},\xi_0^{(n)} \big ) +
\sum_{\alpha \in \cE_n \setminus \{ \xi_0^{(n)} \}}
p(x;\xi_0^{(n)},\alpha) 
\exp \big ( 
-\lambda 
f(t,\alpha)
\big ) 
\Big )
\end{equation*}
and $\psi_{|(0,T)}$ is obtained from $\psi$ by erasing the atoms at $(t,x)$ with $t \geqslant T$.
Similarly, we can write the Laplace transform of the left-hand side as 
$\exp \big ( \psi_{m|(0,T)}^{} (G_{\lambda,f} )  \big )$. The convergence follows because of our assumption that $\psi$ is the limit of $\psi_m^{}$ in the sense of Definition \ref{def:N2loctopology} 
(which readily implies that $\psi_{|(0,T)}^{}$ is the limit of $\psi_{m|(0,T)}^{}$ in the same sense),
and because $(t,x) \mapsto G_{\lambda,f}(t,x)/\langle x,x \rangle$ is bounded and continuous, and has a continuous extension to $[0,\infty) \times \Delta$.

To see this, note that the uniform bound 
$p(x;\xi_0^{(n)},\alpha) \leqslant K\langle x, x \rangle$ (see the discussion following \eqref{eq:noncoal.psi.s<t.0}) and the series expansion of the logarithm imply that
\begin{equation*}
G_{\lambda, f}(t,x) = -\sum_{\alpha \in \cE_n \setminus \{ \xi_0^{(n)} \}}p(x;\xi_0^{(n)},\alpha)
\big (1 - \ee^{-\lambda f(t,\alpha)}  \big ) + O(\langle x, x \rangle^2). 
\end{equation*}
Finally, if $\alpha \in \cE_n \setminus \{ \xi_0^{(n)} \}$ contains one block with two elements and only singletons otherwise, 
$p(x;\xi_0^{(n)},\alpha) = \langle x, x \rangle + o(\langle x,x\rangle )$; in all other cases,
$p(x;\xi_0^{(n)},\alpha) = o(\langle x, x \rangle)$.

\end{proof}

\begin{remark}
In addition to defining the inhomogeneous $(\psi,\gamma)$-coalescent as a stochastic process for general $\psi$ 
and $\gamma$, 
Definition~\ref{def:inhomcoalescentfullygeneral} does slightly more. Namely, letting 
\begin{equation*}
F^{\psi,\gamma}_{s,t} (\xi) \defeq 
\prod_{
\substack{
(u, \alpha) \in \cC_\psi \cup \cK_c \\
s < u \leqslant t
}
}
\mathsf{Coag}_\alpha (\xi),
\end{equation*}
for 
$s,t$ with $0 \leqslant s \leqslant t < \infty$ and $\xi \in \cE_n$,
we obtain a \emph{stochastic flow} for the inhomogeneous
$(\psi,\gamma)$-coalescent in the following sense. The random maps
$F^{\psi,\gamma}_{s,t} : \cE_n \to \cE_n$ almost surely satisfy 
$F_{s,t}^{\psi,\gamma} =
F_{u,t}^{\psi,\gamma} \circ F_{s,u}^{\psi,\gamma}$ for all $s,u$ and $t$ simultaneously. Moreover,
$F^{\psi,\gamma}_{s,s} = \mathsf{id}_{\cE_n}$ almost surely and
$F^{\psi,\gamma}_{0,\cdot} = \Pi^{\psi,\gamma}$.
\hfill \mbox{$\Diamond$}
\end{remark}

\begin{remark}
\label{rmk:jumpcharacterisation}
As a jump-hold Markov chain, we can characterise the $(\psi,c)$-coalescent (or, more generally, the $(\psi,\gamma)$-coalescent) via the distribution of its first jump, i.e. the joint distribution of the time at which the first jump after any given time $t \geqslant 0$ takes place and the state reached by that jump. See
Appendix~\ref{app:jumphold} for details.
\hfill \mbox{$\Diamond$}
\end{remark}

\paragraph*{Transition probabilities of the naive coalescent conditional on $\Psi_\varepsilon^{(N)}$}
In addition to the above definition of the $\varepsilon$-naive coalescent $\Pi^{(N,\varepsilon)}$ (see Definition~\ref{def:coalescents}) via coalescent flows, we will later need a more explicit description of its (conditional) transition probabilities as a Markov chain indexed by $c_N^{} \N$. 

Assume that for some $t = g c_N^{}$ with $g \in \N$, 
we have $\Pi^{(N,\varepsilon)}_t = \xi \in \cE_n$. We distinguish two cases and assume $\xi \neq \eta$.
\begin{itemize}
\item 
$\Psi^{(N)}_\varepsilon$ does not have an atom at 
time $(g+1) c_N^{}$, i.e.
$(t + c_N^{}, x) \not \in \Psi^{(N)}_\varepsilon$
for all 
$x \in \Delta \setminus \{ {\bf 0} \}$. Then, almost surely, $\cC_\psi$ in Definition~\ref{def:inhomcoalescentfullygeneral} also has no atom of the form 
$(t + c_N^{}, \alpha)$ for any $\alpha \in \cE_n$. For any pair $1 \leqslant i < j \leqslant n$, $\cK_{i,j}^{}$ has an atom at 
$g + c_N^{}$ with probability $c_N^{} c_{\mathsf{pair}}^{}$, while the probability that there are two or more such pairs is in
$O(c_N^2)$. 
Thus, Definition~\ref{def:inhomcoalescentfullygeneral} implies that
\begin{equation*}
\P^{(N)} \big ( 
\Pi^{(N,\varepsilon)}_{t + c_N^{}} = \eta \, | \, 
\Pi^{(N,\varepsilon)}_t = \xi
\big )
= c_N^{} c_{\mathsf{pair}}^{} \1_{\xi \vdash \eta} + O(c_N^2)
\end{equation*}
for $\eta \in \cE_n$, where $\xi \vdash \eta$ means that $\eta$ arises from $\xi$ by a pair coalescence.
\item 
On the other hand, if $(t + c_N^{},x) \in \Psi_\varepsilon^{(N)}$
for some 
$x \in \Delta \setminus \{ {\bf 0}  \}$,
we have 
\begin{equation*}
\P^{(N)} \big (
\Pi^{(N,\varepsilon)}_{t + c_N^{}} = \eta \, | \,
\Pi_t^{(N,\varepsilon)} = \xi
\big ) =
p(x;\xi,\eta).
\end{equation*}
\item 
Letting 
$\widetilde V \defeq \widetilde V^{(g+1)}$, recall that $\Psi^{(N)}_{\varepsilon}$ has an atom at time $(g+1)c_N^{}$ precisely when 
$\| \widetilde V \| \geqslant \varepsilon$.
We can therefore summarise the two cases to see that
\begin{equation*}
\P^{(N)} \big (
\Pi^{(N,\varepsilon)}_{t + c_N^{}} = \eta \, | \,
\Pi_t^{(N,\varepsilon)} = \xi
\big )
=
p_\varepsilon^{} (\widetilde V; \xi, \eta)
\end{equation*}
with
\begin{equation} \label{eq:pepsilon}
p_\varepsilon^{} (\widetilde V; \xi, \eta) 
\defeq
\1_{\| \widetilde V \| < \varepsilon} \1_{ \xi \vdash \eta} \, c_N^{} c_{\textnormal{pair}}^{}
+ \1_{\| \widetilde V  \| \geqslant \varepsilon}\, p(\widetilde V;  \xi, \eta) + O(c_N^2),
\end{equation}
where the constant implied by $O(c_N^2)$ is uniform in $\widetilde V$ and $p$ is given in 
\eqref{eq:paintboxtransitionprime}. Recall also the definition of $c_N^{}$ in \eqref{eq:cN1}.
\end{itemize}

\begin{remark}
The reader should bear in mind that the laws of the processes
$\Pi,\Pi^\varepsilon$ and $\Pi^{(N,\varepsilon)}$ depend on the point processes $\Psi$, $\Psi_\varepsilon$ and $\Psi^{(N)}_\varepsilon$, as well as the asymptotic pair coalescence rate $c_{\mathsf{pair}}^{}$ and the geometric distribution $\gamma_N^{}$. Thus, to be perfectly precise, we would have to denote them by $\Pi^{c_{\mathsf{pair}}^{},\Psi}$,
$\Pi^{c_{\mathsf{pair}}^{},\Psi_\varepsilon}$ and
$\Pi^{\gamma_N^{},\Psi_\varepsilon^{(N)}}$. We refrain from doing this to avoid excess notational clutter.
\hfill \mbox{$\Diamond$}
\end{remark}

Observe that $\Psi_\varepsilon^{(N)}$ converges weakly in distribution to
$\Psi_\varepsilon^{}$ as $N \to \infty$ for any fixed $\varepsilon > 0$. Moreover, as $\varepsilon \to 0$,
$\Psi_\varepsilon^{}$ converges weakly in probability to $\Psi$. Indeed, since $\Psi_\varepsilon^{}$ was defined as the cutoff of $\Psi$, this convergence even holds almost surely by dominated convergence. Together with the convergence of the geometric distribution to the exponential distribution, the following is an immediate consequence of Lemma~\ref{lem:inhomcoalescentfddcontinuity}.

\begin{lemma}[Steps 3 and 4 in the roadmap in Section~\ref{subsec:roadmap}]
\label{lem:naivetolimit}
Let $\Psi^{(N)}_\varepsilon$ and $\Psi$ be as above. Then, for $t_1^{}, \ldots, t_k^{} \geqslant 0$ and 
$\xi_1^{},\ldots,\xi_k^{} \in \cE_n$, the following convergence holds in distribution.
\begin{equation*}
\lim_{\varepsilon \to 0} \lim_{N \to \infty}
\P^{(N)} \big ( \Pi_{t_1^{}}^{(N,\varepsilon)} = \xi_1^{},
\ldots,
\Pi_{t_k^{}}^{(N,\varepsilon)} = \xi_k^{}
\, \big | \, \Psi_{\varepsilon}^{(N)}          
\big )
=
{\bf P} \big (                                                
\Pi_{t_1^{}}^{} = \xi_1^{},
\ldots,
\Pi_{t_k^{}}^{} = \xi_k^{}
\, \big | \, \Psi 
\big ).
\end{equation*}
\end{lemma}

\begin{proof}
Following the same argument as in the paragraph containing \eqref{eq:PhiNconv2}, we see that $\Psi^{(N)}_\varepsilon \to \Psi_\varepsilon^{}$ as $N\to \infty$ in distribution, with respect to the vague topology. 
From this, convergence in the sense of Definition~\ref{def:N2loctopology} follows immediately since neither
$\Psi^{(N)}_\varepsilon$ nor $\Psi_\varepsilon$ have any atoms in $[0,\infty) \times B_\varepsilon(\bf 0)$. Since $\Psi_\varepsilon$ has almost surely no atoms at times
$t_1^{},\ldots,t_k^{}$ and because $\gamma_N^{}$ converges weakly to an exponential distribution with rate 
$c_{\textnormal{pair}}^{}$, 
Lemma \ref{lem:inhomcoalescentfddcontinuity} implies that
\begin{equation*}
 \lim_{N \to \infty}
\P^{(N)} \big ( \Pi_{t_1^{}}^{(N,\varepsilon)} = \xi_1^{},
\ldots,
\Pi_{t_k^{}}^{(N,\varepsilon)} = \xi_k^{}
\, \big | \, \Psi_{\varepsilon}^{(N)}          
\big )
=
{\bf P} \big (                                                
\Pi_{t_1^{}}^{\varepsilon} = \xi_1^{},
\ldots,
\Pi_{t_k^{}}^{\varepsilon} = \xi_k^{}
\, \big | \, \Psi_\varepsilon^{} 
\big )
\end{equation*}
in distribution for all $\varepsilon > 0$. 

Since we defined $\Psi_\varepsilon$ by ignoring small atoms in $\Psi$ (see Section~\ref{subsec:roadmap}),  for any $T > 0$,
\begin{equation*}
\int_{[0,T)\times \Delta \setminus\{\bf 0\}}
f(t,x) \langle x,x \rangle \Psi_\varepsilon(\dd t, \dd x) = 
\int_{[0,T)\times \Delta \setminus\{\bf 0\}} f(t,x) \1_{\|x\| \geqslant \varepsilon} 
\langle x,x \rangle
\Psi (\dd t, \dd x).
\end{equation*}
Because $f$ is bounded,
the fact that $\Psi \in \cN^2_{\mathsf{loc}}$ almost surely implies, by dominated convergence, that the right-hand side converges for $\varepsilon \to 0$ to
\begin{equation*}
\int_{[0,T)\times \Delta \setminus\{\bf 0\}} f(t,x)
\langle x,x \rangle
\Psi (\dd t, \dd x),
\end{equation*}
meaning that $\Psi_\varepsilon \stackrel{\varepsilon \to 0}{\longrightarrow} \Psi$ in the sense of Definition \ref{def:N2loctopology}.
Again, the application of Lemma \ref{lem:inhomcoalescentfddcontinuity} yields 
\begin{equation*}
\lim_{\varepsilon \to 0}{\bf P} \big (             \Pi_{t_1^{}}^{\varepsilon} = \xi_1^{},
\ldots,
\Pi_{t_k^{}}^{\varepsilon} = \xi_k^{}
\, \big | \, \Psi_\varepsilon^{} 
\big ) 
=
{\bf P} \big (                                     
\Pi_{t_1^{}}^{} = \xi_1^{},
\ldots,
\Pi_{t_k^{}}^{} = \xi_k^{}
\, \big | \, \Psi^{} 
\big ) .
\end{equation*}

\end{proof}

Recall that $\Pi^N$ denotes the `true' prelimiting ancestral process that depends both on the pedigree $P^{(N)}$ and the initial choice of the sampled genes; this information is captured by the 
$\sigma$-algebra $\cA^{(N)}$. 
In Section~\ref{sec:paircouplings}, we will carry out the fifth and final step of the roadmap in Subsection~\ref{subsec:roadmap}, proving the following

\begin{lemma}[Step 5 in the roadmap in Section~\ref{subsec:roadmap}]
\label{lem:l2approximation}
For all $0 < t_1^{} < \ldots < t_k^{} < \infty$ and 
$\xi_1^{},\ldots,\xi_k^{} \in \cE_n$, write
$\widetilde t_i^{} \defeq \lfloor t_i^{} / c_N^{} \rfloor$. Then,
\begin{equation}
\label{eq:naiveapproximation}
\begin{split}
& \limsup_{N \to \infty} \,
\E^{(N)} \Big [
\Big (
\P^{(N)} \big ( \Pi^{N}_{\widetilde t_1^{}} = \xi_1^{}, \ldots, 
\Pi^{N}_{\widetilde t_k^{}} = \xi_k^{} \, \big | \, \cA^{(N)} \big ) \\
& \qquad\qquad\qquad\quad -
\P^{(N)} \big ( \Pi_{t_1^{}}^{(N,\varepsilon)} = \xi_1^{},
\ldots,
\Pi_{t_k^{}}^{(N,\varepsilon)} = \xi_k^{} \big )
         \, \big | \, \Psi^{(N)}_\varepsilon  \big )
\Big )^2 
\Big ]
\stackrel{\varepsilon \to 0}{\longrightarrow} 
0.
\end{split}
\end{equation}
\end{lemma}

Before we prove this, let us close this section by showing how the last two lemmas imply the proof of our main result.

\begin{proof}[Proof of Theorem~\ref{thm:fddconvergence}]
We fix $0 < t_1^{} < \ldots < t_k^{} < \infty$ and
$\xi_1^{},\ldots,\xi_k^{} \in \cE_n$
and use the following abbreviations:
\begin{equation*}
\cX^{N} \defeq 
\P^{(N)} \big (  \Pi_{\widetilde t_1^{}}^{N} = \xi_1^{},
\ldots, \Pi_{\widetilde t_k^{}}^{N} = \xi_k^{}
\, \big | \,
\cA^{(N)}
\big ),
\end{equation*}
\begin{equation*}
\cY^{(N,\varepsilon)} \defeq 
\P^{(N)} \big (   
\Pi_{t_1^{}}^{(N,\varepsilon)} = \xi_1^{}, \ldots, 
\Pi_{t_k^{}}^{(N,\varepsilon)} = \xi_k^{}
\, \big | \,
\Psi^{(N)}_{\varepsilon}
\big )
\end{equation*}
and
\begin{equation*}
\cZ \defeq
{\bf P} \big ( \Pi^{}_{t_1^{}} = \xi_1^{},
\ldots,
\Pi^{}_{t_k^{}} = \xi_k^{} 
\, \big | \, \Psi
\big )
\end{equation*}
where $\widetilde t = \lfloor t / c_N^{} \rfloor$ denotes the usual time rescaling. 

Denote by $d_P^{}$ the Lévy-Prokhorov metric on the set of Borel probability measures on $[0,1]$. 
Lemma~\ref{lem:naivetolimit} means that 
\begin{equation*}
\lim_{\varepsilon \to 0} \limsup_{N \to \infty} 
d_P^{} \big ( \cL(\cY^{(N,\varepsilon)}), \cL(\cZ)    \big )
=
0,
\end{equation*}
while Lemma~\ref{lem:l2approximation} implies that
\begin{equation*}
\lim_{\varepsilon \to 0} \limsup_{N \to \infty} 
d_P^{} \big ( \cL(\cY^{(N,\varepsilon)}), \cL(\cX^{N})    \big )
=
0.
\end{equation*}
Indeed, for any fixed $\delta > 0$, $\eqref{eq:naiveapproximation}$ implies that, for sufficiently large $N$ and sufficiently small~$\varepsilon$, 
\begin{equation*}
\P^{(N)} \big (
|\cX^N - \cY^{(N,\varepsilon)}| \geqslant \delta 
\big ) \leqslant \delta.
\end{equation*}
Consequently, for all Borel-measurable $A \subseteq [0,1]$, we have that 
\begin{equation*}
\P^{(N)} (\cY^{(N,\varepsilon)} \in A) \leqslant \P^{(N)}(\cX^N \in A^\delta) + \delta 
\end{equation*} and
\begin{equation*}
\P^{(N)} (\cX^N \in A) \leqslant \P^{(N)}(\cY^{(N,\varepsilon)} \in A^\delta) + \delta,
\end{equation*}
hence $d_P^{} \big ( \cL(\cY^{(N,\varepsilon)}), \cL(\cX^{N})    \big )
\leqslant \delta$.
The triangle inequality yields
\begin{equation*}
\begin{split}
&\limsup_{N \to \infty}
d_P^{} \big ( 
\cL(\cX^{N}), \cL(\cZ)   \big ) \\
& \quad \leqslant
\lim_{\varepsilon \to 0} \limsup_{N \to \infty} 
d_P^{} \big ( \cL(\cY^{(N,\varepsilon)}), \cL(\cZ)    \big )
+
\lim_{\varepsilon \to 0} \limsup_{N \to \infty} 
d_P^{} \big ( \cL(\cY^{(N,\varepsilon)}), \cL(\cX^{N})    \big )
=
0
\end{split}
\end{equation*}
\end{proof}

\subsection{Coupling two coalescents in the same pedigree}
\label{sec:paircouplings}

To complete the proof of Theorem~\ref{thm:fddconvergence}, we still need to show that the $\varepsilon$-naive coalescent
$\Pi^{(N,\varepsilon)}$ is indeed a good approximation of the rescaled prelimiting ancestral process 
$\Pi^{N}_{\lfloor c_N^{-1} \cdot  \rfloor}$ as $\varepsilon\to 0$,
i.e. prove Lemma~\ref{lem:l2approximation}.
Evaluating the squared $L^2$-distance of their conditional finite-dimensional distributions given $\cA^{(N)}$ comes down to 
evaluating the expectations of products of conditional probabilities. Generally, for any two events $A$ and $B$ that are conditionally independent given $\cA^{(N)}$, the definition of conditional independence and the law of total probability implies that
\begin{equation} \label{quenchedtoannealed}
\begin{split}
&\E^{(N)} \big [ 
\P^{(N)} \big ( A \, \big | \, \cA^{(N)}  \big )   
\P^{(N)} \big ( B \, \big | \, \cA^{(N)}  \big )           
\big ] \\
& \quad =
\E^{(N)} \big [ \P^{(N)} \big (
A \cap B \, \big | \, \cA^{(N)}     \big ) \big ]
=
\P^{(N)} \big (A \cap B   \big ).
\end{split}
\end{equation}
Since our goal is to compare the
finite-dimensional distributions of
$\Pi^{N}_{\lfloor c_N^{-1} \cdot \rfloor}$ and $\Pi^{(N,\varepsilon)}$, we want to apply 
this 
to events of the form $\{ \cX_{t_1^{}} = \xi_1^{}, \ldots,
\cX_{t_k^{}} = \xi_k^{} 
\}$ with 
$\cX \in 
\big \{ \Pi^{N}_{\lfloor c_N^{-1} \cdot \rfloor},
\Pi^{(N,\varepsilon))}  \big \}$, 
\mbox{$t_1^{},\ldots,t_k^{} \in \R_{> 0}$} 
and $\xi_1^{},\ldots,\xi_k^{} \in \cE_n$.
Hence, our first task is to couple each choice of
$\cX, \cY \in \big \{\Pi^{N},\Pi^{(N,\varepsilon)}   \big \}$ so that they evolve conditionally 
independently with respect to $\cA^{(N)}$. We will then  
investigate the asymptotics of the \emph{annealed} finite-dimensional distributions, marginal to the pedigree, of each of these couplings. 
Keep in mind that the naive coalescent 
$\Pi^{(N,\varepsilon)}$ depends on  $\Psi_\varepsilon^{(N)}$, which in turn is a function of  $\cA^{(N)}$. 


More specifically, $\cX$ and $\cY$ will use the same realisation of the pedigree $P^{(N)}$, and consequently, in the case of the naive coalescent, the same realisation of $\Psi_{\varepsilon}^{(N)}$. However, they will use two independent realisations 
of the Mendelian coin flips $M_{c,i,g}$ from Definition~\ref{Def:M} 
(or independent realisations of the random point processes
$\cC_{\Psi^{(N)}_\varepsilon}$ and $\cK_{\gamma_N^{}}$
introduced in Section~\ref{sec:inhomcoal}). 
We will use $\widehat \cdot$ and $\widecheck \cdot$ to distinguish these two conditionally independent copies, writing
\begin{itemize}
\item 
$\big (
\widehat \Pi^N, \widecheck \Pi^N 
\big )^{\mathsf{avg}}$ 
for two conditionally independent realisations of $\Pi^N_{\lfloor c_N^{-1}\cdot \rfloor}$,
driven by independent realisations $(\widehat M_{c,i,g})$ and $(\widecheck M_{c,i,g})$ of the Mendelian coin flips.
\item 
$\big (
\widehat \Pi^N, \widecheck \Pi^{(N,\varepsilon)}
\big )^{\mathsf{avg}}$ for 
conditionally independent realisations of  $\Pi^N_{\lfloor c_N^{-1} \cdot \rfloor}$ and $\Pi^{(N,\varepsilon)}$, 
driven respectively by Mendelian coin flips $(M_{c,i,g})$ and point processes $\cC_{\Psi_\varepsilon^{(N)}}, \cK_{\gamma_N^{}}$ independent of $(M_{c,ig})$. 
\item 
$\big (
\widehat \Pi^{(N,\varepsilon)}, \widecheck \Pi^{(N,\varepsilon)}
\big )^{\mathsf{avg}}$ for two conditionally independent realisations of 
$\Pi^{(N,\varepsilon)}$, driven by  realisations
$\widehat \cC_{\Psi_\varepsilon^{(N)}}$, 
$\widecheck \cC_{\Psi_\varepsilon^{(N)}}$ and
$\widehat \cK_{\gamma_N^{}}$,
$\widecheck \cK_{\gamma_N^{}}$
which are conditionally independent given $\Psi_\varepsilon^{(N)}$.
\end{itemize}
By a slight abuse of notation, we let the time rescaling be implicit in the notation 
$\widehat \Pi^N$ and $\widecheck \Pi^N$ so that all of these processes are indexed by 
$\R_{\geqslant 0}$. Noting that they are constant on $[i c_N^{}, (i+1) c_N^{} )$ for all 
$i \in \N_0$, they can also naturally be thought of as discrete-time Markov chains indexed by
$c_N^{} \N_0$, which is the point of view we will adopt below.

In the following subsection, we provide the transition probabilities of these three Markov chains. Following the same strategy as in~\cite{BirknerEtAl2018} to compute the limit of their annealed (averaged over the random pedigree) finite-dimensional distributions as $N \to \infty$, we will make use of a separation-of timescales-argument from~\cite[Lemma 1 and Theorem 1]{Mohle1998a}.

A subtle point is that, in spite of what the notation might suggest, the state space of
$\big (
\widehat \Pi^N, \widecheck \Pi^{N}
\big )^\mathsf{avg}$ is \emph{not} $\cS_n \times \cS_n$. This is because in order to obtain a Markov process, we need to not only keep track of the grouping of ancestral genes into diploid individuals for both coalescent copies individually (which
was the motivation behind introducing the state space 
$\cS_n$ to make $\Pi^{N}$ Markovian), but we also need to keep track of how ancestral individuals are shared among the two copies. For this, we will recall from~\cite{TyukinThesis2015} an appropriate state space for
$\big ( \widehat \Pi^N, \widecheck \Pi^N  \big )^\mathsf{avg}$. 
In contrast, because the blocks of the naive coalescent $\Pi^{(N,\varepsilon)}$ are not associated with concrete ancestral genes with specific positions within the pedigree, it is straightforward to view 
$\big (
\widehat \Pi^N, \widecheck \Pi^{(N,\varepsilon)} 
\big )^\mathsf{avg}$
and
$\big ( \widehat \Pi^{(N,\varepsilon)}, \widecheck \Pi^{(N,\varepsilon)}  
\big )^\mathsf{avg}$ 
as taking values in $\cS_n \times \cE_n$ and
$\cE_n \times \cE_n$. 

We introduced these couplings primarily as a useful technical tool to control the $L^2$-distance between the conditional
finite-dimensional distributions
of the proper prelimiting ancestral process
$\Pi^N$ and those of its naive approximation
$\Pi^{(N,\varepsilon)}$. However, the coupling 
$\big (
\widehat \Pi^N, \widecheck \Pi^N
\big )^\mathsf{avg}$ also has a clear biological interpretation; it represents the (annealed) evolution of two genealogies associated with two independently assorting loci. 

\subsubsection{State spaces and transition probabilities}
\label{subsec:annealedtransitionmatrices}

We postpone addressing the issue of defining an appropriate state space for
$\big ( \widehat \Pi^N, \widecheck \Pi^N  \big )^{\mathsf{avg}}$ and
start by describing the transition matrices of 
$\big (\widehat \Pi^N, \widecheck \Pi^{(N,\varepsilon)}  \big )^{\mathsf{avg}}$ and
$\big ( \widehat \Pi^{(N,\varepsilon)},
\widecheck \Pi^{(N,\varepsilon)}   \big )^{\mathsf{avg}}$.
We will start by considering their evolution conditional on the pedigree, before taking the average. Hence, we will momentarily drop the superscript $\mathsf{avg}$.
Recall from Subsection~\ref{subsec:transitionsconditionalonoffspringnumbers} that
$\pi^{N,n}(V;\cdot,\cdot)$
given in
\eqref{eq:transitionsconditionedonnumbers}
denotes the 
transition matrix of $\Pi^N$ (and hence of $\widehat \Pi^N$ and $\widecheck \Pi^N$) conditional on the offspring numbers
$V = (V_{i,j})_{1 \leqslant i,j \leqslant N}$. Dropping the sample size, we write $\pi^N$ instead of $\pi^{N,n}$. Similarly, we described the conditional transition probabilities of the naive coalescent $\Pi^{(N,\varepsilon)}$ right after Definition~\ref{def:coalescents}.

Let 
$(\widehat \xi, \widecheck \xi), (\widehat \eta, \widecheck \eta) 
\in \cS_n \times \cE_n$. Assume also that $t \in c_N^{} \N$ and that $V$ is the offspring matrix in (backward) generation 
$t/c_N^{} + 1$ 
corresponding to rescaled time
$t + c_N^{}$. 
Conditional on $V$, we let 
$\widecheck \Pi^{(N,\varepsilon)}$ be conditionally independent of the pedigree 
given
$(V^{(g)})_{g \in \N}$, and $\widehat \Pi^N$ a realisation of the coalescent in the pedigree $P^{(N)}$ as in Section~\ref{subsec:fixedpedigrees}.

Then, averaging over the conditional distribution of $P^{(N)}$ given $V$, we get, according to our discussion surrounding
\eqref{eq:pepsilon},
\begin{equation*}
\begin{split}
&\P^{(N)} \big (
\widehat \Pi^{N}_{t + c_N^{}} = \widehat \eta,
\widecheck \Pi^{(N,\varepsilon)}_{t + c_N^{}} = \widecheck \eta
\, |\,
\widehat \Pi^N_t = \widehat \xi, \widecheck \Pi^{(N,\varepsilon)}_t = 
\widecheck \xi, 
V
\big ) \\
& \quad = \P^{(N)} \big (
\widehat \Pi^N_{t + c_N^{}} = \widehat \eta 
\, | \,
\widehat \Pi^N_t = \widehat \xi, 
V
)
\P^{(N)} \big (
\widecheck \Pi^{(N,\varepsilon)}_{t + c_N^{}} = \widecheck \eta 
\, | \,
\widecheck \Pi^{(N,\varepsilon)}_t = \widecheck \xi, V
) \\
& \quad =
\pi^N (V;\widehat \xi, \widehat \eta)
\P^{(N)} \big (
\widecheck \Pi^{(N,\varepsilon)}_{t + c_N^{}} = \widecheck \eta 
\, | \,
\widecheck \Pi^{(N,\varepsilon)}_t = \widecheck \xi, V
) \\
& \quad = 
\pi^N (V; \widehat \xi, \widehat \eta)
p_\varepsilon^{} 
(\widetilde V; \widecheck \xi, \widecheck \eta)
\end{split}
\end{equation*}
with $p_\varepsilon^{}$ from 
\eqref{eq:pepsilon}
and $\widetilde V = (4N)^{-1} (V_{(1)},V_{(1)},V_{(2)},V_{(2)},\ldots,
V_{(N)},V_{(N)},0,0,\dots  )$.
\ From this we obtain the (annealed)
transition probabilities of 
$(\widehat \Pi^N, \widecheck \Pi^{(N,\varepsilon)})$ by averaging over the distribution of $V$.
\begin{defn}[Mixed coupling] \label{def:mixedpairtransitions}
Let $\varepsilon > 0$ and let $(\widehat \xi, \widecheck \xi), 
(\widehat \eta, \widecheck \eta) \in \cS_n \times \cE_n$. We define
\begin{equation*}
\pi_{\mathsf{mix}}^{N,\varepsilon} \big ( (\widehat \xi, \widecheck \xi),
(\widehat \eta, \widecheck \eta)  \big )
\defeq
\E^{(N)} \Big [ 
\pi^{N} (V; \widehat \xi, \widehat \eta )
p_\varepsilon^{} (\widetilde V; \widecheck \xi, \widecheck \eta)
\Big ]
\end{equation*}
with $\pi^{N} (V; \widehat \xi, \widehat \eta )$ given in 
\eqref{eq:transitionsconditionedonnumbers} 
and
with $p_\varepsilon^{}$ given in \eqref{eq:pepsilon}, and define $\big (\widehat \Pi^N, \widecheck \Pi^{(N,\varepsilon)}  \big )^{\mathsf{avg}}$ as the $\cS_n\times\cE_n$-valued Markov chain with transition matrix $\pi_{\mathsf{mix}}^{N,\varepsilon}$.
\end{defn}

Similarly, we can define the coupling of two conditionally independent
copies of the naive coalescent as a process on $\cE_n \times \cE_n$ with the following transition matrix.

\begin{defn}[Pair of naive coalescents]
\label{def:purelynaivetransitions}
Let $\varepsilon > 0$ and let 
$(\widehat \xi, \widecheck \xi), 
(\widehat \eta, \widecheck \eta) \in \cE_n \times \cE_n$. We define
\begin{equation*}
\pi_{\textnormal{pure}}^{N,\varepsilon} 
\defeq 
\E \Big [
p_\varepsilon^{} ( \widetilde V; 
\widehat \xi, \widehat \eta) 
p_\varepsilon^{} (\widetilde V; 
\widecheck \xi, \widecheck \eta )
\Big ],
\end{equation*}
with $p_\varepsilon^{}$ as in Definition~\ref{def:mixedpairtransitions}, and define $\big ( \widehat \Pi^{(N,\varepsilon)}, \widecheck \Pi^{(N,\varepsilon)}  
\big )^\mathsf{avg}$ as the $\cE_n \times \cE_n$-valued Markov chain with transition matrix $\pi_{\textnormal{pure}}^{N,\varepsilon}$.
\end{defn}

We now tackle the problem of defining an appropriate state space for $\big (\widehat \Pi^N, \widecheck \Pi^N   \big )^{\mathsf{avg}}$. For any set $M$, we write 
$\cP (M)$ for the power set of $M$ and 
$\cP_{1,2} (M) \defeq \big \{ A \in \cP (M): \#A \in \{1,2\} \big \}$ for the set of all subsets of $M$ containing either one or two elements. 
For $C \in \cP([n])^2$ we write 
$\widehat C$ and $\widecheck C$ for its two components, i.e.
$C= (\widehat C, \widecheck C)$. The symbol $\biguplus$ denotes the disjoint union.
The following space $\cH_n$  was originally 
introduced in \cite{TyukinThesis2015}.

\begin{defn}  \label{def:tyukinspace}
For all $n \in \N$, let
\begin{equation*}
\cH_n \defeq 
\bigg \{                 
\xi \subseteq \cP_{1,2} \Big (      
\cP([n])^2 \setminus \big \{(\varnothing, \varnothing)
\big \}
\Big ) : 
\biguplus_{I \in \xi} \biguplus_{C \in I} \dot C = [n]
\textnormal{ for } \cdot = \wedge, \vee
\bigg \}
\end{equation*}
\end{defn}

Let us take some time to think about this definition and consider $\xi \in \cH_n$. Each element $I$ of $\xi$ represents an individual that carries either one or two ancestral genes, for each of the two coalescent copies. The (up to) two chromosomes of individual $I$ carrying ancestral genes are each represented by a pair 
$C = (\widehat C, \widecheck C) \in I$ of subsets of $[n]$. Here, $\widehat C$ is the associated block in the first copy, $\widecheck C$ the associated block in the second copy of the coalescent. Either of $\widehat C$ or $\widecheck C$ may be empty, but not both. 
We hope that the next 
example, taken from~\cite[Example~4.5]{TyukinThesis2015}
will help clarify things.
\begin{example}
Consider $\xi \in \cH_5$, given by
\begin{equation*}
\xi =
\Big \{
\big \{
(\{1,2\}, \varnothing) 
\big \},
\big \{
(\{5\}, \{1,4\}), (\varnothing, \{2\} )
\big \},
\big \{
(\{4\}, \varnothing),(\{3\}, \{3,5\})
\big \}
\Big \}.
\end{equation*}
There are three individuals that carry ancestral material; the first carries the ancestor of samples $1$ and $2$ in $\widehat \Pi^N$. The second carries the ancestor of sample $5$ in $\widehat \Pi^N$ and samples $1$ and $4$ in
$\widecheck \Pi^N$, and also the ancestor of sample 
 $2$ in $\widecheck \Pi^N$. Finally, the third individual carries the ancestor of sample $4$ in
 $\widehat \Pi^N$, and also the ancestor of sample $3$ in $\widehat \Pi^N$ and samples $3$ and $5$ in $\widecheck \Pi^N$.
\end{example}

It is clear that for $\cdot \in \{ \wedge, \vee \}$ and 
$\xi \in \cH_n$,
\begin{equation*}
\dot \xi \defeq 
\bigcup_{I \in \xi}\Big  \{
\bigcup_{C \in I} \{ \dot C \} \setminus \{ \varnothing \}        
\Big \} \setminus \Big \{ \varnothing  \Big \}
\end{equation*}
is an element of $\cS_n$ with
\begin{equation} \label{eq:dispersalfromxiinH}
\mathsf{cd} (\dot \xi) = \bigcup_{I \in \xi} \bigcup_{C \in I}
\Big \{
\{ \dot C \}
\Big \} \setminus \Big \{ \{ \varnothing \}   \Big \}.
\end{equation}
Conversely, we can embed $\cS_n \times \cS_n$ (and thus, also 
$\cS_n \times \cE_n$ and $\cE_n \times \cE_n$)
into $\cH_n$. Given
\begin{equation} \label{eq:projectionfromHn1}
\cS_n \ni \widehat \xi = 
\big \{
\{\widehat C_1, \widehat C_2\}, \ldots, 
\{\widehat C_{2x-1}, \widehat C_{2x}    \}, \widehat C_{2x+1}, \ldots, 
\widehat C_{\widehat b} 
\big \}
\end{equation}
and
\begin{equation} \label{eq:projectionfromHn2}
\cS_n \ni \widecheck \xi = 
\big \{
\{\widecheck C_1, \widecheck C_2\}, \ldots, \{\widecheck C_{2y-1}, \widecheck C_{2y}    \}, \widecheck C_{2y+1}, \ldots, 
\widecheck C_{\widecheck b} 
\big \},
\end{equation}
we let
\begin{equation}
\label{eq:embedding}
\begin{split}
\iota (\widehat \xi, \widecheck \xi)
\defeq
&\Big \{
\big \{ 
(\widehat C_1,\varnothing),(\widehat C_2,\varnothing) 
\big \},\ldots,
\big \{
(\widehat C_{2x-1}, \varnothing), (\widehat C_{2x}, \varnothing) \big \}, (\widehat C_{2x+1},\varnothing),\ldots,(\widehat C_{\widehat b},\varnothing), \\
& \quad \big \{
(\varnothing, \widecheck C_1), (\varnothing, \widecheck C_2)
\big \}, \ldots, 
\big \{
(\varnothing, \widecheck C_{2y-1}), (\varnothing, \widecheck C_{2y})
\big \},
(\varnothing, \widecheck C_{2y+1}), \ldots, (\varnothing, \widecheck C_{\widecheck b})
\Big \}
\end{split}
\end{equation}
where we omitted the brackets around the singletons; 
we remind the reader that the elements of $\widehat \xi$, $\widecheck \xi$ and $\iota (\widehat \xi, \widecheck \xi)$ in \eqref{eq:projectionfromHn1} -- \eqref{eq:embedding} are indeed singleton sets or unordered pairs in line with 
\eqref{eq:dispersalfromxiinH}, and that the omission of brackets is purely done for notational convenience (see Remark~\ref{rmk:singletonsnobrackets}). 

We can also extend the complete dispersal map $\mathsf{cd}$ to a map from $\cH_n$ to $\cE_n \times \cE_n$ by letting, for $\xi \in \cH_n$,
$\mathsf{cd} (\xi) \defeq \big (\mathsf{cd} (\widehat \xi), \mathsf{cd} (\widecheck \xi) \big ) $ with $\widehat \xi$ and $\widecheck \xi$
as in \eqref{eq:dispersalfromxiinH}. We call $\xi \in \cH_n$
\emph{completely dispersed} if $\xi = \iota (\mathsf{cd}(\xi))$.

Moreover, we define a \emph{dispersal map}
$\mathsf{d} : \cH_n \to \cH_n$ by letting, for 
\begin{equation*}
    \xi = \{ \{C_1,C_2 \},\ldots, \{ C_{2x-1}, C_{2x} \}, C_{2x+1},\ldots,C_b \} \in \cH_n,
\end{equation*}
\begin{equation*}
\mathsf{d} (\xi) \defeq
\{C_1,\ldots,C_b\};
\end{equation*}
here, we used the shorthand
$C_i = (\widehat C_i, \widecheck C_i)$ for $i \in [b]$.
We call
$\xi \in \cH_n$ \emph{dispersed} if 
$\xi =  \mathsf{d}(\xi)$.

We can now define $\big ( \widehat \Pi^N, \widecheck \Pi^N  \big )^{\mathsf{avg}}$ as a Markov process on $\cH_n$. Let us stress once more that in spite of what the notation might suggest, this is \emph{not} a process 
on $\cS_n \times \cS_n$.

For the moment, we will consider the evolution \emph{conditional on the pedigree}, before taking the average; hence, we drop the subscript and write 
$\big ( \widehat \Pi^N, \widecheck \Pi^N  \big )$ for the quenched version of
$\big ( \widehat \Pi^N, \widecheck \Pi^N  \big )^{\mathsf{avg}}$,
and proceed by mimicking our construction 
of $\Pi^{N}$ in Section~\ref{subsec:fixedpedigrees}, 
based on \emph{two} conditionally independent copies 
$\widehat X_1^{},\ldots, \widehat X_n^{}$ and 
$\widecheck X_1,\ldots, \widecheck X_n$
of the ancestral lines.
Typically, we will let this process start from the initial condition
$\iota(\xi^{(n)}_0,\xi^{(n)}_0)$ with
$\xi^{(n)}_0 = \big \{ \{1\}, \ldots, \{n\}  \big \}$. In that case, 
given a realisation of $P^{(N)}$, we construct $\widehat X_1,\ldots,\widehat X_n$ as before, with initial positions
$\widehat X_j^{}(0) = (0,j)$. For the second set of lineages, we choose the initial positions $\widecheck X_j(0) \defeq (0,n+j)$ as well as an independent realisation of the Mendelian coins.

More generally, when starting with an arbitrary $\xi \in \cH_n$, 
we choose $\widehat X_j (0)$ and $\widecheck X_j (0)$ as follows. First, order $\xi$ so that
\begin{equation*}
\begin{split}
\xi = \Big \{
&\big \{ (\widehat C_1,\widecheck C_1), (\widehat C_2,\widecheck C_2) \big \}, \ldots,
\big \{ (\widehat C_{2x-1}, \widecheck C_{2x-1}), (\widehat C_{2x},\widecheck C_{2x}) \big \},\\
&(\widehat C_{2x+1},\widecheck C_{2x+1}),\ldots, (\widehat C_{b},\widecheck C_{b}),
\Big \}
\end{split}
\end{equation*}
where we allow for some of the blocks $\widehat C_1,\ldots, \widehat C_b$ and $\widecheck C_1,\ldots,\widecheck C_b$ to be empty; however, for each $i$, the union
$\widehat C_i \cup \widecheck C_i$ must not be empty. For $1 \leqslant i \leqslant x$, we pick chromosomes
$w_{2i-1}^{} \defeq (0,i)$ and $w_{2i}^{} \defeq (1,i)$, and for 
$2x + 1 \leqslant i \leqslant b$, we pick chromosomes
$w_i^{} \defeq (0,i)$. Then, for all $1 \leqslant i \leqslant b$ and all $1 \leqslant j \leqslant n$, we set $\widehat X_j^{}(0) \defeq w_{i}^{}$ where $i$ is chosen so that $j \in \widehat C_i^{}$ and we set $\widecheck X_j (0) \defeq w_{i}^{}$ with $i$ chosen so that $j \in \widecheck C_i^{}$. 
The interested reader will quickly verify that this is well-defined. 
Given these initial conditions, we construct the random trajectories $\widehat X_j$ and $\widecheck X_j$ just like $X_j$ in Section~\ref{subsec:fixedpedigrees}, except that we use independent realisations $\widehat M_{c,i,g}$ and $\widecheck M_{c,i,g}$ of the Mendelian coins.

We then define, for each 
$g \in \N$,
$(\widehat \Pi^{N}_g, \widecheck \Pi^{N}_g)$ by first describing its dispersal
$\mathsf{d}(\widehat \Pi^{N}_g, \widecheck \Pi^{N}_g)$.
For this, we define for each $g$ two equivalence relations $\widehat \sim_g^{}$ and $\widecheck \sim_g$ similar to \eqref{eq:xinNm.alt} so that
\begin{equation*}
i \, \dot \sim_g \, j \; \defiff \dot X_i(g) = \dot X_j(g), \quad
\cdot = \wedge, \vee.
\end{equation*}

Then, 
$
\mathsf{d} (\widehat \Pi_g^N, \widecheck \Pi_g^N)$ will have the form
$
\big \{
(\widehat C_1,\widecheck C_1), \ldots,
(\widehat C_b, \widecheck C_b)
\big \}$ where, for each $i \in \{1,\ldots,b\}$, $\widehat C_i$ is either an equivalence class with respect to 
$\widehat \sim_g$, or empty. Put differently, in analogy to Definition~\ref{def:Pi^{N.n}}, each nonempty $\widehat C_i$ is a maximal collection of sample indices whose ancestral lines are located in the same individual.
Similarly, $\widecheck C_i$ is an equivalence class with respect to $\widecheck \sim_g$ as long as it is not empty.

By definition of $\widehat \sim_g$
(and $\widecheck \sim_g$), $( \widehat C,
\widecheck C) \in \mathsf{d} (\widehat \Pi_g^N, \widecheck \Pi_g^N)$ and 
$i,j \in \widehat C$ implies that
$\widehat X_i (g) = \widehat X_j(g)$, and
$i,j \in \widecheck C$ implies that $\widecheck X_i (g) = \widecheck X_j(g)$.
Consequently, for any $i \in \widehat C$ we set
$\widehat X_{\widehat C} (g) \defeq \widehat X_i(g)$ and for $i \in \widecheck C$, we let
$\widecheck X_{\widecheck C}(g) \defeq 
\widecheck X_{i} (g)$.
More precisely, we let
\begin{equation*}
\begin{split}
\mathsf{d}(\widehat \Pi^{N}_g, \widecheck \Pi^{N}_g) \defeq 
& \Big \{
(\widehat C,\widecheck C) : \widehat C \in [n] / \! \widehat \sim_g^{}, \widecheck C \in [n] /  \! \widecheck \sim_g, \, \widehat X^{}_{\widehat C} (g) = \widecheck X_{\widecheck C} (g)
\Big \} \\ 
& \quad \cup \Big \{
(\widehat C,\varnothing) : \widehat C \in [n] / \! \widehat \sim_g,
\widehat X_{\widehat C}^{} (g) \neq \widecheck X_{\widecheck C} (g) \textnormal{ for all } \widecheck C \in [n] / \! \widecheck \sim_g
\Big \} \\
& \quad \cup \Big \{
(\varnothing,\widecheck C) :  \widecheck C \in [n] / \! \widecheck \sim_g,
\widehat X_{\widehat C}^{} (g) \neq \widecheck X_{\widecheck C} (g) \textnormal{ for all } \widehat C \in [n] / \! \widehat \sim_g
\Big \}.
\end{split}
\end{equation*}

Next, we group elements of $\mathsf{d}(\widehat \Pi^{N}_g, \widecheck \Pi^{N}_g)$ together whenever they represent two chromosomes of the same individual. 
We write $\widehat X_{\widehat C}(g) \eqdef (\widehat M_{\widehat C}(g), \widehat I_{\widehat C}(g))$ and 
$\widecheck X_{\widecheck C}(g) \eqdef (\widecheck M_{\widecheck C}(g), \widecheck I_{\widecheck C}(g))$, in line with Subsection \ref{subsec:finitesamplegenealogy},

For each $(\widehat C, \widecheck C) \in \mathsf{d}(\widehat \Pi^{N}_g, \widecheck \Pi^{N}_g)$ (where either $\widehat C$ or $\widecheck C$ may now be empty) we write 
$\mathsf{Ind}(\widehat C,\widecheck C)$ for $\widehat I_{\widehat C}^{}(g)$ ($\widecheck I_{\widecheck C} (g)$) if $\widehat C$ ($\widecheck C$) is not empty. Then, we define $(\widehat \Pi^{N}_g, \widecheck \Pi^{N}_g)$ by replacing two singletons $(\widehat C_1,\widecheck C_1) \neq (\widehat C_2,\widecheck C_2)$ in  $\mathsf{d}(\widehat \Pi^{N}_g, \widecheck \Pi^{N}_g)$
 by an unordered pair
$\big \{ (\widehat C_1,\widecheck C_1), (\widehat C_2,\widecheck C_2)         \big \}$ whenever 
$\mathsf{Ind}(\widehat C_1,\widecheck C_1) = \mathsf{Ind} (\widehat C_2,\widecheck C_2)$.

Now that we have described the evolution of 
$\big (\widehat \Pi^N, \widecheck \Pi^N  \big )$ conditional on the pedigree, we obtain its annealed transition matrix by averaging over the pedigree. 

\begin{defn} \label{def:purecouplingaveragedoverpedigree}
Given $(\widehat \Pi^{N}_0, \widecheck \Pi^{N}_0) = \xi \in \cH_n$, we construct a compatible set of initial conditions $\widehat X(0)$ and $\widecheck X(0)$
as above.
Given the pedigree $P^{(N)}$, we let the gene genealogies $\widehat X$ and $\widecheck X$, driven by independent families $(\widehat M_{c,i,g})_{c,i,g}$ and
$(\widecheck M_{c,i,g})_{c,i,g}$ of Mendelian coins, evolve as in Section~\ref{subsec:fixedpedigrees} to obtain $\widehat X(1)$
and $\widecheck X(1)$ and define $(\widehat \Pi^{N}_1, \widecheck \Pi^{N}_1)$ in the way described above. For $\eta \in \cH_n$, let
\begin{equation*}
\pi^{N}_{\textnormal{pure}} (\xi,\eta) \defeq 
\E^{(N)} \Big [     
\P^{(N)} \big (     
(\widehat \Pi^{N}_1, \widecheck \Pi^{N}_1) = \eta \, \big | \,
P^{(N)}, (\widehat \Pi^{N}_0, \widecheck \Pi^{N}_0) = \xi
\big )
\Big ]
\end{equation*}
Note that the exchangeability of the construction of $P^{(N)}$ means hat the value of this expectation does not depend on the explicit choice of the initial positions $\widehat X(0)$ and $\widecheck X(0)$ as long as this choice is made independently of the pedigree.  Finally, let $\big ( \widehat \Pi^N, \widecheck \Pi^N  \big )^{\mathsf{avg}}$ be the $\cH_n$-valued Markov chain with transition matrix $\pi^{N}_{\textnormal{pure}}$.
\end{defn}

\subsubsection{Separation of timescales}
\label{subsec:timescalesep}
As explained at the beginning of this Section, we want to compute the 
finite-dimensional distributions of the processes with transition matrices $\pi_{\mathsf{pure}}^{N}$, $\pi_{\mathsf{mix}}^{N,\varepsilon}$ and 
$\pi_{\mathsf{pure}}^{N,\varepsilon}$ given in Subsection~\ref{subsec:annealedtransitionmatrices}.
This seems difficult because they (specifically the first two) are only defined in a semi-explicit way via the underlying genealogies. Fortunately, for each of the aforementioned processes, complete dispersal happens on a much faster time scale than coalescence, which makes things tractable in the limit. 
This \emph{seperation of timescales} is treated by the following Lemma, which we quote from~\cite{Mohle1998a}.

\begin{lemma} \label{lem:Moehle}
Let $X_N^{} = \big ( X_N^{} (m) \big )_{m \in \N_0}$ be a sequence of time homogeneous Markov chains on a probability space 
$(\Omega,\cF,\P)$ with the same finite state space $\cS$ and let $\Pi_N^{}$ denote the transition matrix of $X_N^{}$. Assume that the following conditions are satisfied:
\begin{enumerate}
\item 
$A \defeq \lim_{N \to \infty} \Pi_N^{}$ exists and $\Pi_N^{} \neq A$ for all sufficiently large $N$.
\item 
$P \defeq \lim_{m \to \infty} A^m$ exists.
\item 
$G \defeq \lim_{N \to \infty} P B_N^{} P$ exists, where 
$B_N^{} \defeq (\Pi_N^{} - A)/ c_N^{}$ and
$c_N^{} \defeq \| \Pi_N^{} - A \|$ for all $N \in \N$.
\end{enumerate}
If the sequence of initial probability measures $\P_{X_N^{} (0)}$  converges weakly to some probability measure $\mu$, then the finite-dimensional distributions of the process
$
\big ( X_N^{} ( \lfloor t / c_N^{} \rfloor  )         \big )_{t \geqslant 0}
$
converge to those of a time-continuous Markov chain
$(X_t)_{t \geqslant 0}$ with initial distribution
\begin{equation*}
X_0 \stackrel{d}{=} \mu,
\end{equation*}
transition matrix $\Pi(t) \defeq P - I + \ee^{tG} = P \ee^{tG}, t > 0$ and infinitesimal generator $G$.
\end{lemma}

Note that in point $3$, no norm is specified for the matrices; hence, we may assume that it is chosen so that the definition of $c_N^{}$ is consistent with our definition as the annealed pair coalescence probability. 

In order to apply Lemma~\ref{lem:Moehle}, we take inspiration from~\cite{BirknerEtAl2018} and order the set $\cH_n$ in a specific way. We focus on the transition matrix 
$\pi_{\textnormal{pure}}^{N}$, but the argument goes through for 
$\pi_{\mathsf{mix}}^{N,\varepsilon}$ and 
$\pi_{\mathsf{pure}}^{N,\varepsilon}$ in pretty much the same way.

Namely, we start with the elements of $\cE_n \times \cE_n$ (seen as a subset of $\cH_n$, see \eqref{eq:embedding} and the surounding discussion), ordered arbitrarily. In between we add the remaining elements of $\mathsf{cd}^{-1}(\xi)$, again in an arbitrary fashion. 
More precisely, for 
$(\xi,\eta) \in (\cE_n \times \cE_n)^2$,
we let $\xi'_1,\ldots,\xi'_b$ be an enumeration of $\mathsf{cd}^{-1} (\xi)$,
and $\eta'_1,\ldots,\eta'_d$ an enumeration of $\mathsf{cd}^{-1}(\eta)$ such that
$\xi'_1 = \xi$ and $\eta'_1 = \eta$. We define the $b \times d$-matrix 
\begin{equation*}
\overline \pi_{\mathsf{pure}}^N(\xi,\eta) \defeq \big (\pi_{\mathsf{pure}}^N (\xi_i',\eta_j')
\big )_{1\leqslant i \leqslant b, 1\leqslant j \leqslant d}.
\end{equation*}

This yields a 
block-decomposition
\begin{equation} \label{eq:blocks}
\pi_{\textnormal{pure}}^{N} = \big ( \overline \pi_{\textnormal{pure}}^{N} (\xi,\eta) \big )_{
(\xi,\eta) \in (\cE_n \times \cE_n)^2}
\end{equation}

By Lemma~\ref{lem:Moehle}, we need to compute 
\begin{equation*}
A_{\textnormal{pure}} 
\defeq
\lim_{N \to \infty} 
\pi_{\textnormal{pure}}^{N}.
\end{equation*}
Thanks to the decomposition in \eqref{eq:blocks}, this amounts to computing the limits of submatrices
\begin{equation*}
\overline A_{\textnormal{pure}} (\xi, \eta) 
\defeq 
\lim_{N \to \infty} 
\overline \pi_{\textnormal{pure}}^{N} (\xi,\eta)
\end{equation*}
for all $\xi$ and $\eta$ in $\cE_n \times \cE_n$.
If $\xi \neq \eta$,
it is clear that a transition from any element of 
$\mathsf{cd}^{-1} (\xi)$
to any element of
$\mathsf{cd}^{-1} (\eta)$ has probability of order $O(c_N^{})$ because it involves the coalescence of at least two ancestral lines, which happens with probability of order
$c_N^{}$. It follows from \eqref{eq:limcN=0} that $\overline A_{\textnormal{pure}} (\xi, \eta) = 0$ 
for $\xi \neq \eta$ so that $A_{\textnormal{pure}}$ will be block diagonal. To understand the diagonal blocks
$\overline A_{\textnormal{pure}} (\xi,\xi)$, 
let $\xi',\xi'' \in \mathsf{cd}^{-1} (\xi)$
and let $(\widehat \Pi^{N}_0, \widecheck \Pi^{N}_0) \defeq \xi'$, using the convention that  $\pi_{\mathsf{pure}}^N$ represents a transition between time $0$ and $1$. In the notation that we introduced just before Definition \ref{def:purecouplingaveragedoverpedigree}, this means that
\begin{equation*}
\mathsf{d}(\widehat \Pi^{N}_0, \widecheck \Pi^{N}_0) = \{\widehat X_{\widehat C}(0) : (\widehat C, \widecheck C) \in \mathsf{d}(\xi') \} \cup \{\widecheck X_{\widecheck C}(0) : (\widehat C, \widecheck C) \in \mathsf{d}(\xi') \}.
\end{equation*}
Let $(\widehat C, \widecheck C) \in \mathsf{d}(\xi')$ and
$(\widehat D, \widecheck D) \in \mathsf{d}(\xi')$ with $(\widehat C, \widecheck C) \neq (\widehat D, \widecheck D)$.
We write $\mathsf{Ind_1}(\widehat C, \widecheck C)$ for $\widehat I_{\widehat C}(1)$ (respectively $\widecheck I_{\widecheck C}(1)$), if $\widehat C$ (respectively $\widecheck C$) is not empty. We define
$\mathsf{Ind}_1 (\widehat D, \widecheck D)$ analogously. In simple words, $\mathsf{Ind}_1(\widehat C, \widecheck C)$ and $\mathsf{Ind}_1(\widehat D, \widecheck D)$ are the ancestors of the genes represented by $(\widehat C,\widecheck C)$ and $(\widehat D, \widecheck D)$.

With probability $1 - O(c_N^{})$, we have that $\mathsf{Ind}_1(\widehat C, \widecheck C) \neq \mathsf{Ind}_1 (\widehat D, \widecheck D)$.
This is due to the absence of selfing and the fact that the probability that any given pair of lineages finds the same parent in one time-step is of the same order as the probability that they coalesce, which is $c_N^{}$.
On the other hand, if both $\widehat C$ and $\widecheck C$ are nonempty, with probability $1/2$, we have $\{(\widehat C, \varnothing),(\varnothing,\widecheck C) \} \in (\widehat \Pi^{N}_1, \widecheck \Pi^{N}_1)$, corresponding to the event that
$\widehat M_{\widehat C}(1) \neq \widecheck M_{\widecheck C}(1)$. 

To summarise, we have seen that, with probability $1 - O(c_N^{})$, each $I \in (\widehat \Pi^{N}_1, \widecheck \Pi^{N}_1)$ 
is either
\begin{itemize}
\item 
a singleton contained in $\mathsf{d} (\xi')$
\item 
of the form $\{(\widehat C, \varnothing), (\varnothing, \widecheck C) \}$ for 
$(\widehat C, \widecheck C) \in \mathsf{d}(\xi')$.
\end{itemize}
We conclude that 
$\lim_{N \to \infty} \pi_{\mathsf{pure}}^N (\xi',\xi'') = 0$ unless $\xi''$ can be obtained from $\xi'$ by first replacing 
each element of $\xi'$ of the form
$\{ (\widehat C, \widecheck C), (\widehat D, \widecheck D) \}$ by two singletons
$\{(\widehat C, \widecheck C) \}$ and
$\{(\widehat D, \widecheck D) \}$, and subsequently replacing an arbitrary number of singletons
$\{ (\widehat C, \widecheck C) \}$ by 
$\{(\widehat C, \varnothing),(\varnothing, \widecheck C)\}$. In this case, $\pi_{\mathsf{pure}}^N (\xi',\xi'')$ converges to $2^{-K}$ where $K$ is the number of singletons thus replaced.



To put this simply, the matrix $A_{\textnormal{pure}}$ is block diagonal. Each diagonal block
$\overline A_{\textnormal{pure}}(\xi,\xi)$ is the transition matrix of a discrete-time Markov chain on $\mathsf{cd}^{-1}(\xi)$ with absorbing state $\xi$. 

The projectors needed for 2.\ in Lemma~\ref{lem:Moehle} are readily obtained. 
\begin{equation*}
P_{\textnormal{pure}} \defeq \lim_{m \to \infty} 
\big (A_{\textnormal{pure}}    \big )^m
=
\big (\overline P_{\textnormal{pure}} (\xi,\eta) \big )_{(\xi,\eta) \in (\cE_n \times \cE_n)^2}
\end{equation*}

where $\overline P_{\textnormal{pure}} (\xi,\xi)$ is the 
$\# \mathsf{cd}^{-1} (\xi) \times \# \mathsf{cd}^{-1} (\xi)$-matrix
that consists of only $1$s in the first column and only of $0$s everywhere else, while $\overline P_{\textnormal{pure}} (\xi,\eta)$ vanishes for $\xi \neq \eta$.

Next, in line with 3. in Lemma~\ref{lem:Moehle}, we set
\begin{equation} \label{eq:purecouplinggenerator}
B^{N}_{\textnormal{pure}} \defeq \frac{1}{c_N^{}} \big ( \pi_{\textnormal{pure}}^{N} - A_{\textnormal{pure}}   \big )
\end{equation}
and compute
\begin{equation*}
G_{\textnormal{pure}} \defeq \lim_{N \to \infty}
P_{\textnormal{pure}} B^{N}_{\textnormal{pure}} 
P_{\textnormal{pure}}.
\end{equation*}
A `back-of-the-envelope' calculation shows that 
$\pi^{N}_{\textnormal{pure}} P_{\textnormal{pure}}$ has the following form. For all $(\xi,\eta) \in \cH_n^2$, the $(\xi,\eta)$-entry vanishes whenever $\eta$ is not completely dispersed; this is because the $\eta$-column of $P_{\mathsf{pure}}$ vanishes in this case. 
On the other hand, 
if $\eta$ is completely dispersed, then the $(\xi,\eta)$-entry is given by 
\begin{equation} \label{eq:pipuretilde}
\widetilde \pi_{\mathsf{pure}}^N (\xi,\eta) 
\defeq
\sum_{\eta' \in \mathsf{cd}^{-1}(\eta)} \pi^{N}_{\textnormal{pure}} (\xi,\eta')
\end{equation}
because, due to the block structure of $P_{\mathsf{pure}}$, the $\eta$-column carries $1$s at the positions corresponding to $\mathsf{cd}^{-1} (\eta)$ and $0$s otherwise. Then, multiplying $P_{\textnormal{pure}}$ from the left, we see that 
$P_{\textnormal{pure}} \pi^{N}_{\textnormal{pure}} P_{\textnormal{pure}}$ has the following block structure
\begin{equation*}
P_{\textnormal{pure}} \pi^{N}_{\textnormal{pure}} 
P_{\textnormal{pure}}
=
\big (\overline Q^{N}_{\textnormal{pure}} (\xi,\eta)   
\big )_{(\xi,\eta) \in (\cE_n \times \cE_n)^2}.
\end{equation*}
The entries of the first column of $\overline Q_{\textnormal{pure}}^{N} (\xi,\eta)$ are all equal to 
$\sum_{\eta' \in \mathsf{cd}^{-1}(\eta)}
\pi_{\textnormal{pure}}^{N} (\xi,\eta')$, and the entries of all other columns vanish. 
From the definition of $P_{\textnormal{pure}}$, it is clear that 
$P_{\textnormal{pure}} A_{\textnormal{pure}}
P_{\textnormal{pure}} =
P_{\textnormal{pure}}$. Hence,
\begin{equation*}
P_{\mathsf{pure}} B^{N}_{\mathsf{pure}} P_{\mathsf{pure}} =
\frac{1}{c_N^{}} \big ( \overline Q^{N}_{\mathsf{pure}} (\xi,\eta)  
- \delta_{\xi, \eta} 
\overline P_{\mathsf{pure}}(\xi,\eta)
\big)_{(\xi,\eta) \in (\cE_n \times \cE_n)^2  }.
\end{equation*}
Below, we will show that 
\begin{equation*}
\lim_{N \to \infty} \frac{1}{c_N^{}} 
 \big (\widetilde \pi^{N}_{\textnormal{pure}} 
(\xi,\eta) - \delta_{\xi,\eta}    \big )
\end{equation*}
exists for all $\xi, \eta \in \cE_n$, which implies 
\begin{lemma} \label{lem:annealedlimitgenerator}
The generator $G_{\textnormal{pure}}$ has the block structure
\begin{equation*}
G_{\textnormal{pure}} = \big (
\overline G_{\textnormal{pure}} (\xi,\eta)
\big )_{(\xi,\eta) \in (\cE_n \times \cE_n)^2}.
\end{equation*}
The entries of the first column of $\overline G_{\textnormal{pure}} (\xi,\eta)$ are all equal to 
\begin{equation*}
q_{\xi,\eta}^{\mathsf{pure}} \defeq 
\lim_{N \to \infty} \frac{1}{c_N^{}} 
 \big (\widetilde \pi^{N}_{\textnormal{pure}} 
(\xi,\eta) - \delta_{\xi,\eta}    \big )
\end{equation*}
with
$
\widetilde \pi^{N}_{\textnormal{pure}} 
(\xi,\eta) 
$
given in \eqref{eq:pipuretilde}
and all other entries are zero.
In particular, writing $\widetilde t$ for 
$\lfloor t / c_N^{} \rfloor$, 
the finite-dimensional distributions of 
$\big ( \widehat \Pi^{N}_{\widetilde t}, \widecheck 
\Pi^{N}_{\widetilde t}  \big )^{\mathsf{avg}}_{t \in \R_+}$ converge, as $N \to \infty$, to those of a continuous-time Markov chain on $\cE_n \times \cE_n$ with transition rates 
$q_{\xi, \eta}^{\mathsf{pure}}$. Analogous statements hold for the other two chains as well. 
More precisely, the finite-dimensional distributions of
$\big ( 
\widehat \Pi^N_{\widetilde t}, \widecheck \Pi_{t}^{(N,\varepsilon)}
\big)_{t \in \R_+}^{\mathsf{avg}}$
and
$\big (
\widehat
\Pi_{t}^{(N,\varepsilon)},
\widecheck
\Pi_{t}^{(N,\varepsilon)}
\big)_{t \in \R_+}^{\mathsf{avg}}
$ (recall that the naive coalescence is already defined on the correct timescale to begin with and does not need rescaling)
converge to continuous-time Markov chains on
$\cE_n \times \cE_n$ with transition rates 
$q_{\xi,\eta}^{\mathsf{mix}, \varepsilon}$ and 
$q_{\xi,\eta}^{\mathsf{pure},\varepsilon}$ respectively, given by
\begin{equation*}
q_{\xi,\eta}^{\mathsf{mix},\varepsilon} \defeq 
\lim_{N \to \infty} \frac{1}{c_N^{}} 
 \big (\widetilde \pi^{N,\varepsilon}_{\textnormal{mix}} 
(\xi,\eta) - \delta_{\xi,\eta}    \big )
\end{equation*}
and
\begin{equation*}
q_{\xi,\eta}^{\mathsf{pure},\varepsilon} \defeq 
\lim_{N \to \infty} \frac{1}{c_N^{}} 
 \big (\widetilde \pi^{N,\varepsilon}_{\textnormal{pure}} 
(\xi,\eta) - \delta_{\xi,\eta}    \big )
\end{equation*}
with the \emph{aggregate transition probabilities}
$\widetilde \pi^{N,\varepsilon}_{\textnormal{mix}}$ and
$\widetilde \pi^{N,\varepsilon}_{\textnormal{pure}}$ given by 
\begin{equation*}
\widetilde \pi^{N,\varepsilon}_{\textnormal{mix}} 
(\xi,\eta) 
\defeq
\sum_{\eta' \in \mathsf{cd}^{-1} (\eta) } 
\pi_{\textnormal{mix}}^{N,\varepsilon} (\xi,\eta').
\end{equation*}
and
\begin{equation*}
\widetilde \pi^{N,\varepsilon}_{\textnormal{pure}} 
(\xi,\eta) 
\defeq
\sum_{\eta' \in \mathsf{cd}^{-1} (\eta) } 
\pi_{\textnormal{pure}}^{N,\varepsilon} (\xi,\eta').
\end{equation*}

\end{lemma}
\begin{proof}
The statement about the f.d.d. convergence follows immediately from Lemma~\ref{lem:Moehle}. Proving the existence of the limits will be the content of Section~\ref{subsec:aggregates}.
\end{proof}

\subsubsection{Aggregate transition probabilities and their scaling limits}
\label{subsec:aggregates}

To finish the proof of Lemma~\ref{lem:l2approximation} and thus of our main theorem, we need to
compute, for all $\xi, \eta \in \cE_n \times \cE_n$, the asymptotic transition rates 
$q_{\xi,\eta}^{\textnormal{pure}}$, 
$q_{\xi,\eta}^{\textnormal{pure},\varepsilon}$ and 
$q_{\xi,\eta}^{\textnormal{mix},\varepsilon}$. This computation will rely on a more explicit description of the `aggregate transition matrices' 
$\widetilde \pi^{N}_{\textnormal{pure}}$ and 
$\widetilde \pi^{N,\varepsilon}_{\textnormal{mixed}}$; the basic idea is similar to the argument in Subsection~\ref{subsec:transitionsconditionalonoffspringnumbers} for a single coalescent copy.

Let $\xi = (\widehat \xi, \widecheck \xi)$ and $\eta = (\widehat \eta, \widecheck \eta)  
\in \cE_n \times \cE_n$ 
with $\dot b \defeq \# \dot \xi$ such that
$\dot \eta$ arises from $\dot \xi$ by merging $\dot r$ groups of blocks of sizes $\dot k_1,\ldots,\dot k_{\dot r} \geqslant 2$ with
$\dot s = \dot b - \sum_{\ell = 1}^{\dot r} \dot k_\ell$
blocks of $\dot \xi$
not participating in any merger. 
Recall that $\dot \xi \vdash \dot \eta$ means that 
$\dot \eta$ is obtained from $\dot \xi$ by a pair coalescence,
that is $\dot r = 1$ and $\dot k_1^{} = 2$.

Intuitively, the fact that $\xi$ is completely dispersed means that all blocks of $\widehat \xi$ as well as $\widecheck \xi$ represent different genes located on different individuals, and that the sets of genes represented by 
the blocks of $\widehat \xi$ and those of $\widecheck \xi$ are disjoint.
As in Subsection~\ref{subsec:transitionsconditionalonoffspringnumbers}, we consider without loss of generality only $0$-genes and recall that $ \widehat V_i$ denotes the number of individuals that have $i$ as their $0$-parent.
We start with the computation of $q_{\xi,\eta}^{\mathsf{pure}}$.
\begin{lemma} \label{lem:qpure}
For 
$\xi = (\widehat \xi, \widecheck \xi),
\eta = (\widehat \eta, \widecheck \eta) \in \cE_n \times \cE_n \subseteq \cH_n$
 and $\xi \neq \eta$, 
\begin{equation*}
\begin{split}
&\lim_{N \to \infty} 
\frac{1}{c_N^{}} 
\widetilde \pi_{\textnormal{pure}}^{N} (\xi,\eta) \\
&\quad  = 
\int_{\Delta \setminus \{\mathbf{0}\}} 
p(x;\widehat \xi,\widehat \eta) p(x;\widecheck \xi, \widecheck \eta) 
\frac{\Xi(\dd x)}{\langle x, x \rangle}
+
c_{\textnormal{pair}}^{}
\1_{\widehat \xi \vdash \widehat \eta, \widecheck \xi = \widecheck \eta^{}}
+
c_{\textnormal{pair}}^{}
\1_{\widehat \xi = \widehat \eta, \widecheck \xi \vdash \widecheck \eta}^{} 
\eqdef q_{\xi,\eta}^{\mathsf{pure}},
\end{split}
\end{equation*}
where $p$ is given in \eqref{eq:paintboxtransitionprime} and 
$c_{\textnormal{pair}}^{} = 1 - \Xi 
\big (
\Delta \setminus \{ \mathbf{0} \} 
\big )$.
\end{lemma}
\begin{proof}
See Appendix~\ref{app:lemmaproof}.
\end{proof}

Next, we deal with the mixed coupling $\big (
\widehat \Pi^N, \widecheck \Pi^{(N,\varepsilon)}
\big )^{\mathsf{avg}}$.

\begin{lemma} \label{lem:mixedrates}
For $\varepsilon > 0$, $\xi$ and $\eta$ as above and $\xi \neq \eta$, 
\begin{equation*}
\begin{split}
&\lim_{N \to \infty} \frac{1}{c_N^{}}  \widetilde \pi_{\textnormal{mix}}^{N,\varepsilon} \big ( (\widehat \xi, \widecheck \xi), (\widehat \eta, \widecheck \eta)   \big ) \\
& \quad =
\1_{\widecheck \xi = \widecheck \eta} \int_{\Delta \setminus \Delta_{\geqslant \varepsilon}} p(x;\widehat \xi, \widehat \eta)  \frac{\Xi(\dd x)}{\langle x, x \rangle} 
+
\int_{\Delta_{\geqslant \varepsilon}} p(x;\widehat \xi, \widehat \eta) p(x; \widecheck \xi, \widecheck \eta) \frac{\Xi(\dd x)}{\langle x, x \rangle} \\
& \quad \, + c_{\textnormal{pair}}^{} \1_{\widehat \xi \vdash \widehat \eta, \widecheck \xi = \widecheck \eta} 
+ c_{\textnormal{pair}}^{} \1_{\widehat \xi  = \widehat \eta, 
\widecheck \xi \vdash \widecheck \eta} \eqdef q_{\xi,\eta}^{\mathsf{mix},\varepsilon},
\end{split}
\end{equation*}
where $p$ is given in \eqref{eq:paintboxtransitionprime},
$c_{\textnormal{pair}}^{} = 1 - \Xi 
\big (
\Delta \setminus \{ \mathbf{0} \} 
\big )$ and
$ \Delta_{\geqslant \varepsilon} \defeq \{ x \in \Delta : \|x\| \geqslant \varepsilon\}$
.
\end{lemma}

\begin{proof}
Recall the definitions of $\pi_{\mathsf{mix}}^{N,\varepsilon}$
(Definition~\ref{def:mixedpairtransitions}) and
$\widetilde \pi_{\mathsf{mix}}^{N,\varepsilon}$
(Lemma~\ref{lem:annealedlimitgenerator}). We will decompose 
$\widetilde \pi_{\mathsf{mix}}^{N,\varepsilon}$
to account for the contributions of generations with low and high individual progeny, respectively, corresponding to the two integrals on the right-hand side of the claim. We write 
$\widetilde 
\pi_{\mathsf{mix}}^{N,\varepsilon} = 
\widetilde
\pi_{\mathsf{mix},<}^{N,\varepsilon}
+
\widetilde
\pi_{\mathsf{mix},\geqslant}^{N,\varepsilon}
$
with
\begin{equation*}
\widetilde
\pi_{\mathsf{mix},<}^{N,\varepsilon} \big ( (\widehat \xi, \widecheck \xi),
(\widehat \eta, \widecheck \eta)  \big )
\defeq 
\E^{(N)} \Big [ \1_{\| \widetilde V \| < \varepsilon }
\widetilde
\pi^N (V; \widehat \xi, \widehat \eta )
p_\varepsilon^{} (\widetilde V; \widecheck \xi, \widecheck \eta)
\Big ]
\end{equation*}
and
\begin{equation*}
\widetilde \pi_{\mathsf{mix},\geqslant}^{N,\varepsilon} \big ( (
\widehat \xi, \widecheck \xi),
(\widehat \eta, \widecheck \eta)  \big )
\defeq 
\E^{(N)} \Big [ \1_{\| \widetilde V \| \geqslant \varepsilon }
\widetilde
\pi^N (V; \widehat \xi, \widehat \eta )
p_\varepsilon^{} (\widetilde V; \widecheck \xi, \widecheck \eta)
\Big ],
\end{equation*}
where
$
\widetilde \pi^N (V;\widehat \xi, \widehat \eta)
$
is the aggregated transition probability (conditional on the total offspring numbers) for a \emph{single} coalescent that was given in \eqref{eq:aggregatedtransprobs}.
Recall the definition of
$p_\varepsilon^{}$ from \eqref{eq:pepsilon}.
For $\| \widetilde V \| < \varepsilon$ (and $\widecheck \xi \neq \widecheck \eta$), 
$p_\varepsilon^{}
\big ( \widetilde V ; \widecheck \xi, \widecheck \eta \big ) =
\1_{\widecheck \xi \vdash \widecheck \eta} c_N^{} c_{\mathsf{pair}}^{}
+ O(c_N^2)
$. 
Recalling from \eqref{eq:transitionsconditionedonnumbersaggregated} that $\widetilde \pi^{N}(\cdot,\cdot)
=
\E^{(N)} \big [ \widetilde \pi^N (V;\cdot,\cdot) \big ]$, we get
\begin{equation*}
\begin{split}
& 0 \leqslant \widetilde  \pi_{\mathsf{mix},<}^{N,\varepsilon} \big ( (\widehat \xi, \widecheck \xi),
(\widehat \eta, \widecheck \eta)  \big ) \\
& \quad = 
\1_{\widecheck \xi \vdash \widecheck \eta} c_N^{} c_{\mathsf{pair}}^{}
\E^{(N)} \Big [ \1_{\| \widetilde V \| < \varepsilon }
\widetilde \pi^{N} (V; \widehat \xi, \widehat \eta )
\Big ] + O(c_N^2).
\end{split}
\end{equation*}
Since $\1_{\| \widetilde V  \| < \varepsilon} \widetilde \pi^N \big (V; \widehat \xi, \widehat \eta  \big )$ is bounded between $0$ and $1$ and
converges in probability to $\1_{\widehat \xi = \widehat \eta}$ as
$N \to \infty$,
dividing by $c_N^{}$ and letting $N \to \infty$, dominated convergence gives us
\begin{equation} \label{eq:pitildemixedsmall}
\lim_{N \to \infty} \frac{1}{c_N^{}} 
\widetilde  \pi_{\mathsf{mix},<}^{N,\varepsilon} \big ( (\widehat \xi, \widecheck \xi),
(\widehat \eta, \widecheck \eta)  \big )
=
c_{\mathsf{pair}}^{} \1_{\widehat \xi = \widehat \eta, \, \widecheck \xi \vdash \widecheck \eta }
\quad
\textnormal{ for } \widecheck \xi \neq \widecheck \eta.
\end{equation}
On the other hand,
when 
$\widecheck \xi = \widecheck \eta$
(and hence 
$\widehat \xi \neq \widehat \eta$
by our assumption that $\xi \neq \eta$),
we use the fact that 
$p_\varepsilon^{}(\widetilde V;\cdot,\cdot)$ 
is a transition probability kernel whence
$p_\varepsilon^{}(\widetilde V; \widecheck \xi, \widecheck \xi) = 1 - \sum_{\widecheck \eta \neq \widecheck \xi} p_\varepsilon^{} 
(\widetilde V; \widecheck \xi, \widecheck \eta)$. Thus, 
\begin{equation*}
\widetilde  \pi_{\mathsf{mix},<}^{N,\varepsilon} \big ( (\widehat \xi, 
\widecheck \xi),
(\widehat \eta, \widecheck \xi)  \big )
=
\E^{(N)} \Big [
\1_{\|\widetilde V \| < \varepsilon} 
\widetilde \pi^N (V;\widehat \xi, \widehat \eta)
\Big ] 
- \sum_{\widecheck \eta \neq \widecheck \xi}
\widetilde  \pi_{\mathsf{mix},<}^{N,\varepsilon} \big ( (\widehat \xi, 
\widecheck \xi),
(\widehat \eta, \widecheck \eta)  \big )
\end{equation*}
and by \eqref{eq:pitildemixedsmall}, the sum disappers after dividing by $c_N^{}$ and letting $N \to \infty$ since $\widehat \xi \neq \widehat \eta$ and it follows that
\begin{equation*}
\begin{split}
&\lim_{N \to \infty} 
\frac{1}{c_N^{}} 
\widetilde  \pi_{\mathsf{mix},<}^{N,\varepsilon} \big ( (\widehat \xi, 
\widecheck \xi),
(\widehat \eta, \widecheck \eta)  \big ) \\
& \quad =
\lim_{N \to \infty} 
\frac{1}{c_N^{}} 
\widetilde \pi^{N} (\widehat \xi, \widehat \eta) -
\lim_{N \to \infty} 
\frac{1}{c_N^{}} 
\E^{(N)} \Big [ \1_{\| \widetilde V \| \geqslant \varepsilon }
\big ( p(\widetilde V; \widehat \xi, \widehat \eta ) + O(N^{-1}) \big ) \Big ],
\end{split}
\end{equation*}
recalling from \eqref{eq:transitionsconditionedonnumbersaggregated} that $\widetilde \pi^{N}(\cdot,\cdot)
=
\E^{(N)} \big [ \widetilde \pi^N (V;\cdot,\cdot) \big ]$ and from Lemma \ref{lem:largerep} that $\widetilde \pi^N (V ; \widehat \xi, \widehat \eta) = p(\widetilde V; \widehat \xi, \widehat \eta) + O(N^{-1})$.
As was shown in~\cite[Eq.~(1.7) and Theorem A.5]{BirknerEtAl2018},
the first limit evaluates to
\begin{equation*}
\int_{\Delta \setminus \{ { \bf 0 } \} } p(x;\widehat \xi, \widehat \eta)  \frac{\Xi(\dd x)}{\langle x, x \rangle} + 
c_{\mathsf{pair}}^{} \1_{\widehat \xi \vdash \widehat \eta}.
\end{equation*}
By the vague convergence assumption for the law of $\widetilde V$ (see \eqref{eq:PhiNconv2})
and the theorem of dominated convergence, the second limit becomes (except maybe for some discrete set of $\varepsilon$)
\begin{equation*}
\int_{\Delta_{\geqslant \varepsilon}} p(x;\widehat \xi,\widehat \eta)
\frac{\Xi(\dd x)}{\langle x,  x \rangle }.
\end{equation*}
Combining the two cases $\widecheck \xi = \widecheck \eta$ and 
$\widecheck \xi \neq \widecheck \eta$, we arrive at
\begin{equation*}
\lim_{N \to \infty} \frac{1}{c_N^{}} 
\widetilde  \pi_{\mathsf{mix},<}^{N,n,\varepsilon} \big ( (\widehat \xi, \widecheck \xi),
(\widehat \eta, \widecheck \eta)  \big )
=
c_{\mathsf{pair}}^{} \1_{\widehat \xi= \widehat \eta, \widecheck \xi \vdash \widecheck \eta} +
c_{\mathsf{pair}}^{} \1_{\widehat \xi \vdash \widehat \eta, \widecheck \xi = \widecheck \eta} + 
\1_{\widecheck \xi = \widecheck \eta} \int_{\Delta \setminus \Delta_{ \geqslant \varepsilon}} p(x;\widehat \xi,\widehat \eta)
\frac{\Xi(\dd x)}{\langle x,  x \rangle }.
\end{equation*}
Lemma~\ref{lem:largerep}, the vague convergence for $\widetilde V$ and the theorem of dominated convergence immediately yield
\begin{equation*}
\lim_{N \to \infty} \frac{1}{c_N^{}} 
\widetilde 
\pi_{\mathsf{mix},\geqslant}^{N,n,\varepsilon} \big ( (\widehat \xi, \widecheck \xi),
(\widehat \eta, \widecheck \eta)  \big )
=
\int_{\Delta_{\geqslant \varepsilon}} 
p(x;\widehat \xi,\widehat \eta) 
p(x;\widecheck \xi,\widecheck \eta)
\frac{\Xi(\dd x)}{\langle x,x \rangle}.
\end{equation*}
\end{proof}

\begin{remark}
If we extend $p(x;\xi, \eta)/\langle x, x \rangle$ continuously to $x = \bf 0$ with value
$\1_{\xi \vdash \eta}$, then
  \begin{equation*}
\int_{\Delta \setminus \{ { \bf 0 } \} } p(x;\xi, \eta)  \frac{\Xi(\dd x)}{\langle x, x \rangle} + 
c_{\mathsf{pair}}^{} \1_{\xi \vdash \eta}
= \int_\Delta p(x;\xi, \eta)  \frac{\Xi(\dd x)}{\langle x, x \rangle}.
\end{equation*}
\hfill \mbox{$\Diamond$}
\end{remark}

Finally, we deal with two coupled naive coalescents.
\begin{lemma} \label{lem:qpureepsilon}
For $\varepsilon > 0$, $\xi,\eta \in \cE_n \times \cE_n$ as above and $\xi \neq \eta$, 
\begin{equation*}
\begin{split}
&\lim_{N \to \infty} \frac{1}{c_N^{}}  \widetilde \pi_{\textnormal{pure}}^{N,\varepsilon} \big ( (\widehat \xi, \widecheck \xi), (\widehat \eta, \widecheck \eta)   \big ) \\
& \quad =
\int_{\Delta_{\geqslant \varepsilon}} p(x;\widehat \xi, \widehat \eta) p(x;\widecheck \xi,\widecheck \eta)  \frac{\Xi(\dd x)}{\langle x, x \rangle} 
 + c_{\textnormal{pair}}^{} \1_{\widehat \xi \vdash \widehat \eta, \widecheck \xi = \widecheck \eta} 
+ c_{\textnormal{pair}}^{} \1_{\widehat \xi  = \widehat \eta, \widecheck \xi \vdash \widecheck \eta} \eqdef q_{\xi,\eta}^{\mathsf{pure},\varepsilon},
\end{split}
\end{equation*}
where $p$ is given in \eqref{eq:paintboxtransitionprime} and 
$c_{\textnormal{pair}}^{} = 1 - \Xi 
\big (
\Delta \setminus \{ \mathbf{0} \} 
\big )$.
\end{lemma}
\begin{proof}
Recall from Definition~\ref{def:purelynaivetransitions} that
\begin{equation*}
\widetilde \pi_{\textnormal{pure}}^{N,\varepsilon} 
\big ( (\widehat \xi,\widehat \eta), 
(\widecheck \xi,\widecheck \eta) \big )
=
\pi_{\textnormal{pure}}^{N,\varepsilon} 
\big ( (\widehat \xi,\widehat \eta), 
(\widecheck \xi,\widecheck \eta) \big )
=
\E^{(N)} \Big [
p_\varepsilon^{} ( \widetilde V; \widehat \xi, \widehat \eta) 
p_\varepsilon^{} (\widetilde V; \widecheck \xi, \widecheck \eta )
\Big ];
\end{equation*}
we may drop the tilde because the couple of naive coalescents is defined on $\cE_n \times \cE_n$ to begin with.  We split
\begin{equation*}
\begin{split}
\E \Big [
&p_\varepsilon^{} ( \widetilde V; \widehat \xi, \widehat \eta) 
p_\varepsilon^{} (\widetilde V; \widecheck \xi, \widecheck \eta )
\Big ] \\
&\quad =
\E \Big [ \1_{\| \widetilde V \| \geqslant \varepsilon}
p_\varepsilon^{} ( \widetilde V; \widehat \xi, \widehat \eta) 
p_\varepsilon^{} (\widetilde V; \widecheck \xi, \widecheck \eta )
\Big ]
+
\E \Big [ \1_{\| \widetilde V \| < \varepsilon}
p_\varepsilon^{} ( \widetilde V; \widehat \xi, \widehat \eta) 
p_\varepsilon^{} (\widetilde V; \widecheck \xi, \widecheck \eta )
\Big ].
\end{split}
\end{equation*}
Dividing by $c_N^{}$ and letting $N \to \infty$, the first expectation converges to
\begin{equation*}
\int_{\Delta_{\geqslant \varepsilon}} p(x;\widehat \xi, \widehat \eta) p(x;\widecheck \xi,\widecheck \eta)  \frac{\Xi(\dd x)}{\langle x, x \rangle} 
\end{equation*}
by the vague convergence of $c_N^{-1} \cL(\widetilde V)$ (and dominated convergence).
To compute the limit of the second expectation, note that \eqref{eq:pepsilon} implies
that for
$\| \widetilde V \| < \varepsilon$,
for
$\cdot \in \{\wedge, \vee  \}$, and
$\dot \xi \neq \dot \eta$,
\begin{equation*}
p_{\varepsilon}^{} (\widetilde V; \dot \xi,\dot \eta) 
=
c_N^{} c_{\mathsf{pair}}^{} \1_{\dot \xi \vdash \dot \eta} +  O(c_N^2)
\end{equation*}
and consequently $p_{\varepsilon}^{}(\widetilde V,\dot\xi, \dot \eta) = 1+ O(c_N^{})$ for $\dot \xi = \dot \eta$ since $p_\varepsilon(\widetilde V,\dot \xi, \cdot)$ is a probability distribution.
Therefore,
\begin{equation*}
\lim_{N \to \infty} \frac{1}{c_N^{}}
\E \Big [ \1_{\| \widetilde V \| < \varepsilon}
p_\varepsilon^{} ( \widetilde V; \widehat \xi, \widehat \eta) 
p_\varepsilon^{} (\widetilde V; \widecheck \xi, \widecheck \eta )
\big ] = 
c_{\textnormal{pair}}^{} \1_{\widehat \xi \vdash \widehat \eta, \widecheck \xi = \widecheck \eta} 
+ c_{\textnormal{pair}}^{} \1_{\widehat \xi  = \widehat \eta, \widecheck \xi \vdash \widecheck \eta}.
\end{equation*}
\end{proof}

We finally have everything in place to prove Lemma~\ref{lem:l2approximation}.

\begin{proof}[Proof of Lemma~\ref{lem:l2approximation}]

By the preliminary discussion at the beginning of this section, we have
\begin{equation*}
\begin{split}
& \E^{(N)} \Big [
\Big (
\P^{(N)} \big ( \Pi^{N}_{\widetilde t_1^{}} =  \xi_1^{}, \ldots, 
\Pi^{N}_{\widetilde t_k^{}} = \xi_k^{} \, \big | \, \cA^{(N)} \big ) \\
& \quad -
\P^{(N)} \big ( \Pi_{t_1^{}}^{(N, \varepsilon)} = \xi_1^{},
\ldots,
\Pi_{t_k^{}}^{(N,\varepsilon)} = \xi_k^{}
         \, \big | \, \Psi^{(N)}  \big )
\Big )^2 
\Big ] \\
& \quad = 
\P^{(N)} \Big ( \big ( \widehat \Pi^N_{\widetilde t_i^{}},
\widecheck \Pi^N_{\widetilde t_i^{}} \big ) = \big(\xi_i^{}, \xi_i^{}
\big )  \textnormal{ for all } i = 1, \ldots, k   \Big ) \\
& \quad - 2 
\P^{(N)} \Big ( \big ( \widehat \Pi^{(N,\varepsilon)}_{t_i^{}},
 \widecheck \Pi^N_{\widetilde t_i^{}} \big )^ = \big(\xi_i^{}, \xi_i^{}
\big )  \textnormal{ for all } i = 1, \ldots, k   \Big ) \\
& \quad +
\P^{(N)} \Big ( \big ( \widehat \Pi^{(N,\varepsilon)}_{t_i^{}},
  \widecheck \Pi^{(N,\varepsilon)}_{t_i^{}} \big ) = \big(\xi_i^{}, \xi_i^{}
\big )  \textnormal{ for all } i = 1, \ldots, k   \Big ).
\end{split}
\end{equation*}
recalling that we write 
$\widetilde t = \lfloor t / c_N^{} \rfloor$ for the rescaled time.
By our computations in this subsection and 
Lemma~\ref{lem:annealedlimitgenerator}, the right hand converges as $N \to \infty$
to
\begin{equation} \label{eq:80r}
\begin{split}
& {\bf P} \big ( 
\Pi^{\infty,\mathsf{pure}}_{t_i^{}} = \xi_i^{} 
\textnormal{ for all } i = 1,\ldots,k
\big ) \\
& \quad -2
{\bf P} \big ( 
\Pi^{\infty,\mathsf{mix},\varepsilon}_{t_i^{}} = \xi_i^{} 
\textnormal{ for all } i = 1,\ldots,k
\big ) \\
& \quad +
{\bf P} \big ( 
\Pi^{\infty,\mathsf{pure},\varepsilon}_{t_i^{}} = \xi_i^{} 
\textnormal{ for all } i = 1,\ldots,k
\big )
\end{split},
\end{equation}
where $\Pi^{\infty,\mathsf{pure}}$, $\Pi^{\infty,\mathsf{mix},\varepsilon}$ and $\Pi^{\infty,\mathsf{pure},\varepsilon}$ are 
continuous-time Markov chains with transition rates 
$q_{\xi,\eta}^{\mathsf{pure}}$, $q_{\xi,\eta}^{\mathsf{mix},\varepsilon}$ and 
$q_{\xi,\eta}^{\mathsf{pure},\varepsilon}$, respectively.
By Lemmas \ref{lem:qpure}, \ref{lem:mixedrates} and \ref{lem:qpureepsilon}, 
\begin{equation*}
\lim_{\varepsilon \to 0} q_{\xi,\eta}^{\mathsf{mix},\varepsilon} 
=
\lim_{\varepsilon \to 0} q_{\xi,\eta}^{\mathsf{pure},\varepsilon}
=
q_{\xi,\eta}^{\mathsf{pure }},
\end{equation*}
which implies that \eqref{eq:80r} vanishes for
$\varepsilon \to 0$.
\end{proof}

\section{Examples}\label{sec:examples} 

In this section, we apply our main theorem, Theorem \ref{thm:fddconvergence}, to various
examples of the diploid Cannings model. These include the examples in \cite[Section 2]{BirknerEtAl2018} where the corresponding annealed limits ($n$-$\Xi$-coalescent processes) as $N\to\infty$ were computed. We refer the reader to \cite[Section 2]{BirknerEtAl2018} for details and motivation of these examples. 

\subsection{Diploid Wright-Fisher model} 
\label{ex:diploidWF}

Our first example is the diploid Wright-Fisher model with no selfing,  where each individual is independently assigned two distinct parents, chosen uniformly at random. 
The offspring numbers therefore follow a multinomial distribution with uniform weights:
\begin{equation*}
  \left(V_{i,j}\right)_{1 \leqslant i < j \leqslant N} \mathop{=}^d {\rm Multinomial}\left(N;\;\frac{2}{N(N-1)}, \cdots, \frac{2}{N(N-1)}\right).
\end{equation*}
This model is a special case considered in \cite{Mohle1998a} when the selfing rate is zero.
In this case, $c_N = 1/(2N)$ for all $N\geqslant 1$, so \eqref{eq:limcN=0} holds. Furthermore, \eqref{eq:PhiNconv} and Theorem \ref{thm:fddconvergence} hold with $\Xi'=\Xi=\delta_0$ and
$c_{\textnormal{pair}}^{} = 1$. That is, the quenched law is Kingman's coalescent. This special case of Theorem \ref{thm:fddconvergence} has been proved in~\cite{TyukinThesis2015} as mentioned in Remark \ref{rmk:tyukin}.

\subsection{Diploid population model with random individual fitness}
\label{ex:diploidind}

This and the example in the next section include diploid versions of the (haploid) model of \cite{schweinsberg2003coalescent}.  These can produce two distinct forms of ``diploid Beta coalescents'' as $N\to\infty$.

Let $(W_i)_{1 \leqslant i \leqslant N}$ be i.i.d. copies of a nonnegative random variable $W$ with $\E[W]>0$.
Suppose that, given 
$(W_i)_{1 \leqslant i \leqslant N}$, the offspring numbers follow the following  multinomial distribution:
\begin{equation*}
  \left(V_{i,j}\right)_{1 \leqslant i < j \leqslant N} \mathop{=}^d {\rm Multinomial}\left(N;\;\frac{W_1W_2}{Z_N}, \frac{W_1W_3}{Z_N} \cdots, \frac{W_{N-1}W_N}{Z_N}\right),
\end{equation*}
whenever $Z_N:=\sum_{1\leqslant i<j\leqslant N}W_iW_j>0$. Section 2.1 of \cite{BirknerEtAl2018} considers separately the finite variance case and the case of (strictly) regularly varying tail.

In the first case where it is assumed that $\E[W^2]<\infty$, $c_N$ is of order $N^{-1}$ and  Theorem \ref{thm:fddconvergence} holds with $\Xi'=\Xi=\delta_0$  and
$c_{\textnormal{pair}}^{} = 1$.  Here, as for the model in Section~\ref{ex:diploidWF}, the quenched law is Kingman's coalescent.

In the second case, it is  assumed that $\P(W\geqslant z)\sim c_{W}\,z^{-\alpha}$ as $z\to\infty$, where $c_{W}\in(0,\infty)$ and $\alpha\in(1,2)$ are constants.
This model is a diploid version of that in \cite{schweinsberg2003coalescent}.
In this case,  $c_N$ is of order $N^{1-\alpha}$ and  Theorem \ref{thm:fddconvergence} holds with 
\begin{equation}\label{2foldBeta}
\Xi= \int_{[0,1]} \delta_{(z/4, \,z/4,\, 0,0,0,\cdots)}\,{\rm Beta}(2-\alpha,\alpha)(\dd z),
\end{equation}
where the ${\rm Beta}(2-\alpha,\alpha)(\dd z)$ measure has density
\[
\frac{1}{\Gamma(2-\alpha)\Gamma(\alpha)}z^{1-\alpha}(1-z)^{\alpha-1}, \quad z\in (0,1).
\]
Therefore, $c_{\textnormal{pair}}^{} =0$ and hence there is no additional Kingman part in the quenched limit, as described in Definition~\ref{def:inhomcoalescentfullygeneral}.
The quenched limit in Theorem \ref{thm:fddconvergence} is determined solely by the large paintbox-merging events associated with a Poisson point process $\Psi$ on 
$[0,\infty) \times ( \Delta \setminus \{ \mathbf{0} \}  )$ with intensity
$\dd t \, \frac{1}{\langle x,x \rangle} \Xi (\dd x)$. Note that  $\Psi\in \cN_{\mathsf{loc}}^2$ almost surely, so the inhomogeneous $(\Psi,c_{\textnormal{pair}}^{})$-coalescent 
is well-defined. 
Also, note that since all atoms of $\Psi$ are of the form $\delta_{(z/4,z/4,0,0,0,\ldots)}$, $\Psi$ could alternatively be represented as a Poisson point process on
$[0,\infty) \times [0,1]$ with intensity 
$\dd t  {\rm Beta}(2-\alpha,\alpha)(\dd z)$. 

\subsection{Supercritical Galton-Watson process for parent couples}
\label{ex:GWforcouples}

In this example, we consider parent \textit{couples} that may produce a large family, unlike the previous example in which \textit{individual} parents may have many offspring due 
to unusually large fitness. This will lead to a different type of Beta-coalescent as $N\to\infty$, in which $\delta_{(\frac{z}{4},\frac{z}{4},0,0,0,\cdots)}$ in \eqref{2foldBeta} would change to $\delta_{(\frac{z}{4},\frac{z}{4},\frac{z}{4},\frac{z}{4},0,0,0,\cdots)}$. This change
corresponds to a change of the large paintbox-merging events, from 2 chromosomes of a single parent with many offspring to 4 chromosomes of a parent couple that produces a large family. Also in this example, $\Psi$ is equivalent to a Poisson point process on $[0,\infty) \times [0,1]$. 

To be more precise, let $X^{(N)}_{i,j}, 1\leqslant i\leqslant j\leqslant N$, be the ``potential offspring" of parents $i$ and $j$. Suppose $X^{(N)}_{i,j}=X^{(N)}_{j,i}$ and $X^{(N)}_{i,i}=0$ (no selfing). The  next generation is formed by sampling $N$ of all potential offspring at random without replacement (whenever $\sum_{i<j}X^{(N)}_{i,j}\geqslant N$). Then, we let $V_{i,j}$ be the number of offspring of parents $i$ and $j$ among that sample.

Assume that $(X^{(N)}_{i,j})_{1\leqslant i\leqslant j\leqslant N}$ are i.i.d. with $\P(X^{(N)}_{i,j}>0)\sim c\,N^{-1}$ as $N\to\infty$, where $c\in(0,\infty)$ is a constant. Assume also that the conditional law
$\mathcal{L}(X^{(N)}_{i,j}\,|\, X^{(N)}_{i,j}>0)$ is equal to the law $\mathcal{L}(X)$ of a (nonnegative) random  variable $X$ that does not depend on $N$ and satisfies $\E[X]\in (2/c,\,\infty)$.

Then results analogous to those of the previous example hold. Namely, if $\E[X^2]<\infty$, then 
 $c_N$ is of order $N^{-1}$ and  Theorem \ref{thm:fddconvergence} holds with $\Xi'=\Xi=\delta_0$  and
$c_{\textnormal{pair}}^{} = 1$, i.e. the quenched law is Kingman's coalescent.
Suppose, instead, $\P(X \geqslant k)\sim c_{X}\,k^{-\alpha}$ as $k\to\infty$, where $c_{X}\in(0,\infty)$ and $\alpha\in(1,2)$ are constants. 
This model is another diploid version of that in \cite{schweinsberg2003coalescent}.
In this case, $c_N^{}$ is of order $N^{1-\alpha}$ and  Theorem \ref{thm:fddconvergence} holds with 
\begin{equation}\label{4foldBeta}
\Xi= \int_{[0,1]} 
\delta_{(z/4, \, z/4, \, z/4, \, z/4, \, 0,0,0,\cdots)}
\,{\rm Beta}(2-\alpha,\alpha)(\dd z).
\end{equation}

\subsection{Occasional large families of a couple} 
\label{ex:HRcouple}

Consider a diploid population that involves rare but large reproduction events due to a single pair of distinct individuals that have a large number of offspring, such as the model from~\cite{BirknerEtAl2013a} or that from~\cite{DiamantidisEtAl2024}. 

The model from~\cite{BirknerEtAl2013a}, also discussed in \cite[Section 2.3.2]{BirknerEtAl2018}, is as follows. At each generation we randomly choose two distinct individuals
$\{I_1, I_2\}$ from $[N]$ to form a couple. This couple has a random number $\Psi_N$ of children together but with no one else, i.e. $V^{\mathsf{hr}}_{I_1,I_2}=\Psi_N$ and $V_{I_i,j}^{\mathsf{hr}}=0$ for $j\neq I_{3-i}$, $i=1,2$. The other $N-2$ individuals $[N]\setminus \{I_1,I_2\}$ in the parent generation give birth to $N-\Psi_N$ children according to the diploid Wright-Fisher model with no selfing. Formally, the offspring numbers are therefore given (for $i \neq j$) by
$V_{i,j}^{} = V_{i,j}^{\mathsf{hr}} + V_{i,j}^{\mathsf{wf}}$
where $V_{i,j}^{\mathsf{wf}}$ is distributed like $V_{i,j}^{}$ in Section~\ref{ex:diploidWF},  
with $N-\Psi_N$ trials and uniform success probabilities $1/(N-2)(N-3)$.
The limiting behavior depends on the laws 
$\big ( \mathcal{L}(\Psi_N)\big )_{N}$. 
Suppose that there exist  constants $\psi\in(0,1)$ and $\gamma>0$ such that
\begin{equation}\label{law_PsiN}
\P(\Psi_N=\lfloor\psi N\rfloor) =1- \P(\Psi_N=1) =N^{-\gamma}.
\end{equation}
This means the large family occurs with a small probability $N^{-\gamma}$ in each generation and constitutes a fixed fraction $\psi$ of the entire population. The larger the $\gamma$, the rarer the large family.

In this case,  Theorem \ref{thm:fddconvergence} holds and the limiting behavior of the conditional coalescent depends on whether $\gamma <1$, $\gamma=1$ or $\gamma>1$. Note that 
$c_N^{}$ is of the order $N^{-(\gamma \wedge 1)}$.
\begin{itemize}
\item[I.]When $\gamma <1$, the large family is not so rare,  Theorem \ref{thm:fddconvergence} holds with 
\[
\Xi=\delta_{(\psi/4,\, \psi/4,\, \psi/4,\, \psi/4,\,0,0,0,\cdots)}.
\]
In particular, the quenched limit has no Kingman part and is determined solely by the large paintbox-merging event as follows: according to a Poisson process with rate $\int_{ \Delta \setminus \{ \mathbf{0} \}  }\frac{1}{\langle x,x \rangle} \Xi (\dd x)=\frac{4}{\psi^2}$, it performs a $(\frac{\psi}{4},\frac{\psi}{4},\frac{\psi}{4},\frac{\psi}{4},0,0,0,\cdots)$-merger as described around \eqref{eq:paintboxtransitionprime}.

\item[II.]When $\gamma =1$, Theorem \ref{thm:fddconvergence} holds with 
\[
\Xi=\frac{\psi^2}{\psi^2+2}\delta_{(\psi/4,\, \psi/4,\, \psi/4,\, \psi/4,\,0,0,0,\cdots)}+\frac{2}{\psi^2+2}\delta_{\bf 0}
\]
which is a mixture of the other two cases (interpolated between the case $\psi\to \infty$ and $\psi\to 0$). 
The quenched limit can be more explicitly described as follows: according to a Poisson process with rate $\int_{ \Delta \setminus \{ \mathbf{0} \}  }\frac{1}{\langle x,x \rangle} \Xi (\dd x)=\frac{\psi^2}{\psi^2+2}\frac{4}{\psi^2}=\frac{4}{\psi^2+2}$, it performs a $(\frac{\psi}{4},\frac{\psi}{4},\frac{\psi}{4},\frac{\psi}{4},0,0,0,\cdots)$-merger, and independently, pairwise coalescence occurs at rate $c_{\textnormal{pair}}^{} =\frac{2}{\psi^2+2}$. 

\item[III.]When $\gamma >1$, the large family is extremely rare and the quenched limit is Kingman's coalescent.
\end{itemize}
Note that the above limits are obtained after rescaling time so that one unit of time is about $1/c_N$ generations. In particular, when $\gamma=1$, 
 the (annealed) rate of coalescence  for any fixed pair of lineages is exactly $\frac{4}{\psi^2+2}\frac{\psi^2}{4}+c_{\textnormal{pair}}=1$.
The model from~\cite{DiamantidisEtAl2024} is slightly different and  time was rescaled so that one unit of time corresponds to roughly $N^{\gamma \wedge 1}$ many generations, but we anticipate that Theorem \ref{thm:fddconvergence} would still hold true and the quenched limit  should be the same up to time-change by a constant factor.

\subsection{Occasional large families of an individual} \label{ex:HRindiv}

Suppose, instead of choosing a couple $\{I_1, I_2\}$ that gives rise to $\Psi_N$ many children as in Section~\ref{ex:HRcouple}, we choose a single individual $I_1$ which randomly mates with (possibly many different) partners. That is,
at each generation we choose a uniformly distributed individual 
$I_1 \in [N]$, specify that this individual has a random number $\Psi_N$ of children and that the other parents of each of these children are chosen uniformly with replacement
from 
$[N] \setminus \{I_1\}$. The remaining  $N-\Psi_N$ children are produced 
according to the diploid Wright-Fisher model  with potential parents among 
$[N] \setminus \{I_1\}$.

Suppose $(\Psi_N )_N$ satisfies \eqref{law_PsiN}. Then  Theorem \ref{thm:fddconvergence} holds and the quenched limit is the same as in Section \ref{ex:HRcouple}, but with 
$\delta_{(\psi/4,\, \psi/4,\, \psi/4,\, \psi/4,\,0,0,\cdots)}$ 
replaced by  $\delta_{(\psi/4,\, \psi/4,\,\,0,0,0,0,\cdots)}$. 
The models in Sections \ref{ex:HRcouple} and \ref{ex:HRindiv} can be extended to 
incorporate different numbers of potential parents (not necessarily exactly two parents); see \cite[Section 2.3.2]{BirknerEtAl2018}. 

\subsection{A two-sex version of the model} 
\label{Ex:two-sex}

\cite[Remark~2.3.3]{BirknerEtAl2018} give a brief sketch of how one can embed a
two-sex model into the set-up from Section~\ref{SS:Cannings}, i.e., define a matrix $(V_{i,j})$ satisfying \eqref{eq:fixedN} and \eqref{eq:exchangeable}.  Here we explain this in more detail: Let $0<r<1$
and assume that there are two sexes, called $1$ and $2$. Any given
generation consists of $\lfloor r N \rfloor$ individuals of sex\ $1$
and $N - \lfloor r N \rfloor$ individuals of sex\ $2$. With a small
abuse of notation, we will write $rN = \lfloor r N \rfloor$ and
$(1-r)N = N - \lfloor r N \rfloor$ in the following.  Sexes will be
assigned randomly to the $N$ children forming the offspring generation
(by drawing uniformly a subset of size $rN$ from $[N]$ and declaring that
individual $i$ has sex~$1$ if and only if $i$ belongs to that subset
and sex~$2$ otherwise).

Let $(O_{k,\ell}) = (O_{k,\ell})_{k=1,\dots,rN; \ell=1,\dots,(1-r)N}$ be a separately exchangeable
$\N_0$-valued array, i.e.\
\[
  (O_{k,\ell})_{k=1,\dots,rN; \ell=1,\dots,(1-r)N} \mathop{=}^d (O_{\pi(k),\pi'(\ell)})_{k=1,\dots,rN; \ell=1,\dots,(1-r)N}
\]
for any permutation $\pi$ of $rN$ and $\pi'$ of $(1-r)N$. 
We
interpret $O_{k,\ell}$ as the number of joint offspring of the $k$-th
individual of sex~$1$ and the $\ell$-th individual of sex~$2$. By
assigning the sexes randomly, we can 
read $(O_{k,\ell})$ as a
specific choice for the law of $(V_{ij})$ in \eqref{eq:exchangeable} as follows.

Let $R=(R_1,\dots,R_N)$ be uniformly distributed on
$\{ (a_1,\dots,a_N) \in \{1,2\}^N : \sum_{i=1}^N \1(a_i=1) = rN \}$ and independent of 
$(O_{k,\ell})$. Put
\[
  F_a(i) \defeq \sum_{m=1}^i \1(R_m=a), \quad a \in \{1,2\}, \, 1 \leqslant i \leqslant N
\]
(with sex assignment encoded by $R$, $F_{R_i^{}}(i)$ is the index of
individual $i \in [N]$ within its sex group).
Then it only remains to define for $1\leqslant i \leqslant j \leqslant N$
\begin{equation}
  V_{i,j} \defeq
  \begin{cases}
    0, & \text{if } R_i = R_j,\\
    O_{F_1(i),F_2(j)}, & \text{if } R_i =1, R_j=2, \\
    O_{F_1(j),F_2(i)}, & \text{if } R_i =2, R_j=1.
  \end{cases}
\end{equation}
This way of encoding sexes of individuals works well for the description of the transmission of autosomal genetic information. However, if one wishes to explicitly include sex-linked genetic information, one would have to enhance the notation and set-up accordingly.

A more explicit, simple example along the lines of \cite{DiamantidisEtAl2024} but
with asymmetric sex roles could look as follows: let $0<\beta<1$,
$\lambda>0$, put $\alpha_N^{} = \lambda/N$. With probability
$1-\alpha_N^{}$, $(O_{k,\ell})$ has a symmetric multinomial distribution
$\mathrm{Multinomial}(N, u^{(N)})$ where $u^{(N)}=(u^{(N)}_{k,\ell})$
has $u_{k,\ell}^{(N)}=1/(rN \cdot (1-r)N)$. This corresponds to assigning
to each child uniformly a parent of sex~$1$ and a parent of sex~$2$,
as in a two-sex Wright-Fisher model.  With probability $\alpha_N^{}$,
there is a highly reproductive individual of sex~$1$ and we put
$(O_{k,\ell}) \defeq (O'_{k,\ell}) + (O''_{k,\ell})$ with
$(O'_{k,\ell}) \sim \mathrm{Multinomial}(N-\lfloor \beta N \rfloor;
u^{(N)})$ and independently $(O''_{k,\ell})$ has only one non-zero row
whose position is uniform on $\{1,2,\dots,rN\}$ 
and whose entries are
$\mathrm{Multinomial}(\lfloor \beta N \rfloor; 1/((1-r)N), \dots,
1/((1-r)N))$. One randomly chosen individual of sex~$1$ has
$\beta N$ offspring with randomly chosen partners, the other children
are created as in a two-sex Wright-Fisher model.

The quenched limit in this example looks similar to the one in Section~\ref{ex:HRindiv}. More precisely, we have
\begin{equation*}
\Xi = \frac{\lambda \beta^2 r (1-r)}{1+ \lambda \beta^2 r(1-r)} \delta_{(\beta/2, \, \beta/2, \, 0, 0, \ldots )} + 
\frac{1}{1+ \lambda \beta^2 r(1-r)} \delta_{{\mathbf 0}}
.
\end{equation*}
Note that taking $\lambda = 2$, $r=1/2$ and $\beta = \psi$, this agrees precisely with Section~\ref{ex:HRindiv}.



\begin{remark}[Diploid Moran model]
The haploid Cannings model covers the (haploid) Moran model mathematically, provided that one suitably re-interprets $g$ as  a time-step of the Markov chain so that one generation corresponds to $N$ time-steps, and some  individuals at different time-steps are considered to be the same. 
We expect something analogous to hold for diploid Moran models, like the ones considered in \cite{linder2009,coron2022pedigree,NewmanEtAl2025,fan2025quenched}.

Consider the diploid Moran model without selfing in \cite{coron2022pedigree}.
At each time-step $g$, a randomly chosen individual dies and is replaced by a single offspring produced by two distinct parents in the previous time-step $g+1$, and all the other individuals are considered to be the same as those in the previous time-step (hence generations overlap). In the formulation of \cite{coron2022pedigree}, each update is specified by a uniformly distributed triple $(\pi_1^{},\pi_2^{},\kappa) \in [N]^3$, containing the two parents 
$\pi_1^{} \neq \pi_2^{}$ and an offspring $\kappa$. 

In this case, 
\begin{equation*}
\begin{cases}
V_{\pi_1,\pi_2}=1 & \text{ (the single offspring) }\\
V_{i,i}=1 &\text{ for all } i\in[N] \setminus \{\kappa\} \text{ (the same individuals in consecutive time-steps)},\\
V_{i,j}=0 & \text{ otherwise}.
    \end{cases}
\end{equation*}

In contrast, our diploid Cannings model assumes that $V_{i,i}=0$ and that Mendel’s law applies to all individuals, not only to the offspring $\kappa$.  Nonetheless, we conjecture that Theorem \ref{thm:fddconvergence} may still hold for a suitably adjusted version of our model, with the quenched limit being Kingman's coalescent when time is measured proportional to $N^2$ time-steps.\hfill \mbox{$\Diamond$}
\end{remark}


\section{Illustrating the effects of pedigrees on data}
\label{sec:data}

To show some of the ways population pedigrees shape patterns of genetic variation, we simulated gene genealogies for multiple unlinked loci conditional on the pedigree.  We used the model in Section~\ref{ex:HRindiv}, in which large families are composed of half siblings with one highly successful parent.  We assumed a sample of size $n=100$, where a single gene is taken from each of $100$ individuals.  We implemented a two-parameter version of the model with $\gamma=1$, that is with large families and pair-merger events occurring on the same time scale.  The first parameter, $\psi$, is again the fraction of the population replaced by each large family.  The second parameter, $\lambda$, is a scalar determining the relative rate of large families compared to pair merger events, cf.\ \cite{DiamantidisEtAl2024}.

In this model, the pedigree is just a list of times of large families. In between, the evolution follows that of Kingman's coalescent.  At each large-family event, the remaining blocks trace back to one of the highly successful parent's two chromosomes, each with probability $\psi/4$, or escape the event with probability $1-\psi/2$.  In the latter case, a block is either not in the large family (probability $1-\psi$) or traces back to one of the highly successful parent's many partners (probability $\psi/2$).  

Figure~\ref{fig:sfs} shows the site-frequency spectrum (SFS) for each of five  simulated pedigrees for three values of $\psi$ and two values of $\lambda$.  When $\lambda=10^6$, nearly all coalescence is due to large families and there is a relatively strong pedigree effect, unless $\psi$ is small.  Broadly, pedigrees under this model have two effects on the SFS.  First, they modulate the relative number of low-frequency polymorphic sites, the lower curves having more of these.  Singleton polymorphisms in particular will be abundant if the time back to the first large family is long.  Second, they create bumps in the SFS at intermediate frequencies.  When $\lambda$ is large, the first event in the ancestry will comprise two mergers, each with expected size $n\psi/4$ lineages.  Intuitively, the corresponding bump in the SFS will be pronounced if the time to the second event in the ancestry is long.  When $\lambda=1$, most coalescence is pairwise coalescence owing to the very fast rate ($\propto n^2$) of these in the ancestry.  The relatively small number of large families that occur in the ancestry when the number of blocks is already small due to binary mergers.  The SFS is not much different than that of  Kingman's coalescent, though a pedigree effect may occasionally be seen when $\psi$ is large.  

\begin{figure}[!ht]
\center
\includegraphics[scale=0.70]{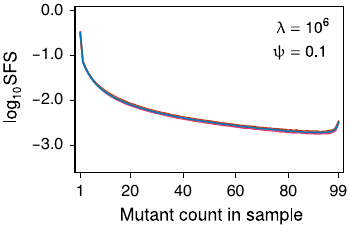}$\quad$\includegraphics[scale=0.70]{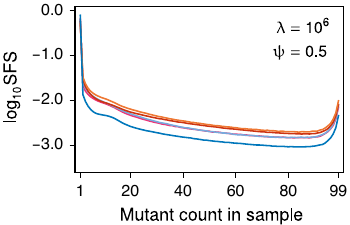}$\quad$\includegraphics[scale=0.70]{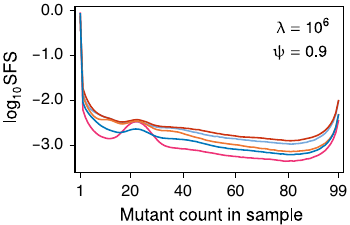}\\
\vspace{8pt}
\includegraphics[scale=0.70]{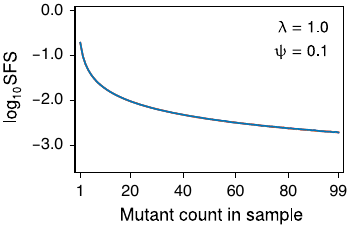}$\quad$\includegraphics[scale=0.70]{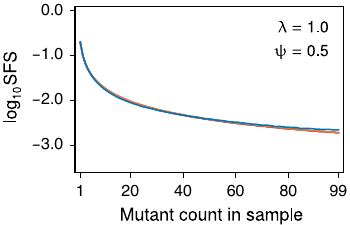}$\quad$\includegraphics[scale=0.70]{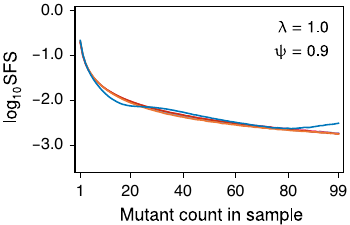}
\caption{\small Log base $10$ of the expected SFS, specifically the expected proportion of polymorphic sites for each possible count of the mutant base in a sample of size $n=100$, for three values of $\psi\in\{0.1,0.5,0.9\}$ under the $\delta_{(\frac{\psi}{4},\frac{\psi}{4},0,0,\cdots)}$ model with a Kingman component, in which $\lambda$ is the rate of $\delta_{(\frac{\psi}{4},\frac{\psi}{4},0,0,\cdots)}$ events relative to pair merger events.  In the top row ($\lambda=10^6$) large reproduction events dominate pair merger events.  In the bottom row ($\lambda=1.0$) both occur at the same rate.  The five lines in each panel are the SFS for five independent pedigrees (lists of times of large families).  For each of these, the SFS was estimated from the simulated gene genealogies of one million unlinked loci.}\label{fig:sfs}
\end{figure}

Figure~\ref{fig:Ttotal} shows the effects of pedigrees on the total length of the gene genealogy, in the case where large reproduction events dominate pair merger events ($\lambda=10^6$).  We quantified the importance of the pedigree effect using the law of total variance,
\begin{equation}
\textrm{Var}{\left(T_\textrm{total}\right)} = \mathbb{E}{\left[\textrm{Var}{\left(T_\textrm{total}|G\right)}\right]} + \textrm{Var}{\left(\mathbb{E}{\left[T_\textrm{total}|G\right]}\right)} \label{eq:totalVarTtotal}
\end{equation}
in which the term $\textrm{Var}{\left(\mathbb{E}{\left[T_\textrm{total}|G\right]}\right)}$ represents the portion explained by the pedigree.  The left panel of Figure~\ref{fig:Ttotal} plots the two terms on the right-hand side of \eqref{eq:totalVarTtotal} as fractions of the total for a series of values of $\psi$.  We did this by computing $\textrm{Var}{\left(T_\textrm{total}\right)}$ for a single locus on $10^5$ simulated pedigrees, and $\mathbb{E}{\left[\textrm{Var}{\left(T_\textrm{total}|G\right)}\right]}$ for $20000$ loci on $2000$ simulated pedigrees. On the far left, as $\psi$ approaches zero, the particular features of the pedigree have little effect and the bulk of the variation is among loci given the pedigree.  The opposite is true on the far right, as $\psi$ approaches one.  But even if $\psi=1$, variation among loci given the pedigree does not decrease to zero because ancestral lineages do not necessarily coalesce when a large family occurs.  When $\psi=1$, each ancestral lineage still has probability $1/2$ of not even tracing back to the highly successful individual.

\begin{figure}[!ht]
\centering
\includegraphics[scale=0.7]{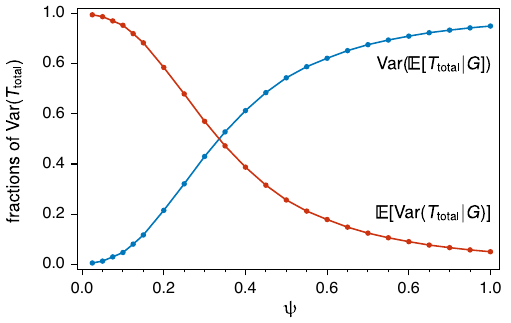}$\qquad$\includegraphics[scale=0.7]{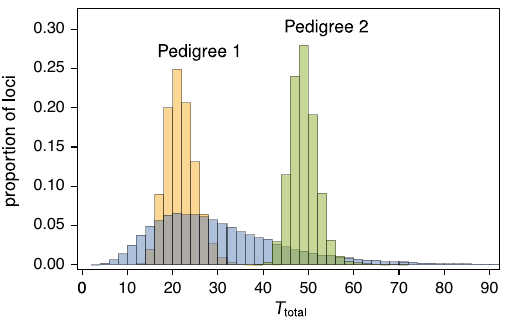}
\caption{\small Effects of pedigrees on the total length of the gene genealogy for $n=100$ under the $\delta_{(\frac{1}{4},\frac{1}{4},0,0,\ldots)}$ model with a negligible Kingman component ($\lambda=10^6$).  Left: Relative importance of the two sources of variation in $T_\textrm{total}$ among loci from \eqref{eq:totalVarTtotal} expressed as fractions of the total variance, for values of $\psi$ ranging from $0.025$ to $1$.  The contributions sum to one, but both are displayed for illustration.  Right: Distributions of $T_\textrm{total}$ among $50000$ unlinked loci given each of two randomly generated pedigrees for the case $\psi=1$, compared to the corresponding distribution (in blue) for the annealed model, i.e.\ the distribution of $T_\textrm{total}$ at a single locus among $50000$ randomly generated pedigrees.}\label{fig:Ttotal}
\end{figure}

The right panel of Figure~\ref{fig:Ttotal} displays the pedigree effect on the distribution of $T_\textrm{total}$ for the extreme case, $\psi=1$.  The two different simulated pedigrees have distinct means of $T_\textrm{total}$ and much narrower distributions than the corresponding distribution among pedigrees (shown in blue).  From these simulated data we have $\widehat{\textrm{Var}}{\left(T_\textrm{total}|G_1\right)}=10.5$, $\widehat{\textrm{Var}}{\left(T_\textrm{total}|G_2\right)}=9.0$ and $\widehat{\textrm{Var}}{\left(T_\textrm{total}\right)}=227.3$ as estimates.  These give a rough estimate of $0.043$ for the fraction of the total variance due to variance given the pedigree which even with only two pedigrees is comparable to the value ($0.051$) for $\psi=1.0$ in the left panel.

\section{Further remarks and possible extensions}
\label{sect:extensions}

Conditional on the population pedigree in the diploid biparental model of reproduction of \cite{BirknerEtAl2018}, Theorem~\ref{thm:fddconvergence} establishes a new limiting inhomogeneous coalescent process in which simultaneous multiple mergers occur at fixed times determined by past outcomes of high reproductive success.  It can also include a Kingman component if the pedigree includes enough small reproduction events, or be fully characterized by Kingman's coalescent if large reproduction events become negligible in the limit.  These results, like those for the simple illustrative model in \cite{DiamantidisEtAl2024}, have consequences for coalescent-based inference and simulation as well as for future modeling.  

Figures~\ref{fig:sfs} and~\ref{fig:Ttotal} illustrate the potentially dramatic effects of pedigrees on the site-frequency spectrum and the total length of the gene genealogy, which is a key quantity affecting the overall level of genetic variation. Again, Figures~\ref{fig:sfs} and~\ref{fig:Ttotal} depict variation among loci in the same sample of individuals on a single shared pedigree.  The differences among pedigrees in these figures are entirely due to differences in the times at which large reproduction events have occurred in the ancestry of the population.  They suggest that multilocus genetic data may contain signatures of particular large reproduction events in the past.  Thus, inference shifts to the times and sizes of fixed events in the past, away from or in addition to the population-level parameters important in establishing those events, such as the rate at which large families occur.  For some additional discussion of these issues, see \cite[page~15]{DiamantidisEtAl2024}.

Although Theorem~\ref{thm:fddconvergence} also proves that it is not necessary to keep track of the entire pedigree, even just accounting for a series of times and sizes of large reproduction events presents challenges for analysis and simulation.  Predictions for various measures of genetic variation should be straightforward to compute, at least in principle.  But recall that the chance of a pairwise coalescent event in a generation with a large family whose offspring comprise a fraction $x$ of the population is just $x^2/4$, so very many small-$x$ events will need to be accounted for in generating such predictions.  Similarly, simulations of gene genealogies will be slow when small-$x$ events are frequent.  

For example, populations with highly skewed distributions of offspring numbers, in the traditional (unconditional) approach modeled using a Beta-coalescent \cite{schweinsberg2003coalescent} (e.g.\ as arising from the examples in Sections~\ref{ex:diploidind} and \ref{ex:GWforcouples}), may present a computational challenge under the new conditional framework.  The $\mathrm{Beta}(2-\alpha,\alpha)$ measure for $\alpha \in (1,2)$ has probability density
\begin{equation}
  \label{betadensity}
  \frac{1}{\Gamma(2-\alpha)\Gamma(\alpha)} x^{1-\alpha}(1-x)^{\alpha-1}, \quad 0 < x < 1, 
\end{equation}
which is of course manageable.  But the Poisson point process
\begin{equation}
  \label{PPP}
  \Phi = \sum_{j} \delta_{(x_j,t_j)}
\end{equation}
on $(0,1] \times [0,\infty)$ underlying the construction of a
$\mathrm{Beta}(2-\alpha,\alpha)$-coalescent has intensity measure
$\nu(\dd x) \dd t$ where $\nu$ has density
\begin{equation}
  \label{betadensity2}
  \frac{1}{\Gamma(2-\alpha)\Gamma(\alpha)} x^{-1-\alpha}(1-x)^{\alpha-1}, \quad 0 < x < 1 ,
\end{equation}
in other words $\nu(\dd x) = x^{-2} \mathrm{Beta}(2-\alpha,\alpha)(\dd x)$.   Since $\nu$ is an infinite measure, one will have to suitably truncate it for practical applications and thus replace \eqref{PPP} by an approximation which has only finitely many atoms, and possibly add an artificial ``effective'' pair merger rate. The required precision of the approximation will likely depend on the sample size. We leave these issues to future work.

\subsection{Sex-specific patterns of inheritance} Recall from Section~\ref{Ex:two-sex} that two-sex population models are in principle included in our framework as long as one focuses on an autosomal locus which is transmitted symmetrically by both sexes. We believe that a suitable adaptation of our results will continue to hold for a sex-linked locus, e.g.\ on an X chromosome. After all, which individual an ancestral gene resides in, and in particular also the sex of that individual, averages very quickly on the time scale $1/c_N$.

\subsection{Multi-locus genealogies and recombination} Our main result, Theorem~\ref{thm:fddconvergence}, describes the law of the gene genealogy of the sample at a single autosomal locus conditioned on the pedigree. It also describes the joint conditional law for two or more freely recombining loci - e.g.\ residing on different chromosomes - because then this conditional law is simply the product of the single-locus conditional law, as observed in Section~\ref{sec:data}. The situation is different for linked loci. In this situation, we expect a conditional analogue to the results in\ \cite{BirknerEtAl2013c}: Consider, say, two loci on the same chromosome, assume that each child inherits from each parent with probability $1-\rho c_N$ both loci from the same parental chromosome and with probability $\rho c_N$ the two loci from different parental chromosomes, where $\rho>0$ is a recombination parameter (we implicitly assume $\rho c_N < 1$, which always holds for large $N$). Then we expect under the assumptions of Theorem~\ref{thm:fddconvergence} the conditional law of the joint gene genealogies at the two loci to be approximately described by an inhomogeneous $(\Psi,c)$-ancestral recombination graph, where lineages which are ancestral for both loci split independently at rate $\rho$ in addition to the mechanisms under the $(\Psi, c)$-coalescent. We leave the details to a forthcoming work \cite{Alberti+al2026}. 

\subsection{Corresponding forward-time models} While backward-time coalescent models are well suited to the analysis of samples from natural populations, it is also interesting to ask about forward-time models that are connected to an (inhomogeneous) $(\psi,c)$-coalescent $\Pi^{\psi,c}$ from Definition~\ref{def:inhomcoalescentfullygeneral} via duality. For a given space $E$ of neutral ``types'', these are time-inhomogeneous Markov processes with values in the probability measures on $E$, intuitively they are ``quenched versions'' of $\Xi$-Fleming-Viot processes (e.g.\ \cite{BirknerEtAl2009}). For a formal construction, one could either construct their transition semigroups using the duality approach from \cite{Evans1997} or one could extend the ``look down'' construction from \cite{BirknerEtAl2009} to a time-inhomogeneous scenario.

In general, these would be processes with discrete jumps at particular times, potentially with Wright-Fisher diffusion in between.  Their predictions could be applied to data from laboratory evolution experiments like the famous one by \cite{Buri1956} using \textit{Drosophila melanogaster}.  Importantly, the same set of jumps with fixed times would affect all loci in the genome.  Already it is known for forward-time (``annealed'') versions of $\Lambda$-coalescents and $\delta_\psi$-coalescents that properties, including the time to fixation of a mutant under neutrality and the probability of fixation of a mutant under selection, differ quite dramatically from those of the Wright-Fisher diffusion \cite{DerEtAl2011,DerEtAl2012,BirknerEtAl2024,EldonAndStephan2024,CorderoEtAl2025}.  \cite{DerEtAl2011} applied such a model to the data of \cite{Buri1956} and suggested that even this small experiment ($107$ replicate populations of $8$ males and $8$ females over $19$ generations) provided some evidence of non-diffusive evolution.  Thus, the application of quenched and diploid (``$\Xi$'' rather than just ``$\Lambda$'') versions of these models is another interesting area for future work.

\begin{appendix} 
\section{Jump-hold characterisation of inhomogeneous coalescents}
\label{app:jumphold}

To complement our description of the inhomogeneous $(\psi,c)$-coalescent processes in terms of stochastic flows in Section~\ref{sec:inhomcoal}, we briefly outline an alternative description by providing the distribution of the first jump. This point of view is particularly useful for carrying out simulations.

Given a point configuration 
$\psi 
=
\sum_i \delta_{ (t^{(i)}, x^{(i)} )}^{}
\in \cN^2_{\mathsf{loc}}$, recall that the expression
\begin{equation}\label{eqq:noncoal.psi.s<t.0}
\prod_{i \, : \, s < t^{(i)} < u} 
p(x^{(i)}; \xi,\xi)
\end{equation}
is well-defined and strictly positive for
$0 \leqslant s < u < \infty$ and is the probability that, using the notation introduced in \eqref{eq:paintboxeff} in Section~\ref{sec:inhomcoal}, 
$\cC_\psi^{} \cap (s,u) = \varnothing$.
This means that no SG caused a coalescence event during the time interval $(s,u)$. By taking the logarithm, we can rewrite this as
\begin{equation} \label{eqq:Gnon.coal.s.u.psi}
  G^{\mathsf{nc}}_n(s,u ; \psi) := \exp\left( - \int_{(s,u) \times \Delta} g^{\mathsf{nc}}_n(x) \, \psi\big(\dd(t,x)\big) \right),
\end{equation}
where 
\begin{equation} \label{eqq:gnnc}
  g^{\mathsf{nc}}_n(x)
  = - \log\left( \sum_{\ell=0}^n \sum_{\substack{i_1,i_2,\dots,i_l \in \N \\ \text{pairw.\ diff.}}}
  \binom{n}{\ell} x_{i_1} x_{i_2} \cdots x_{i_\ell} \big( 1 - |x|_1^{} \big)^{n-\ell} \right),
  \quad x \in \Delta \setminus \{ \bf 0 \}.
\end{equation}
The superscript `$\mathsf{nc}$' stands for `non-coalescence',
note that $g^{\mathsf{nc}}_n(x) = -\log p(x; \xi^{(n)}_0, \xi^{(n)}_0)$ with
$\xi^{(n)}_0 \defeq \big  \{\{1\},\{2\},\dots,\{n\} \big \} \in \mathcal{E}_n$.

To take the independent pairwise `Kingman' mergers into account, we note that since those happen at rate $c$, the probability that none of them occur on the time interval $(s,u)$ is simply
$\exp \big (
-c \binom{n}{2} (u-s)
\big )$.
The probability that we do not see any coalescence event on the time interval
$(s,u)$ (either a pairwise or a multiple merger) is therefore given by
\begin{align} 
  \label{eqq:Gnon.coal.s.u.psi+pair}
  G^{\mathsf{nc}}_n(s,u ; c, \psi)
  :=
  G^{\mathsf{nc}}_n(s,u ; \psi) \exp\left( - c \textstyle{ \binom{n}{2} (u-s)} \right).
\end{align}
Assuming $\Pi_t^n = \xi$ and for any $\eta \in \cE_n$,
in order for us to observe the transition $\xi \to \eta$ at exactly time $u$, two things need to happen. First, there must be no coalescence on the time interval $(t,u)$. Secondly, we need to jump to $\eta$ at time $u$.
The probability that the latter happens by virtue of a large merger is 
\begin{equation*}
\sum_{i \geqslant 1} p(x^{(i)};\xi,\eta) 
\delta_{t^{(i)}}^{} (\dd u),
\end{equation*}
while the probability that this happens due to a pairwise merger is 
\begin{equation*}
c \1_{\xi \vdash \eta} \dd u.
\end{equation*}
Multiplying this with the non-coalescence probability in \eqref{eqq:Gnon.coal.s.u.psi+pair}, we obtain
\begin{equation*}
\begin{split}
& G_{\# \xi}^{\mathsf{nc}} (t,u;c,\psi) 
\Big(
c \1_{\xi \vdash \eta} \dd u 
+
\sum_{i \geqslant 1} 
p(x^{(i)};\xi,\eta) 
\delta_{t^{(i)}}^{} (\dd u)
\Big ) \\
& \quad =
c \1_{\xi \vdash \eta} 
G_{\# \xi}^{\mathsf{nc}} (t,u;c,\psi)  \dd u
+
\sum_{i \geqslant 1} 
G_{\# \xi}^{\mathsf{nc}} (t,u;c,\psi)
p(x^{(i)};\xi,\eta) 
\delta_{t^{(i)}}^{} (\dd u) \\
& \quad =
c \1_{\xi \vdash \eta} 
G_{\# \xi}^{\mathsf{nc}} (t,u;c,\psi)  \dd u
+
\sum_{i \geqslant 1} 
G_{\# \xi}^{\mathsf{nc}} (t,t^{(i)};c,\psi)
p(x^{(i)};\xi,\eta) 
\delta_{t^{(i)}}^{} (\dd u).
\end{split}
\end{equation*}
To summarise:
\begin{lemma}
\label{lem:inhomogeneouscoalescentjump}
For $c \geqslant 0$,
$\psi \in \cN_{\mathsf{loc}}^2$ and
$t \geqslant 0$, denote by 
$J^{\psi, c}_t := \inf\{ u > t : \Pi^{\psi,c}_u \neq \Pi^{\psi,c}_t \}$
the time of the first jump after time $t$.
The joint law of the first jump time and target state after time $t$ is given by 
\begin{align} \label{eqq:jumplaw}
  & {\bf P}\big( J^{\psi,c}_t \in \dd u, \,
    \Pi^{\psi,c}_{J^{\psi, c}_t} = \eta \,\big|\,  \Pi^{c, \psi}_t = \xi \big) \notag \\
  & = \sum_{i \, : \, t < t^{(i)} } G^{\mathsf{nc}}_{\#\xi}(t,t^{(i)} ; c, \psi) p(x^{(i)}; \xi,\eta) \delta_{t^{(i)}}(\dd u)
    \notag \\
  & \quad + c \1_{\xi \vdash \eta}  G^{\mathsf{nc}}_{\#\xi}(t,u ; c, \psi) \dd u,
\end{align}
\end{lemma}

\section{Proof of Lemma~\ref{lem:qpure}}
\label{app:lemmaproof}

\begin{proof}
We first treat the case where $\widehat s = 0 = \widecheck s$. We first fix a choice of ancestral genes for $\widehat \eta$ and $\widecheck \eta$, letting
$(\widehat c_1^{}, \widehat \imath_1^{}),\ldots,
(\widehat c_{\widehat r}^{}, \widehat \imath_{\widehat r}^{}) \in 
\{0,1\} \times [N]$ and 
$(\widecheck c_1^{}, \widecheck \imath_1^{}),\ldots,
(\widecheck c_{\widecheck r}^{}, \widecheck \imath_{\widecheck r}^{}) \in 
\{0,1\} \times [N]$
be respectively pairwise distinct. However, we do not assume 
\begin{equation*}
(\widehat c_1^{},\widehat \imath_1^{}),\ldots,
(\widehat c_{\widehat r}^{}, \widehat \imath_{\widehat r}^{}),
(\widecheck c_1^{}, \widecheck \imath_1^{}),\ldots,
(\widecheck c_{\widecheck r}^{},\widecheck \imath_{\widecheck r}^{}) \in 
\{0,1\} \times [N]
\end{equation*}
to be pairwise distinct. This is due to the summation in the definition of
$\widetilde \pi_{\mathsf{pure}}^{N}$, which amounts to ignoring the grouping of ancestral genes into diploid individuals as well as which genes are shared by the two coalescents.

To proceed,
we need to compute the probability that the $\widehat k_j^{}$ genes corresponding to the blocks of $\widehat \xi$ that make up the 
$j$-th block of $\widehat \eta$ have the ancestor 
$(\widehat c_j^{}, \widehat \imath_j^{})$
and that the $\widecheck k_j^{}$ genes corresponding to the blocks of $\widecheck \xi$ that make up the $j$-th block of $\widecheck \eta$ have the ancestor 
$(\widecheck c^{}_j, \widecheck \imath^{}_j)$. The same reasoning as in Section~\ref{subsec:transitionsconditionalonoffspringnumbers} shows that this probability is given by
\begin{equation}\label{eq:aowie}
\frac{1}{(N)_{\widehat b + \widecheck b \downarrow}} \E
\Bigg [
\prod_{\ell \in \{ \widehat \imath_1^{}, \ldots, \widehat \imath_{\widehat r}^{},
         \widecheck \imath_1^{}, \ldots, \widecheck \imath_{\widecheck r} 
\}  }
\big (  \widehat V_{\ell} \big )_{K_\ell \downarrow}
\Bigg ]
\cdot 
2^{-\widehat b - \widecheck b}.
\end{equation}
where 
\begin{equation} \label{eq:Kelldef}
K_\ell \defeq
\sum_{j = 1}^{\widehat r} \1_{\ell = \widehat \imath_j^{}} \widehat k^{}_j 
+
\sum_{j = 1}^{\widecheck r} \1_{\ell = \widecheck \imath_j^{}} 
\widecheck k^{}_j.
\end{equation}
is how often individual $\ell$ is picked as the $0$-parent.
Note that $K_\ell$ contains up to $4$ nontrivial contributions, a maximum of $2$ for each of the two sums.

The first factor in \eqref{eq:aowie} is the probability of the genes in question to choose the correct 
$0$-parents, and 
$2^{-\widehat b - \widecheck b}$ is the conditional probability of picking the correct one of the two potential ancestral genes in each case.

Summing over all choices of ancestral genes gives
\begin{equation} \label{eq:pitildepurewrittenout}
\begin{split}
&\frac{1}{c_N^{}} \widetilde \pi_{\textnormal{pure}}^{N} (\xi, \eta)
\\ & \quad =
\sum_{\substack{  
(\widehat c_1^{}, \widehat \imath_1^{}), \ldots, 
(\widehat c_{\widehat r}^{} , \widehat \imath_{\widehat r}^{}) 
\\
\textnormal{pairwise distinct}
}} \,
\sum_{\substack{  
(\widecheck c_1^{}, \widecheck \imath_1^{}), \ldots, 
(\widecheck c_{\widecheck r}^{} , \widecheck \imath_{\widecheck r}^{}) 
\\
\textnormal{pairwise distinct}
}}
\frac{2^{-\widehat b - \widecheck b}}{c_N^{} (N)_{\widehat b + \widecheck b \downarrow}} \E
\Bigg [
\prod_{\ell \in \{ \widehat \imath_1^{}, \ldots, \widehat \imath_{\widehat r}^{},
         \widecheck \imath_1^{}, \ldots, \widecheck \imath_{\widecheck r}^{} 
\}  }
\big (  \widehat V_{\ell} \big )_{K_\ell \downarrow}
\Bigg ],
\end{split}
\end{equation}
where both sums run over all pairwise distinct chromosome copies in 
$\{0,1\} \times [N]$.

We now rewrite the sum in \eqref{eq:pitildepurewrittenout}. First, we group the terms in the double sum according to the number $d \defeq \# \{\widehat \imath_1^{},\ldots,
\widehat \imath_{\widehat r}^{},
\widecheck \imath^{}_1,\ldots,\widecheck \imath_{\widecheck r}^{} \}$ of distinct individuals. Note that $d$ can take values from $d_{\textnormal{min}}^{} \defeq \lceil \widehat r/2 \rceil \vee \lceil \widecheck r / 2 \rceil$ to 
$d_{\textnormal{max}}^{} \defeq \widehat r + \widecheck r$. 

For each value of $d$, we have a choice of $d$ distinct indices 
$j_1^{}, \ldots, j_d^{} \in [N]$. In order to establish a  correspondence between the choice of
$j_1^{}, \ldots, j_d^{}$ and 
$\widehat \imath_1^{}, \ldots, \widehat \imath_{\widehat r}^{},\widecheck \imath^{}_1, \ldots, \widecheck \imath^{}_{\widecheck r}$, we furthermore choose two injective maps $\widehat \varrho : [\widehat r] \hookrightarrow \{0,1\} \times [d] $ and 
$\widecheck \varrho : [\widecheck r] \hookrightarrow 
\{0,1\} \times [d]$. 

Clearly, to each choice of $d$, $j_1^{},\ldots, j_d^{}$ and of $\widehat \varrho$ and $\widecheck \varrho$ we can uniquely associate 
pairwise distinct $(\widehat c_1^{}, \widehat \imath_1^{}),\ldots,(\widehat c_{\widehat r}^{}, \widehat \imath_{\widehat r}^{})$ and 
$(\widecheck c_1^{},\widecheck \imath^{}_1),\ldots,
(\widecheck c_{\widecheck r}^{},\widecheck \imath_{\widecheck r}^{})$ by letting 
$(\widehat c_\ell^{}, \widehat \imath_{\ell}^{}) \defeq (\widehat \varrho(\ell)_1^{}, j_{\widehat \varrho (\ell)^{}_2}^{})$ and 
$(\widecheck c_\ell^{}, \widecheck \imath_{\ell}^{}) \defeq
(\widecheck \varrho(\ell)_1^{}, j_{\widecheck \varrho (\ell)_2^{}}^{})$ for
$\ell \in [\widehat r]$ and $\ell \in [\widecheck r]$, respectively.

Conversely, given $(\widehat c_1^{}, 
\widehat \imath_1^{}),\ldots,
(\widehat c_{\widehat r}^{}, \widehat \imath_{\widehat r}^{})$ and 
$(\widecheck c_1^{}, \widecheck \imath^{}_1),\ldots,
(\widecheck c_{\widecheck r}^{}, \widecheck \imath_{\widecheck r}^{})$, we pick 
$d$ and
pairwise distinct $j_1^{}, \ldots, j_d^{} \in [N]$ so that
$\{ j_1^{}, \ldots, j_d^{} \} = 
\{ \widehat \imath_1^{}, \ldots, \widehat \imath_{\widehat r}^{}, \widecheck \imath^{}_1, \ldots, \widecheck \imath^{}_{\widecheck r}  \}$. 
The injections $\widehat \varrho$ and $\widecheck \varrho$ are then uniquely determined via $\widehat \varrho (\ell) \defeq ( \widehat c_\ell^{}, \widetilde \ell)$ for all $\ell \in [\widehat r]$, where $\widetilde \ell$ is the unique element of $[d]$ such that $j_{\widetilde \ell} = \widehat \imath_\ell^{}$. Analogously for all 
$\ell \in [\widecheck r]$, we define
$\widehat \varrho (\ell) \defeq (\widecheck c^{}_\ell, \widetilde \ell)$ with 
$\widetilde \ell$ being the unique element of $[d]$ such that
$j_{\widetilde \ell}^{} = \widecheck \imath_\ell^{}$. 

This works just as well when $[N]$ is replaced by an arbitrary index set, e.g., $\N$. Note that this correspondence is not one-to-one. In fact, it is one-to-$d!$ because for any fixed $d$, there are $d!$ permutations of
$j_1^{}, \ldots, j_d^{}$. 

With this, \eqref{eq:pitildepurewrittenout} becomes
\begin{equation*}
\begin{split}
&\frac{1}{c_N^{}} \widetilde \pi_{\textnormal{pure}}^{N} (\xi, \eta)
\\ & \quad = 
\sum_{d = d_{\textnormal{min}}^{}}^{d_{\textnormal{max}}^{}}
\frac{1}{d!} \sum_{\substack{j_1^{},\ldots,j_d^{} \in [N] \\ \textnormal{distinct} }}
\sum_{\widehat \varrho: [\widehat r] \hookrightarrow  
\{0,1\} \times [d]}
\sum_{\widecheck \varrho: [\widecheck r] \hookrightarrow 
\{0,1\} \times [d]}
\frac{2^{-\widehat b - \widecheck b}}{c_N^{} (N)_{\widehat b + \widecheck b \downarrow}}
\E \Big [
\prod_{\ell \in \{j_1^{},\ldots,j_d^{} \} } ( \widehat V_\ell)_{K_\ell \downarrow}
\Big] \\
& \quad = \sum_{d = d_{\textnormal{min}}^{}}^{d_{\textnormal{max}}^{}}
\frac{1}{d!} 
\sum_{\varrho_1^{}: [\widehat r] \hookrightarrow \{0,1 \} \times [d]}
\sum_{\varrho_2^{}: [\widecheck r] \hookrightarrow \{0,1 \} \times [d]}
\frac{(N)_{d \downarrow} 2^{-\widehat b - \widecheck b}}{c_N^{} (N)_{\widehat b + \widecheck b \downarrow}}
\E \Big [
\prod_{\ell =1 }^d ( \widehat V_\ell)_{\widetilde K_\ell \downarrow}
\Big],
\end{split}
\end{equation*}
where we used the exchangeability of the $ \widehat V_i$ in the second step by letting individual $\ell \in [d]$ take the role of individual $j_\ell^{}$ for each $\ell \in [d]$. This implies that 
$\widetilde K_\ell$ is given by
\begin{equation*}
\widetilde K_\ell \defeq 
\widehat k_{\widehat \varrho^{-1}(0,\ell)} 
+
\widehat k_{\widehat \varrho^{-1}(1,\ell)}
+
\widecheck k_{\widecheck \varrho^{-1}(0,\ell)}
+
\widecheck k_{\widecheck \varrho^{-1}(1,\ell)};
\end{equation*}
we set $\dot k_{\dot \varrho^{-1}(\ell,c)} \defeq 0$
whenever there is no preimage. 

Combining Lemma 3.3 and Eqs.(6,7) 
in~\cite{BirknerEtAl2018}, we see that
\begin{equation} \label{eq:quotation}
\begin{split}
&\lim_{N \to \infty}
\frac{(N)_{d \downarrow}}{c_N^{} (N)_{\widehat b+ \widecheck b \downarrow}}
\E \Big [
\prod_{\ell =1 }^d ( \widehat V_\ell)_{\widetilde K_\ell \downarrow}
\Big]  = 
\lim_{N \to \infty}
\frac{1}{c_N^{} N^{\widehat b + \widecheck b- d} }
\E \Big [
\prod_{\ell =1 }^d ( \widehat V_\ell)_{\widetilde K_\ell \downarrow}
\Big] \\
& \quad = 
\lim_{N \to \infty}
\frac{1}{c_N^{} N^{\widehat b + \widecheck b - d}} 2^{\widehat b + \widecheck b }
\E \Big [
\prod_{\ell = 1}^d (V_\ell)_{\widetilde K_\ell \downarrow}
\Big] \\
& \quad  =
\int_{\Delta \setminus \{ \mathbf{0} \}}
\sum_{\substack{j_1^{}, \ldots, j_d^{}
= 1 \\ \textnormal{distinct}}}^\infty 
\prod_{\ell \in [d]} x_{j_\ell^{}}^{\widetilde K_\ell} 
\frac{2 \Xi'(\dd x)}{\langle x, x \rangle} + 
\1_{d=1, \widetilde K_1^{} = 2}
2 \Xi'\big (\{ \bf 0  \} \big).
\end{split}
\end{equation}

For all $x \in \Delta$, we write
\begin{equation} \label{eq:phixrelabelling}
\varphi(x) = (x_1^{}/2,x_1^{}/2,x_2^{}/2,x_2^{}/2,\ldots)
\eqdef
(y_{0,1}^{},y_{1,1}^{},y_{0,2}^{},y_{1,2}^{},\ldots).
\end{equation}
Let us now deal with the product inside the integral.
Recalling the definition of $\widetilde K_\ell$, 
we obtain
\begin{align*}
\prod_{\ell \in [d]} x_{j_\ell^{}}^{\widetilde K_\ell} &=
\prod_{\ell \in [d]} x_{j_\ell^{}}^{\widehat k_{\widehat \varrho^{-1}(0,\ell)} 
+
\widehat k_{\widehat \varrho^{-1}(1,\ell)}
+
\widecheck k_{\widecheck \varrho^{-1}(0,\ell)}
+
\widecheck k_{\widecheck \varrho^{-1}(1,\ell)}} \\
&=
2^{\widehat b + \widecheck b} \prod_{\ell \in [d]} 
y_{0,j_\ell^{}}^{\widehat k_{\widehat \varrho^{-1}(0,\ell)}}
y_{1,j_\ell^{}}^{\widehat k_{\widehat \varrho^{-1}(1,\ell)}}
\prod_{\ell \in [d]}
y_{0,j_\ell^{}}^{\widecheck k_{\widecheck \varrho^{-1}(0,\ell)}}
y_{1,j_\ell^{}}^{\widecheck k_{\widecheck \varrho^{-1}(1,\ell)}}.
\end{align*}
Because $\widehat \varrho$ and $\widecheck \varrho$ are injective, 
$\widehat \varrho^{-1}(0,\ell)$ and $\widehat \varrho^{-1}(1,\ell)$ with
$\ell$ ranging from $1$ to $d$ are just a rearrangement of 
$[\widehat r]$. Hence,
\begin{equation*}
\prod_{\ell \in [d]} 
y_{0,j_\ell^{}}^{\widehat k_{\widehat \varrho^{-1}(0,\ell)}}
y_{1,j_\ell^{}}^{\widehat k_{\widehat \varrho^{-1}(1,\ell)}}
=
\prod_{\ell \in [\widehat r]} 
y_{\widehat \varrho (\ell)_1^{}, j_{\widehat \varrho (\ell)_2^{}}}^{\widehat k_\ell}.
\end{equation*}
Similarly,
\begin{equation*}
\prod_{\ell \in [d]} 
y_{j_\ell^{},0}^{\widecheck k_{\widecheck \varrho^{-1}(\ell,0)}}
y_{j_\ell^{},1}^{\widecheck k_{\widecheck \varrho^{-1}(\ell,1)}}
=
\prod_{\ell \in [\widecheck r]} 
y_{ \widecheck \varrho(\ell)_1^{},j_{\widecheck \varrho(\ell)_2^{}}}^{\widecheck k_\ell}
\end{equation*}
Together with (47), this shows that
\begin{equation*}
\begin{split}
&\lim_{N \to \infty}
\frac{(N)_{d \downarrow} 2^{-\widehat b - \widecheck b}}{c_N^{} (N)_{\widehat b + \widecheck b \downarrow}}
\E \Big [
\prod_{\ell =1 }^d ( \widehat V_\ell)_{\widetilde K_\ell \downarrow}
\Big] \\
&\quad =
\sum_{\substack{j_1^{}, \ldots, j_d^{} = 1 \\ \textnormal{distinct}}}^\infty 
\int_{\Delta \setminus \{ \mathbf{0} \}}
\prod_{\ell \in [\widehat r]} 
y_{\widehat \varrho (\ell)^{}_1, j_{\widehat \varrho(\ell)^{}_2}^{} }^{\widehat k_\ell}
\prod_{\ell \in [\widecheck r]} 
y_{\widecheck \varrho (\ell)^{}_1, j_{\widecheck \varrho (\ell)^{}_2}^{}}^{\widecheck k_\ell}
\frac{2 \Xi' (\dd x )}{\langle x,x \rangle}
+ 
\1_{d=1, \widetilde K_1^{} = 2} 2 \Xi'\big (\{ \bf 0  \} \big).
\end{split}
\end{equation*}
Let us ignore the term with the indicator for the time being. Summing over the choices of $j_1^{},\ldots,j_d^{}$ and all injective maps
$\widehat \varrho$ from $[\widehat r]$ into
$\{0,1\} \times [d]$
and $\widecheck \varrho$ from 
$[\widecheck r]$ into $\{0,1\} \times [d]$,
we see that
\begin{equation*}
\begin{split}
&\lim_{N \to \infty} \frac{1}{c_N^{}} \widetilde \pi_{\textnormal{pure}}^{N} (\xi, \eta)
\\ & \quad =
\sum_{d = d_{\textnormal{min}}^{}}^{d_{\textnormal{max}}^{}}
\frac{1}{d!} \sum_{\substack{j_1^{},\ldots,j_d =1 \\ \textnormal{distinct} }}^\infty
\sum_{\widehat \varrho}
\sum_{\widecheck \varrho}
\int_{\Delta \setminus \{ \mathbf{0} \}}
\prod_{\ell \in [\widehat r]} y_{\widehat \varrho (\ell)^{}_1, j_{\widehat \varrho(\ell)^{}_2}^{}}^{\widehat k_\ell}
\prod_{\ell \in [\widecheck r]} y_{\widecheck \varrho (\ell)^{}_1, j_{\widecheck \varrho(\ell)^{}_2}^{} }^{\widecheck k_\ell}
\frac{2 \Xi' (\dd x )}{\langle x,x \rangle}\textbf{}.
\end{split}
\end{equation*}

The fact that $j_1^{},\ldots,j_d^{}$ are distinct and 
$\widehat \varrho$
as well as $\widecheck \varrho$ are injective means that 
the indices of the $y$ variables in both products respectively are pairwise distinct elements of
$\{0,1\} \times \N$. Conversely, for each such choice, there are $d!$ possible choices of $j_1^{},\ldots,j_d^{}$ and $\widehat \varrho, \widecheck \varrho$ (see above). This last expression is therefore equal to
\begin{equation*} 
\int_{\Delta \setminus \{ \mathbf{0}    \} }
\Big (
\sum_{\substack{(\widehat c_1^{}, \widehat \jmath_1^{}),\ldots,
(\widehat c_{\widehat r}^{}, \widehat \jmath_{\widehat r}^{}) \in  \{0,1 \} \times \N  \\ \textnormal{distinct} }}
\prod_{\ell \in [\widehat r]}
y_{\widehat c_\ell^{}, \widehat \jmath_\ell^{}}^{\widehat k_\ell}
\Big )
\Big (
\sum_{\substack{(\widecheck c_1^{}, \widecheck \jmath_1^{}),\ldots,
(\widecheck c_{\widecheck r}^{}, \widecheck \jmath_{\widecheck r}^{}) \in \{0,1 \} \times \N \\ \textnormal{distinct} }}
\prod_{\ell \in [ \widecheck r]} 
y_{\widecheck c_\ell^{}, \widecheck \jmath_\ell^{}}^{\widecheck k_\ell} 
\Big )
\frac{2 \Xi' (\dd x )}{\langle x,x \rangle}.
\end{equation*}
Recalling the definition of $\widehat k$ and $\widecheck k$ from the beginning of this subsection, their connection to
$\widehat \xi, \widecheck \xi, \widehat \eta, \widecheck \eta$, \eqref{eq:paintboxtransitionprime} for $s=0$
and \eqref{eq:phixrelabelling}
we see that the two brackets in the integral are equal 
(for fixed~$x$) to $p(\varphi(x); \widehat \xi, \widecheck \xi)$ and 
$p(\varphi(x);\widehat \eta, \widecheck \eta)$, respectively. We finally see that
\begin{equation*}
\begin{split}
&\lim_{N \to \infty} \frac{1}{c_N^{}} \widetilde \pi_{\textnormal{pure}}^{N} (\xi, \eta) \\
& \quad =
\int_{\Delta \setminus \{ \mathbf{0}    \} } 
p(\varphi(x);\widehat \xi,\widecheck \xi) 
p(\varphi(x);\widehat \eta, \widecheck \eta)
\frac{\Xi'(\dd x)}{\langle \varphi(x), \varphi(x) \rangle } 
= 
\int_{\Delta \setminus \{ \mathbf{0} \}} 
p(x;\widehat \xi, \widecheck \xi) p(x;\widehat \eta, \widecheck \eta)
\frac{\Xi(\dd x)}{\langle x, x \rangle }.
\end{split}
\end{equation*}
In this last step, we used that $\langle \varphi(x), \varphi(x) \rangle = \langle x,x \rangle /2$ and that $\Xi = \Xi' \circ \varphi^{-1}$.
But we have to keep in mind that we assumed $d>1$ and/or
$\widetilde K_1 > 2$.
A moment's thought reveals that
\begin{equation*} \1_{d=1, \widetilde K_1^{} = 2} 
2 \Xi'\big ({\bf 0}  \big )
=
 c_{\mathsf{pair}}^{} \1_{\widehat \xi \vdash \widehat \eta, \,
\widecheck \xi = \widecheck \eta} 
+
 c_{\mathsf{pair}}^{} \1_{\widehat \xi = \widehat \eta, \,
\widecheck \xi \vdash \widecheck \eta}.
\end{equation*}

We now tackle the case that $\dot k_j = 1$ for some 
$\cdot \in \{\wedge, \vee\}$ and $j \in 1,\ldots,\dot r$, i.e. that there are blocks which do not merge. 

It is not hard to see that 
\eqref{eq:pitildepurewrittenout} holds
also in general (without assuming
$\widehat s = \widecheck s = 0$); we just need to replace the upper indices $\widehat r$ and $\widecheck r$ by 
$\widehat r + \widehat s$ and $\widecheck r + \widecheck s$, respectively, while also letting 
$\widehat k_\ell \defeq 1$
and 
$\widecheck k_\ell \defeq 1$
for $\ell \in \{\widehat r + 1, \ldots, \widehat r + \widehat s \}$
and
$\ell \in \{\widecheck r + 1, \ldots, \widecheck r + \widecheck s \}$,
respectively. We also extend the definition of the $K_\ell$ 
accordingly.

This also immediately shows that $\pi^{N}_{\mathsf{pure}}$ only depends on $\widehat r$,$\widehat s$,$\widecheck r$,$\widecheck s$ as well as
$\widehat k_1,\ldots,\widehat k_{\widehat r}$ and 
$\widecheck k_1,\ldots,\widecheck k_{\widecheck r}$, allowing us to write
\begin{equation*}
\widetilde \pi^{N}_{\mathsf{pure}} (\xi,\eta) 
=
\widetilde \pi^{N}_{\mathsf{pure}} 
\big (\widehat r, \widehat s, \widecheck r, \widecheck s;
\widehat k_1,\ldots,\widehat k_{\widehat r};
\widecheck k_1,\ldots,\widecheck k_{\widecheck r} 
\big).
\end{equation*}
or, alternatively,
\begin{equation*}
\widetilde \pi^{N}_{\mathsf{pure}} (\xi,\eta) 
=
\widetilde \pi^{N}_{\mathsf{pure}} 
\big (\widehat r + \widehat s, \widecheck r + \widecheck s;
\widehat k_1,\ldots,\widehat k_{\widehat r + \widehat s};
\widecheck k_1,\ldots,\widecheck k_{\widecheck r + \widecheck s} \big).
\end{equation*}

In this notation, we have the following asymptotic `consistency relation' as $N\to \infty$; see~\cite[p.~19]{BirknerEtAl2018}.
\begin{equation*}
\begin{split}
& \widetilde \pi^N_{\mathsf{pure}} 
\big ( \widehat r + \widehat s + 1, \widecheck r + \widecheck s;
\widehat k_1,\ldots,\widehat k_{\widehat r + \widehat s},1;
\widecheck k_1,\ldots,\widecheck k_{\widecheck r + \widecheck s} \big ) \\
& \quad \sim
\widetilde \pi^N_{\mathsf{pure}} 
\big ( \widehat r + \widehat s, \widecheck r + \widecheck s;
\widehat k_1,\ldots,\widehat k_{\widehat r + \widehat s};
\widecheck k_1,\ldots,\widecheck k_{\widecheck r + \widecheck s} \big ) \\
& \quad - 
\widetilde \pi^N_{\mathsf{pure}} 
\big ( \widehat r + \widehat s, \widecheck r + \widecheck s;
\widehat k_1 + 1,\ldots,\widehat k_{\widehat r + \widehat s};
\widecheck k_1,\ldots,\widecheck k_{\widecheck r + \widecheck s} \big ) \\
&\quad - \ldots \\
&\quad - 
\widetilde \pi^N_{\mathsf{pure}} 
\big ( \widehat r + \widehat s, \widecheck r + \widecheck s;
\widehat k_1,\ldots,\widehat k_{\widehat r + \widehat s} + 1;
\widecheck k_1,\ldots,\widecheck k_{\widecheck r + \widecheck s} \big ),
\end{split}
\end{equation*}
and similarly in the second component:
\begin{equation*}
\begin{split}
& \widetilde \pi^N_{\mathsf{pure}} 
\big ( \widehat r + \widehat s, \widecheck r + \widecheck s + 1;
\widehat k_1,\ldots,\widehat k_{\widehat r + \widehat s},1;
\widecheck k_1,\ldots,\widecheck k_{\widecheck r + \widecheck s} \big ) \\
& \quad \sim
\widetilde \pi^N_{\mathsf{pure}} 
\big ( \widehat r + \widehat s, \widecheck r + \widecheck s;
\widehat k_1,\ldots,\widehat k_{\widehat r + \widehat s};
\widecheck k_1,\ldots,\widecheck k_{\widecheck r + \widecheck s} \big ) \\
& \quad - 
\widetilde \pi^N_{\mathsf{pure}} 
\big ( \widehat r + \widehat s, \widecheck r + \widecheck s;
\widehat k_1,\ldots,\widehat k_{\widehat r + \widehat s};
\widecheck k_1 + 1,\ldots,\widecheck k_{\widecheck r + \widecheck s} \big ) \\
&\quad - \ldots \\
&\quad - 
\widetilde \pi^N_{\mathsf{pure}} 
\big ( \widehat r + \widehat s, \widecheck r + \widecheck s;
\widehat k_1,\ldots,\widehat k_{\widehat r + \widehat s};
\widecheck k_1,\ldots,\widecheck k_{\widecheck r + \widecheck s} + 1 \big ).
\end{split}
\end{equation*}
Note that an analogous recursion holds for the limiting rates
$q^{\mathsf{pure}}_{\xi, \eta}$. Using this, we can deduce the general case from the special case $\widehat s = \widecheck s = 0$.

To prove the recursion for 
$\widetilde \pi_{\mathsf{pure}}^N$, 
we slightly rewrite \eqref{eq:pitildepurewrittenout} (to be precise, its generalised version):
\begin{equation*}
\begin{split}
& \widetilde \pi^N_{\mathsf{pure}} 
\big ( \widehat r + \widehat s + 1, \widecheck r + \widecheck s;
\widehat k_1,\ldots,\widehat k_{\widehat r + \widehat s},1;
\widecheck k_1,\ldots,\widecheck k_{\widecheck r + \widecheck s} \big ) \\
& \quad  =
\sum_{\substack{  
(\widehat c_1^{}, \widehat \imath_1^{}), \ldots, 
(\widehat c_{\widehat r + \widehat s + 1}^{}, \widehat \imath_{\widehat r + \widehat s + 1}^{}) 
\in \{0,1 \} \times [N] \\
\textnormal{pairwise distinct}
}} \,
\sum_{\substack{  
(\widecheck c_1^{}, \widecheck \imath_1^{}), \ldots,
(\widecheck c_{\widecheck r + \widecheck s}^{}, \widecheck \imath_{\widecheck r + \widecheck s}^{}) 
\in \{ 0,1 \} \times [N] \\
\textnormal{pairwise distinct}
}} \\
& \quad \quad \quad \cdot
\frac{2^{-\widehat b^{} - \widecheck b^{} - 1}}{c_N^{} (N)_{\widehat b^{} + \widecheck b^{}+1 \downarrow}} \E
\Big [
\prod_{\ell \in \{ \widehat \imath_1^{}, \ldots, \widehat \imath_{\widehat r + \widehat s + 1}^{},
         \widecheck \imath_1^{}, \ldots, \widecheck \imath_{\widecheck r + \widecheck s}^{} 
\}  }
\big (  \widehat V_{\ell} \big )_{K_\ell \downarrow}
\Big ] \\
& \quad  =
\sum_{\substack{  
(\widehat c_1^{}, \widehat \imath_1^{}), \ldots,
(\widehat c_{\widehat r + \widehat s}^{}, \widehat \imath_{\widehat r + \widehat s }^{}) 
\in  \{0,1 \} \times [N] \\
\textnormal{pairwise distinct}
}} \,
\sum_{\substack{  
(\widecheck c_1^{}, \widecheck \imath_1^{}), \ldots, (\widecheck c_{\widecheck r + \widecheck s}^{}, \widecheck \imath_{\widecheck r + \widecheck s}^{}) 
\in \{ 0,1 \} \times [N]  \\
\textnormal{pairwise distinct}
}} \frac{2^{-\widehat b^{} - \widecheck b^{} - 1}}{c_N^{} (N)_{\widehat b^{} + \widecheck b^{}+1 \downarrow}} \\
& \quad \quad \quad \cdot
 \E
\Big [
\prod_{\ell \in \{ \widehat \imath_1^{}, \ldots, \widehat \imath_{\widehat r + \widehat s}^{},
         \widecheck \imath_1^{}, \ldots, \widecheck \imath_{\widecheck r + \widecheck s}^{} 
\}  }
\big (  \widehat V_{\ell} \big )_{K_\ell \downarrow}
\sum_{
\substack{
(c, i) 
\in \{0,1\} \times [N] \\
\textnormal{distinct from } \\
(\widehat c_1^{}, \widehat \imath_1^{}), \ldots, ( \widehat c_{\widehat r + \widehat s}^{}, \widehat \imath_{\widehat r + \widehat s}^{}) 
}
}
\big (  \widehat V_i^{} - K_i  \big )
\Big ] \\
& \quad  =
\sum_{\substack{  
(\widehat c_1^{}, \widehat \imath_1^{}), \ldots, (\widehat \imath_{\widehat c_{\widehat r + \widehat s}^{}, \widehat r + \widehat s }^{}) 
\in \{0,1\} \times [N] \\
\textnormal{pairwise distinct}
}} \,
\sum_{\substack{  
(\widecheck c_1^{}, \widecheck \imath_1^{}), \ldots, (\widecheck c_{\widecheck r + \widecheck s}^{}, \widecheck \imath_{\widecheck r + \widecheck s}^{}) 
\in \{ 0,1 \}  \times [N] \\
\textnormal{pairwise distinct}
}} \frac{2^{-\widehat b^{} - \widecheck b^{} - 1}}{c_N^{} (N)_{\widehat b^{} + \widecheck b^{}+1 \downarrow}} \\
& \quad \quad \quad \cdot
 \E
\Big [
\prod_{\ell \in \{ \widehat \imath_1^{}, \ldots, \widehat \imath_{\widehat r + \widehat s}^{},
        \widecheck \imath_1^{}, \ldots, \widecheck \imath_{\widecheck r + \widecheck s}^{}
\}  }
\big (  \widehat V_{\ell} \big )_{K_\ell \downarrow}
\big (
2N - \widehat b^{} -  \widecheck b^{} \\
& - (  \widehat V_{\widehat \imath_1^{}} - K_{\widehat \imath_1^{}}) - \ldots - ( \widehat V_{i_{\widehat r + \widehat s}^{}} - K_{i_{\widehat r + \widehat s}^{}}))
\big )
\Big ].
\end{split}
\end{equation*}

This shows the claim, keeping in mind that, for
$\ell \in \{ 
\widehat \imath_1^{},\ldots,\widehat \imath_{\widehat r + \widehat s}^{}
\}$,
incrementing $K_\ell$ is by \eqref{eq:Kelldef} equivalent to incrementing $\widehat k_j$ for the unique $j$ such that $\widehat \imath_j^{} = \ell$.
\end{proof}

\section{Sufficient conditions for $c_{\mathsf{pair}}=1$}
\label{app:sufficient}
Write $||V||_2^2:=|| (V_1, V_2,\dots, V_N) ||_2^2 := \sum_{i=1}^N V_i^2$ for simplicity throughout this section. 

Suppose Assumption~\eqref{eq:PhiNconv} holds. Then by integrating the left of \eqref{eq:PhiNconv} on a ball of radius $\varepsilon$ and applying \eqref{eq:cN1}, we obtain
\begin{align}
  \label{eq:cpair.r1}
  c_{\mathsf{pair}}
   =& \lim_{\varepsilon \downarrow 0}
  \lim_{N\to\infty} \frac{\E^{(N)}\Big[ V_1(V_1-1) \1\big( ||V||_2^2 \leqslant 4 \varepsilon N^2 \big) \Big]}{\E^{(N)}[V_1(V_1-1)]}.
\end{align}
Under Assumption~\eqref{eq:limcN=0}, the last display is equivalent to
\begin{align}
  \label{eq:cpair.r1.var}
  c_{\mathsf{pair}}
   =& \lim_{\varepsilon \downarrow 0}
  \lim_{N\to\infty} \frac{\E^{(N)}\Big[ ||V||_2^2 \leqslant 4 \varepsilon N^2 \big) \Big]-2}{\E^{(N)}[V_1^2]-2} 
\end{align}
because $\E^{(N)}[V_1]=2$ and $\P^{(N)}(||V||_2^2 \leqslant 4 \varepsilon N^2) \to 1$ as $N\to \infty$. 

To see that $\P^{(N)}(||V||_2^2 \leqslant 4 \varepsilon N^2) \to 1$, note that
exchangeability and \eqref{eq:cN1}  imply
$\E^{(N)}[||V||_2^2]= N\E^{(N)}[V_1^2]= N\E^{(N)}[V_1(V_1-1)] + 2N = 8N(N-1)c_N + 2N$.
Hence, for every fixed $\varepsilon>0$, Markov's inequality and Assumption~\eqref{eq:limcN=0} give
\begin{equation*}
  \P^{(N)}(||V||_2^2 > 4\varepsilon N^2)
  \leqslant \frac{2(N-1)c_N}{\varepsilon N} + \frac{1}{2\varepsilon N}
  \to 0,
\end{equation*}

\begin{lemma}
The equality $c_{\mathsf{pair}}=1$ holds if any of the followings hold: 
\begin{itemize}
\item[(i)]  $(V_i^2 )_{1\leqslant i\leqslant N,\,N\in\mathbb{N}}$ is uniformly integrable (e.g. $\sup_{N\in\mathbb{N}} \E^{(N)}[V_1^{2p}]<\infty$ for some $p\in(1,\infty)$);
 \item[(ii)]  $\lim_{N\to\infty}\frac{\E^{(N)}[V_1^{2p}]}{N^{p-1}}=0$ for some $p\in(1,\infty)$.
\end{itemize}
\end{lemma}

\begin{proof}
Assumption~\eqref{eq:PhiNconv} yields, for $\varepsilon > 0$ such that 
 $\Xi'(\partial B_\varepsilon(\mathbf{0})) = 0$,
\begin{align*}
  & \frac{1}{4N^2}\frac1{2 c_N} \E^{(N)}\Big[ ||V||_2^2
  \1\big( ||V||_2^2 > 4\varepsilon^2 N^2 \big) \Big] \notag \\
  & = \frac1{8 N c_N} \E^{(N)}\Big[ V_1^2 \1\big( ||V||_2^2 > 4 \varepsilon^2 N^2 \big) \Big]
    \mathop{\longrightarrow}_{N\to\infty} \Xi'\big( \Delta \setminus B_\varepsilon(\mathbf{0}) \big).
\end{align*}
Furthermore, for $\varepsilon > 0$ and 
any $(V_1,\dots,V_N)$ with $||V||_2^2 \geqslant 4\varepsilon^2 N^2$, we have
\begin{align*}
& \sum_{i=1}^N V_i^2
  \geqslant \sum_{i=1}^N V_i(V_i-1)
  \geqslant \Big( 1 - \frac{1}{2\varepsilon^2 N} \Big)
  \sum_{i=1}^N V_i^2
\end{align*}
where we used that $\sum_{i=1}^N V_i= 2N$ in the second inequality. Therefore, using \eqref{eq:cN1},
\begin{align*}
  c_{\mathsf{pair}}
  & = 1 - \lim_{\varepsilon \downarrow 0}
  \lim_{N\to\infty} \frac1{8 N c_N} \E^{(N)}\Big[ V_1^2 \1\big( ||V||_2^2 > 4 \varepsilon^2 N^2 \big) \Big] \notag \\
  & = \lim_{\varepsilon \downarrow 0}
  \lim_{N\to\infty} \frac{\E^{(N)}[V_1(V_1-1)] - \E^{(N)}\Big[ V_1(V_1-1) \1\big( ||V||_2^2 > 4 \varepsilon^2 N^2 \big) \Big]}{\E^{(N)}[V_1(V_1-1)]}.
\end{align*}
This proves \eqref{eq:cpair.r1}.

Next, we find sufficient conditions under which $c_{\mathsf{pair}}=1$. Let $f_N(V):=\frac{\sum_{i=1}^NV_i^2}{\E^{(N)}[V_1^2]N}$ for all $N\in\mathbb{N}$. Then $\E^{(N)}[f_N(V)]=1$ and 
\begin{align}\label{E:1-cpair}
&\frac{\E^{(N)}\Big[ V_1^2 \1\big( ||V||_2^2 > 4 \varepsilon N^2 \big) \Big]}{\E^{(N)}[V_1^2]}
=  \E^{(N)}\left[ f_N(V) \1\left(f_N(V)> \frac{\varepsilon\,4N}{\E^{(N)}[V_1^2]} \right) \right].
\end{align}
Suppose 
$\big (f_N(V)\big )_{N\in\mathbb{N}}$ is uniformly integrable,
then \eqref{E:1-cpair} tends to zero as $N\to\infty$ for every $\varepsilon>0$, implying $c_{\mathsf{pair}}=1$ by \eqref{eq:cpair.r1}; recall that $c_N^{} \to 0$ and therefore $\lim_{N \to \infty} \E^{(N)}[V_1^2 ] / N = 0$.
Note that $f_N(V)\leqslant \frac{1}{4N}\sum_{i=1}^NV_i^2$ because 
$\E^{(N)}[V_1^2] \geqslant  
\E^{(N)} [V_1]^2 = 4$. 

Part (i) then follows because 
$\big (
\frac{1}{N}\sum_{i=1}^NV_i^2
\big )_{N\in\mathbb{N}}$ is uniformly integrable if the triangular array 
$(V_i^2)_{1\leqslant i\leqslant N,\,N\in\mathbb{N}}$ is uniformly integrable.
For part (ii), note  that according to H\"older's inequality, \eqref{E:1-cpair} is bounded above by 
\begin{equation*}
\begin{split}
&\frac{\E^{(N)}\Big[ V_1^{2p} \Big]^{\frac{1}{p}} \, \P^{(N)}\left(f_N(V)> \frac{\varepsilon\,4N}{\E^{(N)}[V_1^2]} \right)^{\frac{p-1}{p}}  }{
\E^{(N)}[V_1^2]} \\
& \leqslant  \frac{\E^{(N)}\Big[ V_1^{2p} \Big]^{\frac{1}{p}} \, \big(\frac{1}{4N\varepsilon} \big)^{\frac{p-1}{p}}  }{\E^{(N)}[V_1^2]^{\frac{1}{p}}}
\leqslant  \frac{\E^{(N)}\Big[ V_1^{2p} \Big]^{\frac{1}{p}} \, \big(\frac{1}{4N\varepsilon} \big)^{\frac{p-1}{p}}  }{4^{\frac{1}{p}}}.
\end{split}
\end{equation*}
The first inequality holds because Markov's inequality implies 
$\P^{(N)}\big (  
f_N^{}(V) > \frac{\varepsilon 4 N}{\E^{(N)}[V_1^2]}
\big ) \leqslant 
\frac{\E^{(N)}[V_1^2]}{4 \varepsilon N}$. The right hand side of the above displayed inequalities (and hence \eqref{E:1-cpair}) tend to zero as $N\to\infty$, for every $\varepsilon>0$, under the assumption stated in part (ii).
\end{proof}

We finish with some remarks regarding criteria for the presence/absence of large mergers.

To understand the effect of small families, we reconsider our convergence assumption \eqref{eq:PhiNconv}. For any fixed $\varepsilon > 0$ and any 
$A \subseteq \Delta \setminus B_\varepsilon ( \mathbf{0} )$
with 
$\Xi' (\partial A) = 0$, 
vague convergence implies that
\begin{equation} \label{eq:PhiNconvprime}
\frac{1}{ c_N^{}} \int_A \frac{1}{2}  \langle x, x \rangle \Phi_N^{} (\dd x) \mathop{\longrightarrow}_{N \to \infty} \Xi'(A).
\end{equation} 
In the prelimit, $\Phi_N (\dd x)$ is the probability of observing a SG, leading to an `approximate' $\varphi(x)$-merger  with pair-coalescence probability $\approx \langle \varphi(x), \varphi (x) \rangle = \frac{1}{2} \langle x,x \rangle$. Therefore, \eqref{eq:PhiNconvprime} implies that  whenever two genes find a common ancestor, 
$\Xi'(A)$ is approximately the conditional probability that this coalescence was due to a $\varphi(x)$-merger for some $x \in A$ or, in other words, that the ranked total offspring frequencies in the generation where the most recent common ancestor of the sampled pair lived were given by $x \in A$.

\begin{remark}
By \eqref{eq:cN1},
\begin{equation*}
\frac{1}{2 c_N} \int_{\Delta \setminus \{ \mathbf{0} \}} \langle x,x \rangle \Phi_N (\dd x) 
=\frac{1}{8 c_N N^2} \E^{(N)}\left[\sum_{i=1}^NV_{i}^2\right]
\sim
\frac{\E^{(N)}\left[V_{1}^2\right]}{\E^{(N)}\left[V_{1}^2\right] - \E^{(N)}\left [V_1^{}\right]} \quad 
\textnormal{as } N \to \infty.
\end{equation*}
Note that the limsup of the r.h.s.\ is strictly larger than $1$ if $\liminf_{N\to\infty}\E^{(N)} [V_1^2]<\infty$, in spite of the fact that in \eqref{eq:PhiNconvprime} the limit is $\leqslant 1$ if $\Delta \setminus \{ \mathbf{0}  \}$
is replaced by a subset $A$ bounded away from 
$\mathbf{0}$. This reflects that if $\lim_{N\to\infty}\E^{(N)} [V_1^2]<\infty$, then although each $\Phi_N$ is concentrated on $\Delta \setminus \{\mathbf{0}\}$, part (or possibly all) of the mass of the rescaled $\Phi_N/(2c_N)$ ``escapes'' to $\mathbf{0}$ in the limit as $N\to\infty$.
\hfill \mbox{$\Diamond$}
\end{remark}

\begin{remark}
One might wonder if $\lim_{N\to\infty} \E^{(N)}[V_1^2] < \infty$ 
implies $c_{\mathsf{pair}}=1$. The answer is no, as illustrated by the model from~\cite{DiamantidisEtAl2024}, which clearly exhibits large mergers. In this case, $\lim_{N\to\infty} \E^{(N)}[V_1^2] < \infty$ and $c_{\mathsf{pair}}<1$ (see Section~\ref{ex:HRcouple}).  To compute $\E^{(N)}[V_1^2]$, recall that 
$\psi \in (0,1]$ was used 
in~\cite{DiamantidisEtAl2024} to denote 
the fraction of offspring of the productive pair and that the probability of the occurrence of such a pair per generation is 
$\alpha / N$, we have
\begin{equation*}
\E^{(N)} [V_1^2   ] \leqslant C \big ( \alpha N^{-1} N^{-1} (\psi N)^2 + (1-\alpha) 2N  N^{-1} (1 - N^{-1}) + (2N)^2 N^{-2} \big ).
\end{equation*}  
for some constant $C$.
Note that, in a Wright-Fisher generation, $V_1 \sim \textnormal{Binomial} (2N, N^{-1})$ and in a ``high-offspring generation'', individual $1$ belongs to the large family with probability of order $N^{-1}$. 

\hfill \mbox{$\Diamond$}
\end{remark}

\end{appendix}

\section*{Acknowledgements}
We would like to thank two anonymous reviewers whose critical reading and insightful comments helped to considerably improve the presentation of the manuscript.

\bibliography{main_aap}       

@article{Buri1956,
 author = {Buri, Peter},
 title = {Gene frequency in small populations of mutant {D}rosophila},
 journal = {Evolution},
 volume = {10},
 number = {4},
 pages = {367--402},
 year = {1956},
 doi = {10.1111/j.1558-5646.1956.tb02864.x} }

@article{DerEtAl2011,
 author = {Der, Ricky and Epstein, Charles L and Plotkin, Joshua B},
 title = {Generalized population models and the nature of genetic drift},
 journal = {Theoretical Population Biology},
 volume = {80},
 number = {2},
 pages = {80--99},
 year = {2011},
 doi = {10.1016/j.tpb.2011.06.004} }

@article{DerEtAl2012,
 author = {Der, Ricky and Epstein, Charles and Plotkin, Joshua B},
 title = {Dynamics of neutral and selected alleles when the offspring distribution is skewed},
 journal = {Genetics},
 volume = {191},
 number = {4},
 pages = {1331--1344},
 year = {2012},
 doi = {10.1534/genetics.112.140038} }

@article{BirknerEtAl2024,
 author = {Matthias Birkner and Florin Boenkost and Iulia Dahmer and Cornelia Pokalyuk},
 title = {On the fixation probability of an advantageous allele in a population with skewed offspring distribution},
 volume = {29},
 journal = {Electronic Journal of Probability},
 number = {none},
 pages = {1 -- 22},
 year = {2024},
 doi = {10.1214/24-EJP1198} }

@article{EldonAndStephan2024,
 author = {Eldon, Bjarki and Stephan, Wolfgang},
 title = {Sweepstakes reproduction facilitates rapid adaptation in highly fecund populations},
 year = {2024},
 journal = {Molecular Ecology},
 volume = {33},
 number = {10},
 pages = {e16903},
 doi = {10.1111/mec.16903} }

@article{BertoinLeGall2003,
 author = {Bertoin, P. and Le Gall, J.F.},
 title = {Stochastic flows associated to coalescent processes},
 journal = {Probab. Theory Related Fields},
 year = {2003},
 volume = {126},
 pages = {261--288},
}

@article{CorderoEtAl2025,
 author = {Fernando Cordero and Sebastian Hummel and Gr{\'e}goire V{\'e}chambre},
 title = {{$\Lambda$-Wright–Fisher processes with general selection and opposing environmental effects: {F}ixation and coexistence}},
 volume = {35},
 journal = {The Annals of Applied Probability},
 number = {1},
 pages = {393 -- 457},
 year = {2025},
 doi = {10.1214/24-AAP2117} }

@InProceedings{Donnelly1996,
 author = {Donnelly, Peter},
 editor = {Chadwick, Derek and Cardew, Gail},
 title = {Interpreting genetic variability: the effects of shared evolutionary history},
 booktitle = {Variation in the Human Genome},
 series = {Ciba Foundation Symposium 197},
 year = {1996},
 publisher = {John Wiley \& Sons, Ltd.},
 address = {Chichester, England},
 pages = {25--40} }

@article{MeyerZheng1984,
 author = {Meyer, P.A. and Zheng, W.A.},
 title = {Tightness criteria for laws of semimartingales},
 journal = {Ann. inst. Henri Poincaré (B) Probab. Stat.},
 year = {1984},
 volume = {20},
 number={4},
 pages = {353--372},
}

@article{FuAndLi1993,
 author = {Fu, Yun-{X}in X and Li, Wen-{H}siung},
 title = {Statistical tests of neutrality of mutations.},
 journal = {Genetics},
 volume = {133},
 number = {3},
 pages = {693--709},
 year = {1993},
 doi = {10.1093/genetics/133.3.693} }

@article{Watterson1975,
 author = {Watterson, G A},
 title = {On the number of segregating sites in genetical models without recombination},
 journal = {Theoretical Population Biology},
 volume = {7},
 number = {2},
 pages = {256--276},
 year = {1975},
 doi = {10.1016/0040-5809(75)90020-9} }

@article{SainudiinEtAl2011,
 author = {Sainudiin, Raazesh and Thornton, Kevin and Harlow, Jennifer and Booth, James and Stillman, Michael and Yoshida, Ruriko and Griffiths, Robert and McVean, Gil and Donnelly, Peter},
 doi = {10.1007/s11538-010-9605-5},
 journal = {Bulletin of Mathematical Biology},
 number = {4},
 pages = {829--872},
 title = {Experiments with the Site Frequency Spectrum},
 volume = {73},
 year = {2011} }

@article{SeplyarskiyEtAl2023,
 author = {Seplyarskiy, Vladimir and Koch, Evan M and Lee, Daniel J and Lichtman, Joshua S and Luan, Harding H and Sunyaev, Shamil R},
 title = {A mutation rate model at the basepair resolution identifies the mutagenic effect of polymerase {III} transcription},
 journal = {Nature Genetics},
 year = {2023},
 volume = {55}, 
 pages = {2235--2242},
 doi = {10.1038/s41588-023-01562-0} }

@article{Zeitouni2006,
 doi = {10.1088/0305-4470/39/40/R01},
 year = {2006},
 volume = {39},
 number = {40},
 pages = {R433},
 author = {Zeitouni, Ofer},
 title = {Random walks in random environments},
 journal = {Journal of Physics A: Mathematical and General} }

@article{HassEtAl2024,
 title = {First-passage time for many-particle diffusion in space-time random environments},
 author = {Hass, Jacob B and Corwin, Ivan and Corwin, Eric I},
 journal = {Physical Review E},
 volume = {109},
 issue = {5},
 pages = {054101},
 year = {2024},
 doi = {10.1103/PhysRevE.109.054101} }

@article{CaicedoEtAl2007,
 doi = {10.1371/journal.pgen.0030163},
 author = {Caicedo, Ana L AND Williamson, Scott H AND Hernandez, Ryan D AND Boyko, Adam AND Fledel-Alon, Adi AND York, Thomas L AND Polato, Nicholas R AND Olsen, Kenneth M AND Nielsen, Rasmus AND McCouch, Susan R AND Bustamante, Carlos D AND Purugganan, Michael D},
 journal = {PLOS Genetics},
 title = {Genome-wide patterns of nucleotide polymorphism in domesticated rice},
 year = {2007},
 volume = {3},
 pages = {1--12},
 number = {9} }

@article{SlatkinAndHudson1991,
 author = {Slatkin, Montgomery and Hudson, Richard R},
 title = {Pairwise comparisons of mitochondrial {DNA} sequences in stable and exponentially growing populations},
 journal = {Genetics},
 volume = {129},
 number = {2},
 pages = {555--562},
 year = {1991},
 doi = {10.1093/genetics/129.2.555} }

@article{NordborgAndDonnelly1997,
 author = {Nordborg, Magnus and Donnelly, Peter},
 title = {The coalescent process with selfing},
 journal = {Genetics},
 volume = {146},
 number = {3},
 pages = {1185--1195},
 year = {1997},
 doi = {10.1093/genetics/146.3.1185} }

@article{SeversonEtAl2021,
 title = {Variance and limiting distribution of coalescence times in a diploid model of a consanguineous population},
 journal = {Theoretical Population Biology},
 volume = {139},
 pages = {50--65},
 year = {2021},
 doi = {10.1016/j.tpb.2021.02.002},
 author = {Severson, Alissa L and Carmi, Shai and Rosenberg, Noah A} }

@article{CotterEtAl2021,
 author = {Cotter, Daniel J and Severson, Alissa L and Rosenberg, Noah A},
 title = {The effect of consanguinity on coalescence times on the X chromosome},
 journal = {Theoretical Population Biology},
 volume = {140},
 pages = {32--43},
 year = {2021},
 doi = {10.1016/j.tpb.2021.03.004} }

@article{NewmanEtAl2025,
 author = {Maximillian Newman and John Wakeley and Wai-Tong (Louis) Fan},
 title = {Conditional gene genealogies given the population pedigree for a diploid Moran model with selfing},
 journal = {Theoretical Population Biology},
 volume = {165},
 pages = {29--44},
 year = {2025},
 doi = {10.1016/j.tpb.2025.05.003} }

@article{fan2025quenched,
  title={Quenched coalescent for diploid population models with selfing and overlapping generations},
  author={Fan, Louis Wai-Tong and Newman, Maximillian and Wakeley, John},
  journal={arXiv preprint arXiv:2510.26115},
  year={2025}
}

@article{Kingman1982b,
 title = {The coalescent},
 journal = {Stochastic Processes and their Applications},
 volume = {13},
 number = {3},
 pages = {235--248},
 year = {1982},
 doi = {10.1016/0304-4149(82)90011-4},
 author = {Kingman, J F C} }

@article{Kingman1978,
 author = {Kingman, J F C},
 title = {The representation of partition structures},
 journal = {J.~London~Math.~Soc.},
 volume = {18},
 pages = {374--380},
 year = {1978}
}

@incollection{Kingman1982c,
 author = {Kingman, J F C},
 title = {Exchangeability and the evolution of large populations},
 booktitle = {Exchangeability in Probability and Statistics},
 year = {1982},
 pages = {97--112},
 publisher = {North-Holland},
 editor = {Koch, G and Spizzichino, F},
 address = {Amsterdam}  }

@article{DonnellyAndKurtz1996a,
 author = {Donnelly, Peter and Kurtz, Thomas G},
 title = {A countable representation of the {F}leming-{V}iot measure-valued diffusion},
 volume = {24},
 journal = {The Annals of Probability},
 number = {2},
 pages = {698 -- 742},
 year = {1996},
 doi = {10.1214/aop/1039639359} }

@article{Tajima1983,
 author = {Tajima, Fumio},
 title = {Evolutionary relationship of {DNA} sequences in finite populations},
 year = {1983},
 journal = {Genetics},
 volume = {105},
 number = {2},
 pages = {437--460},
 doi = {10.1093/genetics/105.2.437} }

@article{Tavare1984,
 author = {Tavar{\'e}, S},
 title = {Lines-of-descent and genealogical processes, and their application in population genetic models},
 journal = {Theoretical Population Biology},
 year = {1984},
 volume = {26},
 number = {2},
 pages = {119--164},
 doi = {10.1016/0040-5809(84)90027-3} }

@InProceedings{Donnelly1986,
 author = {Donnelly, Peter},
 editor = {Tautu, Petre},
 title = {Dual processes in population genetics},
 booktitle = {Stochastic Spatial Processes},
 year = {1986},
 publisher = {Springer},
 address = {Berlin, Heidelberg},
 pages = {94--105},
 doi = {10.1007/BFb0076240} }

@article{AdamsAndHudson2004,
 author = {Adams, Alison M and Hudson, Richard R},
 title = {Maximum-likelihood estimation of demographic parameters using the frequency spectrum of unlinked single-nucleotide polymorphisms},
 journal = {Genetics},
 volume = {168},
 number = {3},
 pages = {1699--1712},
 year = {2004},
 doi = {10.1534/genetics.104.030171} }

@article{ArnasonEtAl2023,
 title = {Sweepstakes reproductive success via pervasive and recurrent selective sweeps},
 author = {\'{A}rnason, Einar and Koskela, Jere and Halld\'{o}rsd\'{o}ttir, Katr\'{i}n and Eldon, Bjarki},
 journal = {eLife},
 year = {2023},
 volume = {12},
 pages = {e80781},
 doi = {10.7554/eLife.80781} }

@article{BallEtAl1990,
 author = {Ball, R Martin and Neigel, Joseph E and Avise, John C},
 title = {Gene genealogies within the organismal pedigrees of random-mating populations},
 journal = {Evolution},
 volume = {44},
 number = {2},
 pages = {360--370},
 year = {1990},
 doi = {10.1111/j.1558-5646.1990.tb05205.x} }

@article{BirknerEtAl2009,
 author = {Birkner, Matthias and Blath, Jochen and M{\"o}hle, Martin and Steinr{\"u}cken and Tams, Johanna},
 title = {A modified lookdown construction for the {Xi-Fleming-Viot} process with mutation and populations with recurrent bottlenecks},
 journal = {{ALEA} Latin American Journal of Probability and Mathematical Statistics},
 volume = {6},
 pages = {35--61},
 year = {2009},
 url = {https://alea.impa.br/articles/v6/06-02.pdf} }

@article{BirknerEtAl2013a,
 author = {Birkner, Matthias and Blath, Jochen and Eldon, Bjarki},
 title = {Statistical properties of the site-frequency spectrum associated with $\Lambda$-coalescents},
 journal = {Genetics},
 volume = {195},
 number = {3},
 pages = {1037--1053},
 year = {2013},
 doi = {10.1534/genetics.113.156612} }

@article{BirknerEtAl2013b,
 author = {Birkner, Matthias and \v{C}ern\'{y}, Ji\v{r}\'{i} and Depperschmidt, Andrej and Gantert, Nina},
 title = {Directed random walk on the backbone of an oriented percolation cluster},
 journal = {Electronic Journal of Probability},
 volume = {18},
 pages = {1--35},
 year = {2013},
 doi = {10.1214/EJP.v18-2302} }

@article{BirknerEtAl2013c,
 author = {Birkner, Matthias and Blath, Jochen and Eldon, Bjarki},
 title = {An ancestral recombination graph for diploid populations with skewed offspring distribution},
 journal = {Genetics},
 volume = {193},
 number = {1},
 pages = {255--290},
 year = {2013},
 doi = {10.1534/genetics.112.144329} }

@article{BirknerEtAl2018,
 author = {Birkner, Matthias and Liu, Huili and Sturm, Anja},
 title = {Coalescent results for diploid exchangeable population models},
 volume = {23},
 journal = {Electronic Journal of Probability},
 pages = {1--44},
 year = {2018},
 doi = {10.1214/18-EJP175} }

@article{BolthausenAndSznitman2002a,
 author = {Bolthausen, Erwin and Sznitman, {Alain-Sol}},
 title = {On the static and dynamic points of view for certain random walks in random environment},
 volume = {9},
 journal = {Methods and Applications of Analysis},
 number = {3},
 pages = {345--376},
 year = {2002},
 url = {https://projecteuclid.org/journals/methods-and-applications-of-analysis/volume-9/issue-3/On-the-Satic-and-Dynamic-Points-of-View-for-Certain/maa/1119027729.full} }

@article{BravermanEtAl1995,
 author = {Braverman, John M and Hudson, Richard R and Kaplan, Norman L and Langley, Charles H and Stephan, Wolfgang},
 title = {The hitchhiking effect on the site frequency spectrum of {DNA} polymorphisms},
 year = {1995},
 journal = {Genetics},
 volume = {140},
 number = {2},
 pages = {783--796},
 doi = {10.1093/genetics/140.2.783} }

@article{Cannings1974,
 author = {Cannings, C},
 title = {The latent roots of certain {M}arkov chains arising in genetics: a new approach. {I}. {H}aploid models},
 journal = {Advances in Applied Probability},
 year = {1974},
 volume = {6},
 number = {2},
 pages = {260--290},
 doi = {10.2307/1426293} }

@article{DonnellyAndKurtz1999,
 author = {Donnelly, Peter and Kurtz, Thomas G},
 title = {Particle representations for measure-valued population models},
 volume = {27},
 journal = {The Annals of Probability},
 number = {1},
 pages = {166--205},
 year = {1999},
 doi = {10.1214/aop/1022677258} }

@article{Ewens1974,
 author = {Ewens, W J},
 title = {A note on the sampling theory for infinite alleles and infinite sites models},
 journal = {Theoretical Population Biology},
 volume = {6},
 number = {2},
 pages = {143--148},
 year = {1974},
 doi = {10.1016/0040-5809(74)90020-3} }

@book{Ewens2004,
 author = {Ewens, Warren J},
 title = {Mathematical Population Genetics, Volume I: Theoretical Foundations},
 publisher = {Springer-Verlag},
 year = {2004},
 address = {Berlin} }

@incollection{Ewens1990,
 author = {Ewens, Warren J},
 title = {Population genetics theory -- the past and the future},
 booktitle = {Mathematical and Statistical Developments of Evolutionary Theory},
 year = {1990},
 pages = {177--227},
 publisher = {Kluwer Academic Publishers},
 editor = {Lessard, Sabin},
 address = {Amsterdam} }

@article{ExcoffierEtAl2013,
 author = {Excoffier, Laurent AND Dupanloup, Isabelle AND Huerta-{S}{\'a}nchez, Emilia AND Sousa, Vitor C AND Foll, Matthieu},
 journal = {{PLOS} Genetics},
 title = {Robust demographic inference from genomic and {SNP} data},
 year = {2013},
 volume = {9},
 number = {10},
 pages = {1--17},
 doi = {10.1371/journal.pgen.1003905} }

@article{Fisher1930b,
 author = {Fisher, Ronald Aylmer},
 title = {The distribution of gene ratios for rare mutations},
 journal = {Proceedings of the Royal Society of Edinburgh},
 year = {1930},
 volume = {50},
 pages = {205--220},
 url = {https://hdl.handle.net/2440/15106} }

@book{Fisher1930,
 author = {Fisher, Ronald Aylmer},
 title = {The Genetical Theory of Natural Selection},
 publisher = {Clarendon},
 year = {1930},
 address = {Oxford} }

@article{Fu1995,
 author = {Fu, {Yun-Xin}},
 title = {Statistical properties of segregating sites},
 journal = {Theoretical Population Biology},
 volume = {48},
 number = {2},
 pages = {172--197},
 year = {1995},
 doi = {10.1006/tpbi.1995.1025} }

@article{GaoAndKeinan2016,
 author = {Gao, Feng and Keinan, Alon},
 title = {Inference of super-exponential human population growth via efficient computation of the site frequency spectrum for generalized models},
 journal = {Genetics},
 volume = {202},
 number = {1},
 pages = {235--245},
 year = {2016},
 doi = {10.1534/genetics.115.180570} }

@article{GriffithsAndTavare1998,
 author = {Griffiths, Robert C and Tavar{\'{e}}, Simon},
 title = {The age of a mutation in a general coalescent tree},
 journal = {Communications in Statistics. Stochastic Models},
 volume = {14},
 number = {1-2},
 pages = {273-295},
 year = {1998},
 doi = {10.1080/15326349808807471} }

@article{GutenkunstEtAl2009,
 author = {Gutenkunst, Ryan N AND Hernandez, Ryan D AND Williamson, Scott H AND Bustamante, Carlos D},
 journal = {{PLOS} Genetics},
 title = {Inferring the joint demographic history of multiple populations from multidimensional {SNP} frequency data},
 year = {2009},
 volume = {5},
 pages = {1--11},
 number = {10},
 doi = {10.1371/journal.pgen.1000695} }

@article{Hudson1983a,
 author = {Hudson, Richard R},
 title = {Properties of a neutral allele model with intragenic recombination},
 journal = {Theoretical Population Biology},
 volume = {23},
 number = {2},
 pages = {183--201},
 year = {1983},
 doi = {10.1016/0040-5809(83)90013-8} }

@article{Hudson1983b,
 author = {Hudson, Richard R},
 title = {Testing the constant-rate neutral allele model with protein sequence data},
 year = {1983},
 journal = {Evolution},
 volume = {37},
 number = {1},
 pages = {203--217},
 doi = {10.1111/j.1558-5646.1983.tb05528.x} }

@article{Kallenberg1973,
 author = {Kallenberg, Olav},
 title = {Characterisation and convergence of random measures and point processes},
 journal = {Probab.~Theory~Relat.~Fields},
 volume = {27},
 number = {1},
 pages = {9--21},
 year = {1973}
}

@article{Kimura1969,
 author = {Kimura, Motoo},
 title = {The number of heterozygous nucleotide sites maintained in a finite population due to the steady flux of mutations},
 journal = {Genetics},
 volume = {61},
 number = {4},
 pages = {893--903},
 year = {1969},
 doi = {10.1093/genetics/61.4.893} }

@article{Kingman1982a, 
 author = {Kingman, J F C}, 
 year = {1982}, 
 pages = {27--43}, 
 title = {On the genealogy of large populations}, 
 journal = {Journal of Applied Probability},
 volume = {19}, 
 number = {A}, 
 doi = {10.2307/3213548} }

@article{Mohle1998a, 
 author = {M{\" o}hle, M}, 
 title = {A convergence theorem for {M}arkov chains arising in population genetics and the coalescent with selfing}, 
 journal = {Advances in Applied Probability}, 
 year = {1998}, 
 volume = {30}, 
 number = {2}, 
 pages = {493--512}, 
 doi = {10.1239/aap/1035228080} }

@article{Mohle1999,
 author = {M{\"o}hle, Martin},
 title = {The concept of duality and applications to {M}arkov processes arising in neutral population genetics models},
 journal = {Bernoulli},
 year = {1999},
 volume = {5},
 pages = {761--777} }

@article{MohleAndSagitov2001,
 author = {M{\"{o}}hle, Martin and Sagitov, Serik},
 title = {A classification of coalescent processes for haploid exchangeable population models},
 journal = {The Annals of Probability},
 year = {2001},
 volume = {29},
 number = {4},
 pages = {1547--1562},
 doi = {10.1214/aop/1015345761} }

@incollection{Aldous1985,
 author = {Aldous, David J},
 title = {Exchangeability and related topics},
 booktitle = {\'{E}cole d'\'{E}t\'{e} de Probabilit\'{e}s de Saint-Flour XXIII},
 year = {1985},
 series = {Lecture Notes in Mathematics},
 volume = {1117},
 publisher = {Springer-Verlag},
 editor = {Hennequin, P.L.}
}

@article{Nielsen2000,
 author = {Nielsen, Rasmus},
 title = {Estimation of population parameters and recombination rates from single nucleotide polymorphisms},
 journal = {Genetics},
 volume = {154},
 number = {2},
 pages = {931--942},
 year = {2000},
 doi = {10.1093/genetics/154.2.931} }

@article{Pitman1999, 
 author = {Pitman, Jim}, 
 title = {Coalescents with multiple collisions}, 
 journal = {The Annals of Probability},
 year = {1999}, 
 volume = {27},
 number = {4},
 pages = {1870--1902},
 doi = {10.1214/aop/1022874819} }

@article{Pitman1995,
 author = {Pitman, Jim},
 title = {Exchangeable and partially exchangeable random partitions},
 journal = {Probab. Theory Related Fields},
 year = {1995},
 volume = {102},
 pages = {145--158}
}

@article{Ralph2019,
 author = {Ralph, Peter L},
 title = {An empirical approach to demographic inference with genomic data},
 journal = {Theoretical Population Biology},
 year = {2019},
 volume = {127},
 pages = {91--101},
 doi = {10.1016/j.tpb.2019.03.005} }

@article{Sagitov1999,
 author = {Sagitov, Serik},
 title = {The general coalescent with asynchronous mergers of ancestral lines},
 journal = {Journal of Applied Probability},
 year = {1999},
 volume = {36},
 number = {4},
 pages = {1116--1125},
 doi = {10.1239/jap/1032374759} }

@article{Sagitov2003, 
 author = {Sagitov, Serik}, 
 title = {Convergence to the coalescent with simultaneous multiple mergers}, 
 journal = {Journal of Applied Probability}, 
 year = {2003}, 
 volume = {40}, 
 number = {4}, 
 pages = {839--854}, 
 doi = {10.1239/jap/1067436085} }

@article{Schweinsberg2000,
 author = {Schweinsberg, Jason},
 title = {Coalescents with simultaneous multiple collisions},
 journal = {Electronic Journal of Probability},
 year = {2000},
 volume = {5},
 pages = {1--50},
 doi = {10.1214/EJP.v5-68} }

@article{SjodinEtAl2005,
 author = {Sj{\"o}din, P and Kaj, I and Krone, S and Lascoux, M and Nordborg, M},
 title = {On the meaning and existence of an effective population size},
 journal = {Genetics},
 volume = {169},
 number = {2},
 pages = {1061--1070},
 year = {2005},
 doi = {10.1534/genetics.104.026799} }

@article{Tajima1989,
 author = {Tajima, Fumio},
 title = {Statistical method for testing the neutral mutation hypothesis by {DNA} polymorphism},
 journal = {Genetics},
 year = {1989},
 volume = {123},
 pages = {585--595},
 doi = {10.1093/genetics/123.3.585} }

@article{SawyerAndHartl1992,
 author = {Sawyer, Stanley A and Hartl, Daniel L},
 title = {Population genetics of polymorphism and divergence},
 volume = {132},
 number = {4},
 pages = {1161--1176},
 year = {1992},
 journal = {Genetics},
 doi = {10.1093/genetics/132.4.1161} }

@mastersthesis{TyukinThesis2015,
 title={Quenched Limits of Coalescents in Fixed Pedigrees},
 author={Tyukin, Andrey},
 year={2015},
 school={Johannes-Gutenberg-Universit\"{a}t Mainz},
 address={Germany},
 url={https://www.glk.uni-mainz.de/files/2018/08/andrey_tyukin_msc.pdf} }

@article{WakeleyEtAl2012,
 author = {Wakeley, John and King, L\'{e}andra and Low, Bobbi S and Ramachandran, Sohini},
 title = {Gene genealogies within a fixed pedigree, and the robustness of {K}ingman's coalescent},
 journal = {Genetics},
 volume = {190},
 number = {4},
 pages = {1433--1445},
 year = {2012},
 doi = {10.1534/genetics.111.135574} }

@article{WollenbergAndAvise1998,
 author = {Wollenberg, Kurt and Avise, John C},
 title = {Sampling properties of genealogical pathways underlying population pedigrees},
 journal = {Evolution},
 volume = {52},
 number = {4},
 pages = {957--966},
 year = {1998},
 doi = {10.1111/j.1558-5646.1998.tb01825.x} }

@article{Wright1931,
 author = {Wright, Sewall},
 title = {Evolution in {M}endelian Populations},
 journal = {Genetics},
 volume = {16},
 number = {2},
 pages = {97--159},
 year = {1931},
 url = {https://www.genetics.org/content/16/2/97} }

@article{DiamantidisEtAl2024,
    author = {Diamantidis, Dimitrios and Fan, Wai-Tong (Louis) and Birkner, Matthias and Wakeley, John},
    title = "{Bursts of coalescence within population pedigrees whenever big families occur}",
    journal = {Genetics},
    volume = {227},
    number = {1},
    pages = {iyae030},
    year = {2024},
    month = {02},
    issn = {1943-2631},
    doi = {10.1093/genetics/iyae030},
    url = {https://doi.org/10.1093/genetics/iyae030},
    eprint = {https://academic.oup.com/genetics/article-pdf/227/1/iyae030/57433039/iyae030.pdf},
}

@Article{MohleAndSagitov2003,
 Author = {M{\"o}hle, Martin and Sagitov, Serik},
 Title = {Coalescent patterns in diploid exchangeable population models},
 FJournal = {Journal of Mathematical Biology},
 Journal = {J. Math. Biol.},
 ISSN = {0303-6812},
 Volume = {47},
 Number = {4},
 Pages = {337--352},
 Year = {2003},
 Language = {English},
 DOI = {10.1007/s00285-003-0218-6},
 zbMATH = {2102083},
 Zbl = {1054.92039}
}

@article{schweinsberg2003coalescent,
  title={Coalescent processes obtained from supercritical Galton--Watson processes},
  author={Schweinsberg, Jason},
  journal={Stochastic processes and their Applications},
  volume={106},
  number={1},
  pages={107--139},
  year={2003},
  publisher={Elsevier}
}

@article{coron2022pedigree,
  title={Pedigree in the biparental {M}oran model},
  author={Coron, Camille and Le Jan, Yves},
  journal={Journal of Mathematical Biology},
  volume={84},
  number={6},
  pages={51},
  year={2022},
  doi = {10.1007/s00285-022-01752-0}
}

@misc{linder2009,
author = {Linder, Martin},
URL={https://www.diva-portal.org/smash/get/diva2:310019/FULLTEXT01.pdf},
pages = {50},
year={2009},
title = {Common Ancestors in a Generalized {M}oran model}
}

@article{Evans1997,
 author = {Evans, Steven N.},
 title = {Coalescing {Markov} labelled partitions and a continuous sites genetics model with infinitely many types},
 fjournal = {Annales de l'Institut Henri Poincar{\'e}. Probabilit{\'e}s et Statistiques},
 journal = {Ann. Inst. Henri Poincar{\'e}, Probab. Stat.},
 issn = {0246-0203},
 volume = {33},
 number = {3},
 pages = {339--358},
 year = {1997},
 language = {English},
 doi = {10.1016/S0246-0203(97)80095-7},
 url = {https://eudml.org/doc/77572},
 zbMATH = {1057306},
 Zbl = {0884.60096}
}

@incollection{Hedgecock94,
 author = {Hedgecock, Dennis},
 title = {Does variance in reproductive success limit effective population sizes of marine organisms?},
 booktitle = {Genetics and evolution of aquatic organisms},
 year = {1994},
 pages = {1222--1344},
 publisher = {Chapman and Hall},
 editor = {Beaumont, A},
 address = {London}  }

@book{bertoin2006random,
  title={Random fragmentation and coagulation processes},
  author={Bertoin, Jean},
  volume={102},
  year={2006},
  publisher={Cambridge University Press, Cambridge}
}

@misc{Alberti+al2026,
author = {Alberti, Frederic and Birkner, Matthias and Fan, Louis Wai-Tong and Newman, Maximillian and Wakeley, John},
title ={A conditional ancestral recombination graph for diploid exchangeable population models given the pedigree},
year= {2026},
howpublished ={Manuscript in preparation}
}


\end{document}